\documentclass[12pt]{amsart}
\usepackage{amssymb,verbatim}

\usepackage{graphicx}

\usepackage{amsmath,
amsthm,
amsfonts,
amscd,
latexsym
}

\makeindex

\newtheorem{thm}{Theorem}[section]
\newtheorem{lem}[thm]{Lemma}
\newtheorem{cor}[thm]{Corollary}

\theoremstyle{definition}

\newtheorem{example}[thm]{Example}
\newtheorem{rem}[thm]{Remark}
\newtheorem{dfn}[thm]{Definition}

\newcommand{\fs}{Fatou-Shishikura inequality}
\newcommand{\fss}{Fatou-Shishikura inequality }

\newcommand{\iU}{U^\infty}

\newcommand{\R}{\mathbb{R}}
\newcommand{\RR}{\mathfrak{R}}
\newcommand{\B}{\mathcal{B}}

\newcommand{\mac}{\mathcal{AC}}

\newcommand{\an}{\mathrm{AN}}

\newcommand{\ol}{\overline}
\newcommand{\Ga}{\Gamma}
\newcommand{\G}{\Gamma}
\newcommand{\U}{\Gamma^*}

\newcommand{\0}{\emptyset}
\newcommand{\sm}{\setminus}
\newcommand{\clam}{\ol{\lam}}

\newcommand{\bd}{\mathrm{Bd}}
\newcommand{\dia}{\mathrm{diam}}
\newcommand{\sh}{\mathrm{Sh}}
\newcommand{\ch}{\mbox{$\mathrm{Ch}$}}
\newcommand{\e}{\varepsilon}
\newcommand{\al}{\alpha}
\newcommand{\ph}{\varphi}
\newcommand{\be}{\beta}
\newcommand{\ga}{\gamma}
\newcommand{\si}{\sigma}
\newcommand{\ta}{\theta}
\newcommand{\om}{\omega}
\newcommand{\da}{\delta}
\newcommand{\Da}{\Delta}

\newcommand{\nin}{\not\in}

\newcommand{\imp}{\mbox{$\mathrm{Imp}$}}

\newcommand{\pr}{\mathrm{Pr}}

\newcommand{\he}{\widehat{E}}

\newcommand{\hq}{\widehat{Q}}
\newcommand{\bw}{\widehat{W}}
\newcommand{\cu}{\mathrm{Cut}}

\newcommand{\C}{\mathbb{C}}
\newcommand{\hc}{\mbox{$\mathbb{\widehat{C}}$}}

\newcommand{\bbd}{\mbox{$\mathbb{D}$}}
\newcommand{\disk}{\mathbb{D}}
\newcommand{\di}{\mathbb{D}}
\newcommand{\diskbar}{\overline{\mathbb{D}}}

\newcommand{\ucirc}{\mathbb{S}^1}
\newcommand{\uc}{\mathbb{S}^1}

\newcommand{\val}{\mathrm{val}}
\newcommand{\eval}{\mathrm{eval}}
\newcommand{\tai}{\mathrm{Tail}}

\newcommand{\orb}{\mathrm{orb\,}}

\newcommand{\lam}{\mathcal{L}}

\newcommand{\iy}{\infty}

\newcommand{\tg}{\widetilde{G}}
\newcommand{\tq}{\widetilde{Q}}

\newcommand{\ts}{\widetilde{S}}
\newcommand{\tz}{\widetilde{Z}}

\newcommand{\lineclear}{\rule{0pt}{0pt}\nopagebreak\par\nopagebreak\noindent}

\newcommand{\hide}[1]{}

\begin{document}

\date{September 27, 2011; revised August 27, 2015}
\title[Extended Fatou-Shishikura inequality]
{An Extended Fatou-Shishikura inequality and wandering branch
continua for polynomials}

\author[Blokh]{Alexander Blokh}

\thanks{The first author was partially
supported by NSF grant DMS-0901038, HCAA of ESF, and ICTS at Jacobs University}

\author[Childers]{Doug Childers}

\author[Levin]{Genadi Levin}

\author[Oversteegen]{Lex Oversteegen}

\thanks{The fourth author was partially  supported
by NSF grant DMS-0906316}

\author[Schleicher]{Dierk Schleicher}

\thanks{The fifth author would like to thank the research networks
CODY and HCAA of the European Union and the European Science Foundations,
respectively, as well as the Deutsche Forschungsgemeinschaft.}

\address[A.~Blokh] {Department of Mathematics\\ University of Alabama at
Birmingham\\ Birmingham, AL 35294-1170}
\email[A.~Blokh]{ablokh@math.uab.edu}

\address[D.~Childers]{CAS \\University of South Florida at St. Petersburg\\
  St. Petersburg, Florida 33701}
\email[D.~Childers]{dchilders@mail.usf.edu}

\address[G.~Levin]{Institute of Mathematics\\ Hebrew University\\
Givat Ram 91904, Jerusalem, Israel}
\email[G.~Levin]{levin@math.huji.ac.il}

\address[L.~Oversteegen]{Department of Mathematics\\ University of Alabama
at Birmingham\\Birmingham, AL 35294-1170}
\email[L.~Oversteegen]{overstee@math.uab.edu}

\address[D. Schleicher]{Jacobs University, Research I, Postfach 750
561, D-28725 Bremen, Germany}
\email{dierk@jacobs-university.de}



\begin{abstract}
Let $P$ be a polynomial of degree $d$ with Julia set $J_P$. Let $\widetilde N$ be the
number of non-repelling cycles of $P$.
By the famous \fss $\widetilde N\le d-1$. The goal of the paper is to improve
this bound. The new count includes \emph{wandering collections of
non-(pre)critical branch continua}, i.e., collections of continua or points
$Q_i\subset J_P$ \emph{all} of whose images are pairwise disjoint, contain no
critical points, and contain the limit sets of $\eval(Q_i)\ge 3$ external rays.
Also, we relate individual cycles, which are either non-repelling or repelling
with no periodic rays landing, to individual critical points that are recurrent
in a weak sense.

A weak version of the inequality reads
\[
\widetilde N+N_{irr}+\chi+\sum_i (\eval(Q_i)-2) \le d-1
\]
where $N_{irr}$ counts repelling cycles with no periodic rays landing at points
in the cycle, $\{Q_i\}$ form a wandering collection $\B_\C$ of non-(pre)critical
branch continua, $\chi=1$ if $\B_\C$ is non-empty, and $\chi=0$ otherwise.
\end{abstract}

\dedicatory{Dedicated to the memory of Adrien Douady.}

\maketitle

\section{Introduction}

In the dynamics of iterated rational maps, it is a frequent
observation that many interesting
dynamical features are largely determined by the dynamics of critical points.
The classical \fss states in the polynomial case that a complex polynomial of
degree $d\ge 2$ has at most $d-1$ non-repelling periodic orbits in $\C$. We
extend this in several ways.

\begin{itemize}

\item \emph{Wandering (eventual) branch continua}, defined below, are included in the
      count (such continua are either proper subsets of periodic components
      of the Julia set or wandering components of the Julia set); note that we allow
continua to be points. In the simplest case,
such a continuum corresponds to a point $z$ in the Julia set that is the landing point of 3
or more external rays so that no point in the forward orbit of $z$ is critical
or periodic.

\item Together with non-repelling periodic orbits, we also count orbits of
      repelling periodic points that are not landing points of \emph{periodic} external
      rays (such points may exist if the Julia set is not connected and then
      must be components of the Julia set).
\item Specific critical points are associated to the aforementioned
      periodic orbits and wandering branch continua: (a) every non-repelling
      periodic orbit and every repelling periodic orbit without periodic rays
      has at least one associated critical point, so that different orbits are
      associated to different critical points, and (b) wandering branch
      continua require other critical points not associated to any
      periodic orbits.
\item The inequality is sharpened by counting not all critical points,
      but certain ``weak equivalence classes of weakly recurrent critical
      points'' (other restrictions on critical points apply as well).
\item The key idea is that various phenomena counted on the left hand side of the
inequality can be associated with critical points counted on the
right. In the case of wandering eventual branch continua the
association is not as direct as in the case of specific periodic
points, but sufficient for our purpose.
\end{itemize}

Let $P$ be a polynomial of degree $d\ge 2$ with Julia set $J_P$. A
\emph{rational ray pair} $\mathcal R$ is a pair of (pre)periodic external rays
that land at a common point, together with their common landing point; $\mathcal R$
\emph{weakly separates}\index{weakly separates} two points $z,w\in\C$ if $z$
and $w$ are in two different components of $\C\sm \mathcal R$. A critical point
$c$ is \emph{weakly recurrent}\index{weakly recurrent} if it belongs to the
filled-in Julia set, never maps to a repelling or parabolic point, and for
every finite collection $\mathcal R_1,\dots, \mathcal R_k$ of rational ray
pairs there is an $n\ge 0$ such that $c$ and $P^{\circ n}(c)$ are not weakly
separated by any ray pair $\mathcal R_i,$ $ 1\le i\le k$. Clearly, a {\it
recurrent} critical point is weakly recurrent.


If $J_P$ is not connected, then some external rays of the polynomial $P$
are non-smooth, namely those that contain
preimages of escaping critical
points or escaping critical points themselves
(see Section~\ref{disc1} for details).

In this text, a \emph{continuum}\index{continuum} is a non-empty
compact connected metric space (we allow it to be a point and call a
continuum that is not a point \emph{non-degenerate}). The
\emph{valence $\val_{J_P}(Q)$ of a continuum $Q\subset
J_P$}\index{valence} is the number of external
rays with limit sets in $Q$ (in case $J_P$ is not connected we allow
to the possibility of non-smooth external rays, see Section~\ref{disc1}. Call $Q$ a \emph{branch continuum} if
its valence is $3$ or greater. A continuum $Q\subset J_P$ is
\emph{wandering}\index{wandering continuum} if $P^k(Q)\cap P^m
(Q)=\emptyset$ for all $m>k\ge 0$.  We show that if $Q$ is wandering
then $\val_{J_P}(Q)$ is finite, and show that there exists $m$ such
that $\val_{J_P}(P^n(Q))=m$ for all sufficiently big $n$. If $m>1$
and $Q$ is contained in a (pre)periodic component $E$ of the Julia
set, then $m$ equals the number of components of $P^n(E)\sm P^n(Q)$,
see Corollary~\ref{wand-fin} and Corollary~\ref{wandiscon}. Set
$\eval_{J_P}(Q)=m$ and call it the \emph{eventual
valence}\index{eventual valence} of $Q$. We call a wandering
continuum $Q$ an \emph{eventual branch continuum}\index{eventual
branch continuum} if $\eval_{J_P}(Q)>2$. A collection of eventual
branch continua is called a \emph{wandering
collection}\index{wandering collection of eventual branch continua}
if all their forward images are pairwise disjoint.

Some of our main results are stated in Theorem~\ref{intro-assoc}. The actual
results proven in the body of the paper are significantly stronger than
Theorem~\ref{intro-assoc}, however their statements require additional notions
that will be introduced later in the paper. Observe, that if $J_P$ is
connected, the results can be stated in topological terms because in this case
by Corollary~\ref{wand-fin} the valence of a wandering
continuum $Q$ 
equals the number of components of $J_P\sm Q$ (i.e., can be defined without
invoking external rays); similarly, non-repelling cycles can be defined in a purely topological way.
Consequently, the main results also hold for polynomial-like mappings with
connected Julia set.

\begin{thm}\label{intro-assoc}
The following facts hold for the polynomial $P$.
\begin{enumerate}
\item Every non-repelling periodic orbit has an associated weakly recurrent
    critical orbit (recurrent in the case of irrationally indifferent
    orbits), so that distinct non-repelling orbits have distinct
    associated critical orbits.


\item Every repelling periodic orbit $L$ consisting of points at which
no periodic external ray lands, has an associated escaping
    critical orbit $H$ (such that $H$ is not weakly separated from $L$) so that
    distinct repelling periodic orbits have different
    associated critical orbits, see \cite{lepr}.

\item If $P$ has a wandering collection of $m\ge 1$ eventual
    branch continua $Q_1, \dots, Q_m$,
    then
\[
1+\sum_{i=1}^m(\eval_{J_P}(Q_i)-2)
\]
is bounded from above by the number of weakly recurrent critical points,
weakly separated from all non-repelling periodic points.
\end{enumerate}
\end{thm}

The relation between special dynamical features and associated critical orbits
of a polynomial $P$ with Julia set $J_P$ can be reduced to a count; this will
yield an extension of the classical \fs. More precisely, let us use the following
notation.

\begin{tabular}{ll}
$N_{FC}$\index{$N_{FC}$} & is number of different orbits of bounded Fatou domains \\ & plus the number of Cremer
cycles;\\

$N_{irr}$\index{$N_{irr}$} & is the number of repelling cycles without periodic external \\ & rays (the subscript $irr$ stands for rays with \emph{irrational} arguments); \\

$C_{wr}$\index{$C_{wr}$} & is the set of weakly recurrent critical points in periodic \\ & components
of $J_P$; \\

$C'_{wr}$\index{$C'_{wr}$} & is the set of weakly recurrent critical points in wandering \\ & components
of $J_P$; \\

$C_{esc}$\index{$C_{esc}$} & is the set of escaping critical points; \\

$m$ & is the number of eventual branch continua $Q_i$ in a wandering \\
& collection $\{Q_1, \dots, Q_m\}$ such that each $Q_i$ is contained \\ &
in a (pre)periodic component of $J_P$; \\

$m'$ & is the number of eventual branch continua $Q'_j$ \\
& in a wandering collection $\{Q'_1, \dots, Q'_m\}$
such that each $Q'_j$ is \\ & a component of $J_P$;\\

$N_{co}$\index{$N_{co}$} & is the number of cycles of components of $J_P$ that contain \\&
wandering eventual branch continua; \\

$\chi(l)$\index{$\chi(\cdot)$} & is $1$ if $l>0$ and $0$ otherwise.

\end{tabular}

\smallskip

Given a finite (perhaps empty) set of numbers $\{a_i\}^k_{i=1}$, set \linebreak
$\sum_{i=1}^k a_i = 0$ if $k=0$. Also, let $|A|$ denote the cardinality of a set $A$.

\begin{thm}[The Extended Fatou-Shishikura Inequality]\label{intro-ineq}
\lineclear
For the polynomial $P$ the following inequalities hold:

\[
N_{FC}+ N_{co} +\sum_{i=1}^m (\eval_{J_P}(Q_i)-2) \le |C_{wr}|
\]

and

\[
N_{irr}+\chi(m')+\sum_{j=1}^{m'} (\eval_{J_P}(Q'_j)-2) \le \chi(m')|C'_{wr}| + |C_{esc}|
\]

Summing up, we have

\[
N_{FC}+N_{irr}+N_{co}+\sum_{i=1}^m
(\eval_{J_P}(Q_i)-2)+\chi(m')+\sum_{j=1}^{m'} (\eval_{J_P}(Q'_j)-2) \le
\]
\[ |C_{wr}| + \chi(m')|C'_{wr}| +  |C_{esc}| \le d-1
\]

\end{thm}

We would like to make a few remarks concerning the above results.
\begin{enumerate}
\item An attracting or rationally indifferent cycle is the limit of at
    least one critical orbit as follows from Fatou~\cite{fa}; this is
    the best known case in all results. It is also well-known that
    every Cremer point and every boundary point of a Siegel disk is a
    limit point of at least one recurrent critical orbit (see
    Ma\~n\'e~\cite{mane93}). The idea to use rational ray pairs to
    associate different indifferent cycles to different critical
    points is due to Kiwi~\cite{kiwi00}. Combining this with a
    version of Ma\~n\'e~\cite{mane93} (see \cite{ts00} or \cite{bm})
    we show that different Cremer or Siegel cycles can be associated
    to different individual {\it recurrent} critical points. This
    implies that $N_{FC}\le |C_{wr}|$ which is a version of the first
    inequality of Theorem~\ref{intro-ineq} implying the classical
    Fatou-Shishikura-inequality for polynomials, i.e. $N_{FC}\le d-1$.

\item Using a recent topological result on fixed points in non-invariant
continua (Chapter 7 of \cite{bfmot10}), we show that the recurrent
critical points associated to Cremer or Siegel cycles \emph{cannot}
be associated to wandering eventual branch continua (should the
latter exist). Together with combinatorial results of
\cite{blolev02a, chi07} this yields the first inequality of of
Theorem~\ref{intro-ineq}. The tools similar to those developed in
\cite{bfmot10} are used in a recent paper \cite{ot08}, devoted to
extending isotopies of plane continua onto the entire plane.

\item If there are no wandering eventual branch continua, the inequalities reduce
    to $N_{FC}\le |C_{wr}|$ and $N_{irr}\le |C_{esc}|$; if there are wandering
    eventual branch continua, then $N_{co}+\chi(m')\ge 1$, so at least one weakly
    recurrent critical point is used for the existence of wandering eventual branch
    continua (more if, e.g., the latter are contained in different cycles of
    components of the Julia set), in addition to the individual count in the
    sum $\sum(\eval_{J_P}(Q_i)-2)$.


\item The initial version of the Fatou-Shishikura inequality is due to
    Fatou~\cite{fa} who proved that any rational map of degree $d$ has at
    most $4d-4$ non-repelling periodic cycles (he proved that any pair of
    indifferent cycles can be perturbed into one attracting and one
    repelling cycle, and every attracting cycle attracts one of the $2d-2$
    critical points).

Shishikura \cite{Mitsu} improved the Fatou bound by proving that there can
be at most $2d-2$ non-repelling cycles: using quasiconformal surgery, he showed
that \emph{every} indifferent cycle can be perturbed so as to become
attracting. His method allows to show that this bound is sharp. Rationally
indifferent periodic orbits may attract more than one critical orbit; this
refines the counts above. For rational maps, this inequality also includes
Herman rings: each periodic cycle of Herman ring counts for two non-repelling
periodic cycles.

For a conceptually different proof of the Fatou-Shishikura
inequality, see Epstein's preprint \cite{Adam}. There push-forwards
of quadratic differentials are used and, in certain cases, the
count of rationally indifferent orbits is refined
(Herman rings are not discussed in the preprint \cite{Adam}).

\item For polynomials, we have that every polynomial of degree $d\ge 2$ has at
most $d-1$ non-repelling periodic orbits in $\C$ (this is because
$\infty$ is a critical point of multiplicity $d-1$, and there are no
Herman rings). A simple proof of this inequality in the polynomial
case is due to Douady and Hubbard~\cite{dh85b}; it is based on
perturbations of polynomial-like maps. Conceptually, our approach is
close to that of Kiwi \cite{kiwi00}, yet we use some additional tools
and push the inequality further.

\item The estimates concerning wandering branch points in the locally
    connected case are obtained in \cite{kiwi02, blolev02a,chi07}. For connections
    between wandering continua and topology of the Julia set, see
    \cite{mil92} (Douady-Hubbard examples), \cite{lev98, roe08}.

\item We do not use perturbations and directly allocate to each ``piece of
    dynamics'' distinct critical points (more precisely, their grand orbits).
    It allows us to include in the count the wandering eventual branch continua
    as well. By \cite{triangles04, blover08} the count of wandering branch
    continua in degree $3$ is sharp: there exist uncountably many cubic
    polynomials with locally connected non-separating Julia sets which contain
    a wandering non-(pre)critical branchpoint $z$ of valence $3$ so that the
    inequality in Theorem~\ref{intro-ineq} becomes equality (we believe
    it is sharp in general too).

\end{enumerate}

The relation between non-repelling periodic orbits and critical
points is well-known. To briefly motivate the relation between
wandering continua and critical points, suppose that $J_P$ is
connected and locally connected. For each $y\in J_P$ let $A(y)$ be
the set of all angles $\al$ such that the external ray $R_\al$
lands at $y$. Now, consider the collection of all hyperbolic geodesics in the
boundaries of the convex hulls (in the hyperbolic metric on $\disk$) of all the sets $A(y), y\in J_P$
taken in the closed unit disk $\diskbar$. The set of all such line
segments in $\diskbar$ forms an \emph{invariant (geometric)
lamination} in the sense of Thurston \cite{thu85}.

Consider a non-(pre)periodic non-(pre)critical point $z$ that is the landing
point of at least three external rays (the number of such external rays is
finite by \cite{kiwi02, blolev02a}). Then the arguments $\mathcal A(z)$ of the external rays
landing at $z$ determine a polygon $Q_0\subset \diskbar$. The image point
$P(z)$ determines  the  polygon  say, $Q_1$, with vertices $\mathcal A(P(z))$ of external rays
landing on  $P(z)$. Note that if $\sigma_d(z)=z^d$ for $z\in \uc$, then $\si_d(\mathcal A(z))=\mathcal A(P(z))$ and $|\mathcal A(z)|=|\mathcal A(P(z))|$.

This yields a sequence of polygons $Q_0, Q_1, \dots\subset \diskbar$ with
disjoint interiors and hence Euclidean areas converging to $0$. If $Q_i$ has a
small area, then either \emph{all} its sides are short, or \emph{two} of its
sides have almost equal length and the \emph{remaining} sides are short. Under
$z^d|_{\uc}$ lengths of short sides of $Q_i$ increase. A side $s$ of $Q_i$ can
have a short image \emph{only} if the endpoints of $s$ have angles that differ
by \emph{nearly} $k/d, k=1,2,\dots,d-1$. So, the sequence $Q_i$ must have sides
that (a) converge subsequentially to a chord $\ell\in \diskbar$ such that (b)
the endpoints of $\ell$ in $\uc$ have angles that differ by \emph{exactly}
$k/d$. By (a) $\ell$ corresponds to two different external rays that land at a
common point $c$, and by (b) these rays have equal images. This implies that
$c$ is a critical point of $P$ and motivates why wandering eventual branch
continua are related to critical points.

\textbf{Acknowledgements}. In this paper we combine estimates
concerning wandering branch points in the locally connected case (see
Subsection~\ref{waga}) with the \fs\, and extend this to all
polynomials. That idea was suggested to us by Mitsuhiro Shishikura;
we acknowledge this here with gratitude. Also, we would like to express our
gratitude to the referee for carefully reading the manuscript and making
a number of thoughtful and useful remarks.

\section{Preliminaries} \label{prelim}

\subsection{Laminations and locally connected models}\label{lami}

\subsubsection{Introductory information} Let $\bbd$ be the open
unit disk and $\hc$ be the complex sphere. For a compactum $X\subset\C$, let
$\iU(X)$\index{$\iU(\cdot)$} be the unbounded component of $\C\sm X$ and let
$T(X)=\C\sm \iU(X)$ be the \emph{topological hull} of $X$. Sometimes we use
$\iU(X)$ for $\hc\sm T(X)$ (including the point at $\infty$). We say that $X$
is \emph{unshielded} if $X=\bd(\iU(X))$. If $X$ is a continuum then $T(X)$ is a
non-separating continuum and there exists a Riemann map
$\Psi_X:\hc\sm\ol\bbd\to \iU(X)$; we always normalize it so that
$\Psi_X(\infty)=\infty$ and $\Psi'_X(z)$ tends to a positive real limit as
$z\to\infty$.

Now consider a polynomial $P$ of degree $d\ge 2$ with Julia set $J_P$ and
filled-in Julia set $K_P=T(J_P)$. Clearly, $J_P$ is unshielded. Extend $z^d:
\C\to \C$ to a map $\ta_d$ on $\hc$. If $J_P$ is connected then
$\Psi_{K_P}=\Psi:\C\sm\ol\bbd\to \iU(K_P)$ is such that $\Psi\circ \ta_d=P\circ
\Psi$ \cite{hubbdoua85,miln00}.

\subsubsection{Laminations in the locally connected case}
Let us suppose for now that $J_P$ is locally connected. Then $\Psi$ extends to
a continuous function $\ol{\Psi}: {\hc\sm\bbd}\to \ol{\hc\setminus K_P}$ and
$\ol{\Psi} \circ\, \ta_d=P\circ \ol{\Psi}$; in particular, we obtain a
continuous surjection $\ol\Psi\colon\bd(\bbd)\to J_P$ (the Carath\'eodory
loop). Identify $S^1=\bd(\bbd)$ with $\ucirc = \mathbb{R}/\mathbb{Z}$.

Let $\si_d=\si=\ta_d|_{\uc}$, $\psi=\ol{\Psi}|_{\uc}$. Define an
equivalence relation $\sim_P$ on $\uc$ by $x \sim_P y$ if and only if
$\psi(x)=\psi(y)$, and call it the \emph{($d$-invariant) lamination (of
$P$)}\index{invariant lamination (of $P$)}\cite{blolev02a}. Clearly,
equivalence classes of $\sim_P$ are pairwise \emph{unlinked} (i.e.,
their Euclidian convex hulls are disjoint). The quotient space
$\uc/\sim_P=J_{\sim_P}$ is homeomorphic to $J_P$ and the map
$f_{\sim_P}:J_{\sim_P}\to J_{\sim_P}$ \emph{induced} by $\si$ is
topologically conjugate to $P|_{J_P}$. The set $J_{\sim_P}$ is a
topological (combinatorial) model of $J_P$ and is called the
\emph{topological Julia set}. The induced map $f_{\sim_P}:J_{\sim_P}\to
J_{\sim_P}$ serves as a model for $P|_{J_{P}}$ and is often called a
\emph{topological polynomial}. Moreover, one can extend the conjugacy
between $P|_{J_{P}}$ and $f_{\sim_P}:J_{\sim_P}\to J_{\sim_P}$ to a
conjugacy on the entire plane. Figure~\ref{fig:dr} shows the Julia set
called ``the Douady rabbit'' and the corresponding lamination.

\begin{figure}
\centerline{\includegraphics{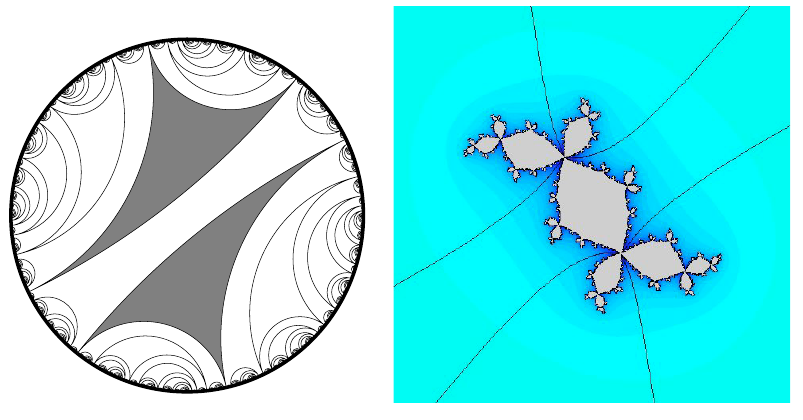}}
\caption{The Douady rabbit and its lamination.}
\label{fig:dr}
\end{figure}

\subsubsection{Laminations in the connected case}
In his fundamental paper \cite{kiwi97} Kiwi extended these ideas to the case of
a polynomial $P$ of degree $d\ge 2$ without irrationally indifferent periodic
points, \emph{not requiring} that $J_P$ be locally connected. In the case when
$J_P$ is connected, he constructed a $d$-invariant lamination $\sim_P$ on
$\ucirc$ such that $P|_{J_{P}}$ is semiconjugate to the induced map
$f_{\sim_P}:J_{\sim_P}\to J_{\sim_P}$ by a monotone map $m:J_P\to J_{\sim_P}$
(a map is \emph{monotone} if all points have connected preimages). Kiwi's
results were extended to \emph{all} polynomials with connected Julia sets in
\cite{bco08}. Equivalences $\sim$ similar to $\sim_P$ can be defined
abstractly, without any polynomials. Then they are called \emph{($d$-invariant)
laminations} and still give rise to similarly constructed \emph{topological
Julia sets $J_{\sim_P}$} and \emph{topological polynomials $f_{\sim_P}$}.

\begin{thm}\cite{bco08}\label{lcmod}
Let $P$ be a polynomial with connected Julia set $J_P$. Then there exists an
essentially unique monotone map $\ph$ of $J_P$ onto a locally connected
continuum which is finest in the sense that for \emph{any} monotone map
$\psi:J_P\to J'$ onto a locally connected continuum there exists a monotone map
$h$ with $\psi=h\circ \ph$. Moreover, there exists an invariant lamination
$\sim_P$ such that $\ph(J_P)=J_{\sim_P}$ and the map $\ph$ semiconjugates
$P|_{J_P}$ and the topological polynomial $f_{\sim_P}|_{J_{\sim_P}}$.
\end{thm}

In this construction, big pieces of $J_P$ may collapse under $\ph$. In fact,
\cite{bco08} contains a criterion for the finest map $\ph$ from
Theorem~\ref{lcmod} to not collapse all of $J_P$ to a point as well as examples
of polynomials for which $\ph(J_P)$ is a point. This shows that the notion of an
invariant lamination cannot be applied to all polynomials, even with connected
Julia set.

\subsubsection{Geometric prelaminations: Thurston's approach}
This shows the limitations of the approach based upon laminations as
equivalences on $\uc$. Therefore, in the present paper, we use Thurston's
original approach \cite{thu85} which was different. Instead of equivalences on
$\uc$, Thurston considered closed families of chords in $\diskbar$ with certain
invariance properties. More precisely, for $A\subset \ucirc\subset \C$, let
$\ch(A)$ be the hyperbolic convex hull of $A$
If $A$ is a $\sim$-class, then call a chord $ab$ on the boundary of $\ch(A)$ a
\emph{leaf}\index{leaf}; we allow for $a=b$ and then call the leaf
\emph{degenerate} (cf.\ \cite{thu85}). Using equivalence classes $A$ of an
equivalence relation $\sim$  we get in this way a collection of leaves
generated by $\sim$. Thurston's idea was to study collections of leaves
abstractly, i.e., without assuming that they are generated by an equivalence
relation with specific properties defined on the circle.

\begin{dfn}[\rm{cf.} \cite{thu85}]\label{geolam}
A \emph{geometric prelamination}\index{geometric prelamination}
$\lam$ is a set of chords in the closed unit disk $\diskbar$ such that any two
distinct chords from $\lam$ meet at most in an endpoint of both of them. Also,
$\lam$ is called a \emph{geometric lamination}
(\emph{geo-lamination})\index{geometric lamination} if $\bigcup \lam$ is
closed.
\end{dfn}

Chords in a geometric prelamination are called \emph{leaves}. If $\lam$
is a geo-lamination then $\lam^+=\bigcup \lam\cup \ucirc$ is a
continuum. A geo-lamination can be obtained if we construct a geometric
prelamination $\lam$ and then add all chords that are limits of
sequences of chords from $\lam$. Denote the new family of chords by
$\ol{\lam}$; it is easy to see that $\ol{\lam}$ is a geo-lamination.

A \emph{gap}\index{gap} of a geometric prelamination $\lam$ is the closure (in
$\C$) of a component of $\disk\sm\bigcup\lam$ that has interior points. The
boundary of a gap consists of leaves in $\clam$ and points in $\uc$. The
\emph{basis}\index{basis of a gap} of a gap or leaf $G$ is $G'=G\cap \ucirc$. A
gap is \emph{finite} if its basis is finite (i.e., if the gap is a polygon), and
\emph{infinite} otherwise. For a closed subset of $\uc$, we call its convex
hull a \emph{(degenerate) leaf or gap} even if it is not coming from any
lamination. Slightly abusing the language, we often identify a gap and its
basis, or a gap and its boundary. Note that gaps and leaves  of an \emph{invariant} lamination have additional
properties specified in Definition~\ref{geolaminv}.

\subsubsection{Geometric prelaminations and dynamics}\label{geoprel}

We extend $\si$ to $\si^*:\ol{\lam}^{^+}\to\ol{\disk}$ by mapping each leaf
$\ell=ab\in\clam$ linearly onto the chord $\si(a)\si(b)$. For a
(degenerate) leaf $\ell$, we define $\si(\ell)$ as $\ch(\si(\ell'))$.

\begin{dfn}[\rm{cf.} \cite{thu85}]\label{geolaminv}
A geometric prelamination $\lam$ of degree $d$ is said to be
\emph{invariant}\index{invariant geometric prelamination} if the
following conditions are satisfied:

\begin{enumerate}

\item (Leaf invariance) For each leaf $\ell\in \lam$, $\si(\ell)$ is
     a (degenerate) leaf in $\lam$ and, if $\ell$ is non-degenerate, there exist $d$ pairwise disjoint
     leaves $\ell_1,\dots,\ell_d$ in $\lam$ such that for each $i$,
     $\si(\ell_i)=\ell$.

\item (Gap invariance) For a gap $G$ of $\lam$, $\ch(\si(G'))$ is {\rm(1)} a
    (degenerate) leaf, or {\rm(2)} the boundary of a gap $H=\ch(\si(G'))$ of
    $\lam$ and $\si^*|_{\bd(G)}:\bd(G)\to \bd(H)$ is a \emph{positively
    oriented composition of a monotone map and a covering map}. We consider
    $\si(G)=\ch(\si(G'))$ as \emph{defined} only if {\rm(1)} or {\rm(2)} is
    satisfied.
\end{enumerate}

\end{dfn}

If a geometric prelamination $\lam$ satisfies conditions (1)--(2) except for
the last part of (1), it is called \emph{forward invariant}. By Thurston
\cite{thu85} if $\lam$ is invariant or forward invariant, then $\ol{\lam}$ is
an invariant or forward invariant geo-lamination. A leaf or gap $G$ is
\emph{critical} if $\si(G)$ is defined and the map $\si^*$ on $\bd(G)$
(equivalently, if $\si|_{G'}$) is not one-to-one.

\begin{dfn}\label{gencol}
Let $\mathcal C$ be a collection of pairwise disjoint leaves and gaps such
that for every
element $G\in \mathcal C$, $\si(G)\in \mathcal C$ is well-defined. Let $\lam$
be the set of all leaves in $\mathcal C$, all boundary leaves of gaps in
$\mathcal C$, and of all points in $\uc$. Then $\lam$ is a forward invariant
geometric prelamination, $\clam$ is a forward invariant geometric lamination,
and $\mathcal C$ is called a \emph{generating family of
$\lam$}\index{generating family (of $\lam$)} (or of $\ol{\lam}$).
\end{dfn}

For an element $G$ of $\lam$ or $\clam$ we can talk about its image as
either $\si(G)$ or $\si^*(G)$, and we will use these notations
interchangeably. A gap is \emph{periodic}\index{periodic gap} if some
iterate of $\si$ maps the basis of the gap \emph{into} itself. If $G\in
\mathcal C$ and $\si^n(G)\subset G$ then it follows from the definition
that $\sigma^n(G)=G$. A leaf of $\clam$ which is the limit of other
leaves of $\lam$ from one or both sides is called a (one-sided or
two-sided) \emph{limit leaf}\index{limit (isolated) leaf}. A leaf that
is not a limit leaf on either side is called \emph{isolated}. If a leaf
is not a two-sided limit leaf, then it is a boundary leaf of a gap. We
use the term \emph{gap-leaf}\index{gap-leaf} for a gap, or a two-sided
limit leaf, or a degenerate leaf that is the limit of non-degenerate
leaves which separate it from the rest of $\uc$. Call a gap-leaf $G$
\emph{all-critical}\index{all-critical gap-leaf} if $\si(G)$ is a
point. Figure~\ref{fig:allcri} shows an all-critical triangle with the
edges which are one-sided limit leaves.

\begin{figure}[ht]
\centerline{\includegraphics{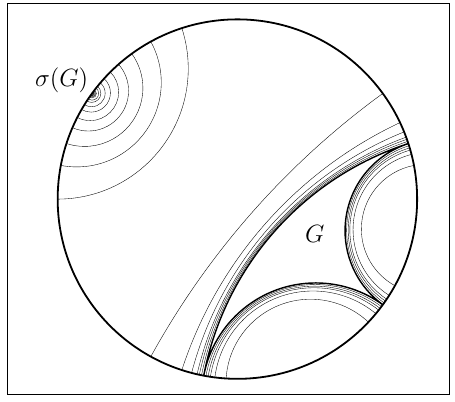}}
\caption{An all-critical triangle.}
\label{fig:allcri}
\end{figure}

\begin{lem}\label{no-crit-leaf}
Let $\mathcal C$ be a generating family of a geometric prelamination $\lam$
with no critical leaves in $\lam$. Then the following claims hold.

\begin{enumerate}

\item Let $\ell$ be a critical leaf of $\ol{\lam}$. Then $\ell$ is a
     boundary leaf of an all-critical gap-leaf $G$ of $\ol{\lam}$ all boundary
     leaves of which are limit leaves, $\si(G)$ is a point not belonging to any
     gap or non-degenerate leaf of $\clam$ and separated from the rest of $\uc$
     by a sequence of leaves of $\lam$, and so $\si^n(G)\cap G=\0$ for every
     $n>0$.

\item If $(\si^*)^n(H)\subset H$ for a leaf or gap $H$ of $\clam$, then
    $(\si^*)^n(H)=H$.

\item If $G$ is a (pre)periodic gap-leaf of $\clam$ that is not
    all-critical for $\si^n$ for any $n$, then all leaves in $\bd(G)$ are
    non-(pre)critical and (pre)periodic (in particular, this holds if $G$
    is infinite). Moreover, there are at most finitely many periodic
    leaves in $\bd(G)$.

\end{enumerate}

\end{lem}

\begin{proof}
(1) Since $\lam$ contains no critical leaves, $\ell$ is the limit leaf of a
sequence of leaves $\ell_i$ disjoint from $\ell$. Clearly, $\ell\in \ol{\lam}$
lies on the boundary of a gap-leaf $G$ of $\ol{\lam}$ and $\ell_i\cap G=\0$. If
$\si(G)$ is not a point, then $\si(\ell_i)$ either cross a leaf in the boundary
of $\si(G)$, or intersect the interior of $\si(G)$, a contradiction (this is
where the invariantness of the lamination is used). Hence $G$ is all-critical,
and $\si(G)$ is separated from the rest of the circle by a sequence of leaves
of $\lam$. Since the same argument applies to all leaves in $\bd(G)$, they are
all limit leaves. This implies the rest of the lemma (e.g., if the point
$\si^n(G)$ belongs to a gap or leaf $Q$ of $\clam$, the leaves $\si^n(\ell_i)$
will cross $Q$, a contradiction).

(2) Suppose that $\si(H)\subsetneqq H$ (the arguments for $n>1$ are
similar). If there are critical leaves (of any power of $\si$) in
$\bd(H)$ then by (1) $H$ is all-critical and $\si^n(H)\cap H=\0$ for
all $n$. Hence we may assume that \emph{$H$ is a gap without critical
leaves in its boundary}. Since $\si(H)\subsetneqq H$,
$\si(H)=\al\be=\si(\al\be)$ is an invariant leaf in $\bd(H)$ and so $H$
is finite. We may assume that $\si(\al)=\al, \si(\be)=\be$. If there
are no limit leaves in $\bd(H)$, then $H$ is an invariant gap from
$\mathcal C$, hence $\si(H)=H$. So, there are limit leaves in $\bd(H)$
and $\al\be$, which is their image, is also a limit leaf. If the leaf
$\be\ga\in \bd(H)$, adjacent to $\al\be$, is a limit leaf, then images
of leaves, approaching $\be\ga$, will cross $H$, a contradiction. Hence
$\be\ga$ is isolated in $\clam$ and so $\be\ga\in \lam$. Since there
are no critical leaves in $\bd(H)$, $\si(\be\ga)=\al\be\in \lam$. By
Definition~\ref{gencol} we conclude that there is an element of
$\mathcal C$ which contains vertices $\al, \be, \ga$ and has to
coincide with $H$. This implies that $\si^*(H)=\si(H)=H$, a
contradiction.

(3) If there is a (pre)critical leaf in $\bd(G)$ then by (1) $G$ is
all-critical for a power of $\si$, a contradiction. Consider $n$ with
$(\si^*)^n(G)$ periodic of period $m$. Clearly, it is enough to prove the rest
of lemma for $(\si^*)^n(G)=H$. By (2) and above all leaves in $\bd(H)$ stay in
$\bd(H)$ under $\si^m$ and are all (pre)periodic. We show that they belong to
the backward orbits of finitely many periodic leaves. Indeed, any leaf from
$\bd(H)$ of length less than some $\e(m)=\e>0$ increases its length under
$\si^m$. Since for geometric reasons there are finitely many leaves in $\bd(H)$
of length greater than $\e$ and no leaf ever collapses, then for any leaf
$\ell$ in $\bd(H)$ there is a moment right before the length of the leaf drops,
and by the above at this moment the image of $\ell$ is a leaf $\ell'\in \bd(H)$
of length greater than $\e$. Thus, all leaves in $\bd(H)$ pass through a finite
collection of leaves and are therefore (pre)periodic; moreover, there are at
most finitely many periodic leaves in $\bd(H)$ as desired. The claim about
infinite gaps follows from a Theorem of Kiwi \cite{kiwi02} by which all
infinite gaps are (pre)periodic.
\end{proof}

\subsection{Hedgehogs}\label{hedge}

The contents of the first two paragraphs of this subsection are due to
Perez-Marco \cite{pere94a,pere97}. Consider an irrationally indifferent
periodic point $q$ of period $1$ and let $\Da$ be $q$ (in the Cremer case) or
the maximal open Siegel disk (in the Siegel case). Suppose that $U$ is a simply
connected neighborhood of $\ol\Da$ such that $\ol U$ contains no critical
point. The \emph{hedgehog}\index{hedgehog} $H(U)$ is defined as the component
containing $\Da$ of the set of all points for which the whole orbit stays in
$\ol{U}$ \cite{pere94a,pere97}; it has the property that $H(U)\cap
\bd(U)\ne\0$. If $\Da$ is a Siegel disk with a critical point on the boundary,
then there are no hedgehogs. In the rest of this subsection $H$ denotes a
hedgehog.

It is known that $\bd(H)\subset J_P$. A hedgehog contains no periodic points
other than $q$. Hence if an invariant non-separating continuum contains an
irrationally indifferent periodic point and another periodic point, it contains
a critical point. Also, $P|_H$ is \emph{recurrent}: there is a sequence $m_n\to
\infty$ with $P^{m_n}|_H$ converging uniformly to the identity on $H$.
Moreover, the map $P|_H$ is transitive, i.e. there is a dense $G_\da$-subset of
$H$ consisting of points with dense orbits in $H$.

Two hedgehogs intersect only if they are generated by the same $\Da$; in
this case their union is another hedgehog of the same $\Da$. The \emph{mother
hedgehog}\index{hedgehog!mother} $M_q$ \cite{chi06} is the union of $\ol{\Da}$
and the closure of the union of all hedgehogs containing $\Da$. Thus, $M_q$ is
\emph{always} non-empty -- if there are no true hedgehogs, $M_q=\ol{\Da}$ (this
occurs for a Siegel disk $\Da$ containing critical points in its boundary). In
the Cremer case, $\bd(M_q)=M_q$.

If the period of $q$ is greater than $1$, everything is analogous. Thus, for
each point $y$ of $Q=\orb q$ its mother hedgehog $M_y$ is defined and invariant
under the appropriate power of $P$. The union $M_Q=\cup_{y\in Q} M_y$ is called
the \emph{mother hedgehog} of $Q$; clearly, $P(M_Q)=M_Q$.

\subsection{Continuum theory preliminaries}\label{cont}

Here we introduce a few basic notions of Carath\'eodory's prime end theory (see
\cite{miln00,Pom2}) and state a continuum theory result from \cite{bfmot10}.
Let $X$ be an unshielded
continuum. A \emph{crosscut} of $X$ (or of $T(X)$) is the image $C\subset
\iU(X)$ of $(0, 1)$ under an embedding $\psi:[0, 1]\to \C$ with $\psi(0)\ne
\psi(1)\in X$ and $\psi((0,1))\subset \iU(X)$. Let $\sh(C)$ (the \emph{shadow
of $C$}) be the bounded component of $\iU(X)\sm C$.

As above, let $\Psi_X:\hc\sm\ol\disk\to \iU(X)$ be a conformal isomorphism with
$\Psi_X(\iy)=\iy$ and such that $\Psi'_X(z)$ has a positive real limit as
$z\to\infty$. To each angle $\alpha\in\uc$ we associate the \emph{(conformal)
external ray} $R_\alpha$ as the $\Psi_X$-image of the infinite radial segment
$\{ (1,\infty)e^{2\pi i\alpha}\}$. The \emph{principal set} (or \emph{limit
set}) of the ray $R_\alpha$ is the set $\pr(\alpha):=\ol{R_\alpha}\sm
R_\alpha$. If $\pr(\alpha)=\{z\}$ is a singleton, then we say that the ray
$R_\alpha$ \emph{lands} at $z$. If $X=K_P$ is the filled-in Julia set of a
polynomial $P$ of degree $d\ge 2$ and $K_P$ is connected, then
$P(R_\alpha)=R_{\sigma(\alpha)}$. In this case, every periodic ray lands at a
periodic point of $J_P$, and every repelling or parabolic periodic point in
$J_P$ is the landing point of a positive finite number of rays, all of them
with the same period \cite{hubbdoua85,miln00}.

For any $\al\in\uc$ there exist two sequences
$\be_1<\be_2<\dots<\dots<\ga_2<\ga_1$ of angles-arguments of landing rays with $\lim
\be_i=\lim\ga_i=\al$ such that the landing points of $R_{\al_i}$ and $R_{\ga_i}$ can be
joined by a crosscut $Q_i$ with $\dia(Q_i)\to 0$ \cite[Lemma~17.9]{miln00}.
The \emph{impression of the ray $R_\al$}
(or of the angle $\al$) is defined as the set $\imp(\al)=\bigcap
\ol{\mathrm{Sh}(Q_i)}$; it does not depend on the sequences $\be_i$ and
$\ga_i$. Alternatively, the impression $\imp(\al)$ is the set of all limit
points of sequences $z_i=\Psi_X(y_i)\in U^\iy(X)$ where
$y_i \in \C\sm\ol\disk$ are points with $y_i\to\al\in\uc$.

A point $z\in \bd (X)$ is \emph{accessible} if there exists an injective curve
$l\colon[0,1]\to\C$ with $l([0,1))\subset \iU(X)$ and $l(1)=z$. For any
injective curve $l\colon[0,1)\to\iU(X)$ with $l(t)\to X$ as $t\to 1$, one can
define the principal set $\pr(l)=\ol l\sm l$ as above.

Figure~\ref{fig:primpre} illustrates the above introduced notions. The
ray $R_\al$ lands at the point with coordinates $(0, 1)$, so
$\pr(\al)=\{(0, 1)\}$ and the point $(0, 1)$ is accessible. However it
is easy to see that the impression $\imp(\al)$ of $\al$ is the segment
connecting $(0, 0)$ and $(0, 1)$.

\begin{figure}[ht]
\centerline{\includegraphics{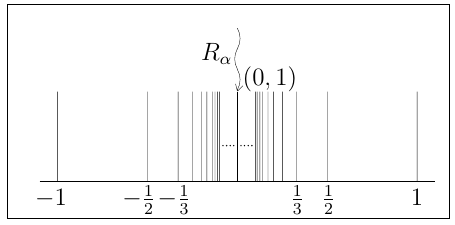}} \caption{The principal
set $\pr(\al)$ and the impression $\imp(\al)$ are not the same.}
\label{fig:primpre}
\end{figure}


\begin{thm}\cite[a short version of Theorem 4.2]{bfmot10}\label{degimp}
Let $X\subset J_P$ be a non-se\-pa\-rating invariant continuum. If all
fixed points in $X$ are repelling or parabolic and \emph{all rays landing at
them are fixed} then $X$ is a fixed point. In particular, if all periodic
points in $X$ are repelling or parabolic and the number of periodic points in
$X$, at which at least two external rays land, is finite, then $X$ is a point.
Also, if $X$ is non-degenerate then either $X$ contains a fixed CS-point, or
$X$ contains a repelling or parabolic fixed point at which non-fixed rays land.
\end{thm}

\subsection{Wandering gaps}\label{waga}
\hide{A \emph{non-(pre)critical} finite set $B\subset \uc$ is a finite set with
$\si$-images, all of the same cardinality.} Suppose that $A\subset \uc$ is a
finite set with $|A|>2$ such that (1) all sets $A, \si(A), \dots$ have pairwise
disjoint convex hulls, (2) $\si^n\colon A\to\si^n(A)$ is injective for all
$n\ge 1$, and (3) the sets $\ch(\si^n(A))$ satisfy gap invariance so that we
can define images of $\ch(A)$ under powers of $\si$ (see
Definition~\ref{geolaminv}(2)); then the set $\ch(A)$ is called a
\emph{wandering gap}\index{wandering gap} (here we talk about gaps \emph{in the
absence} of a lamination). Thus, in the definition we already assume that $A$
is \emph{non-(pre)critical}. A collection of finite gaps is
\emph{wandering}\index{wandering collection of gaps} if all \emph{all} images
of \emph{all} gaps have disjoint convex hulls. In particular, if $x$ is a
wandering non-(pre)critical branch point of a locally connected Julia set, then
the external angles of the rays that land at $x$ form a wandering gap.

By the No Wandering Triangle Theorem of Thurston \cite{thu85}, in the quadratic
case there are no wandering gaps; Thurston posed the problem of extending this
to the higher degree case and emphasized its importance. The theorem was
instrumental in the construction of a combinatorial model of the
\emph{Mandelbrot set} $\mathcal{M}$ \cite{thu85}. The next result is due to
Kiwi \cite{kiwi02}; it says that in an invariant lamination of degree $d$ a
\emph{wandering gap} consists of \emph{at most $d$ angles}. Then in
\cite{blolev02a} it was proven that for a non-empty wandering collection
$\B_\di$ of gaps $G_i$ we have $\sum_{\B_\di} (|G_i'|-2))+N'\le d-2$ where $N'$
is the number of cycles of infinite gaps in the lamination.

In \cite{blo05} the role of recurrent critical points in the dynamics of
wandering gaps was studied in the cubic case. In \cite{chi07} the results of
\cite{blo05} were generalized. We need a few definitions. Given a wandering gap
$B$, a \emph{limit leaf} of $B$ is a leaf which is a limit of a sequence of
convex hulls of images of $B$. Let $L^B_{\lim}$ be the family of such limit
leaves of $B$. Clearly, $L^B_{\lim}$ is a forward invariant geo-lamination.
Also, a chord $ab$ is called \emph{recurrent}\index{recurrent chord (leaf)} if
at least one of its endpoints is recurrent, and \emph{critical}\index{critical
chord (leaf)} if $\si(a)=\si(b)$.

\begin{thm}\cite{chi07}\label{doug}
Consider  a non-empty wandering collection of gaps $G_1, \dots, G_s$. Then the
following holds.

\begin{enumerate}

\item For each $G_i$ there exist $|G'_i|-1$ recurrent critical chords
    $t^i_j\in L^{G_i}_{\lim}, 1\le j\le |G'_i|-1$ with pairwise disjoint
    infinite orbits and the same limit set $\om_i$.

\item For each leaf $\ell\in L^{G_i}_{\lim}$ 
    we have $\ell\cap \om_i\ne\0$.

\item Let $k'$ be the maximal number of recurrent critical chords from
    $\bigcup_{i=1}^s L^{G_i}_{\lim}$ with pairwise disjoint orbits. Let $l$ be the
    number of their distinct $\om$-limit sets. Then $$\sum_{i=1}^s
    (|G'_i|-2)\le k'-l\le d-1-l\le d-2.$$

\end{enumerate}

\end{thm}

\section{The tools, or disk to plane and back again}\label{cs-wan}

In Sections~\ref{cs-wan}, ~\ref{wan-theo} and ~\ref{maint}, unless explicitly
stated otherwise, \emph{we consider a polynomial $P$ of degree $d$ with
\textbf{connected} Julia set $J_P$}. 
We use the following terminology and notation. We call irrationally indifferent
periodic points \emph{CS-points}\index{CS-points} (i.e., \emph{Cremer} points
or \emph{Siegel} points). Also, let $\RR$ be the set of all repelling or
parabolic periodic bi-accessible points and their iterated preimages. Let
$Y\subset Z$ be two continua (not necessarily subsets of any Julia set). Define
$\val'_Z(Y)$\index{valence!$\val'_Z(\cdot)$} as the number of components of $Z\sm Y$, and call $Y$ a
\emph{cut-continuum of $Z$}\index{cut-continuum} if $\val'_Z(Y)>1$ (i.e., $Z\sm
Y$ is not connected).

Section~\ref{cs-wan} prepares tools for the rest of the paper. In
Subsection~\ref{wacotop} we show that wandering cut-continua in $J_P$ contain
the principal sets of finitely many rays. This creates cuts of the plane. In
Subsection~\ref{planlam} we consider these cuts, and cuts created by rays
landing at points in $\RR$. We associate to them convex hulls of sets of
arguments of rays with principal sets in a wandering cut-continua \emph{or} in
a  point of $\RR$; the boundary leaves of these convex hulls form a geometric
prelamination. Cuts of the plane allow us to define \emph{fibers}, i.e.
intersections of \emph{closed wedges} created by cuts. This generalizes the
notion of fibers as in \cite{sch98}: in the latter reference, fibers were
defined using pairs of dynamic rays that land at common points, and
intersecting subsets of the filled Julia set that are not separated by such ray
pairs. On the other hand, the parallel construction in the disk allows us to
define subsets of the disk corresponding to such fibers. This correspondence
plays an important role in what follows.

\subsection{Wandering continua and their rays}\label{wacotop}

For a continuum $Z\subset J_P$, let $A(Z)$  \index{$A(Z)$} be the set of all angles whose rays
have principal sets in $Z$. Let $\tai(Z)$ \index{$\tai(Z)$} be the union of $Z$ and all rays with
arguments in $A(Z)$ (thus, if there are no rays with principal sets in $Z$,
then $\tai(Z)=Z$). Clearly, in the case when $A(Z)\ne \0$ the set $\tai(Z)$ \index{$\tai'(Z)$} is an unbounded connected set which is
closed if $A(Z)$ is finite. Also, by $\tai'(Z)$ we denote the union of $Z$ and
long bounded segments of rays with arguments in $A(Z)$ (to get $\tai'(Z)$, on
each ray we choose a point and remove the unbounded segment of this ray to
infinity).
Note that $|A(Z)|=\val_{J_P}(Z)$.

\begin{lem}\label{wand-fin1}
Let $X$ be an unshielded continuum and $K\subset X$ be a cut-continuum of $X$
which does not separate the plane. If $\val'_X(K)\ge n>1$, then there are $n$ distinct
external rays to $X$ with principal sets in $K$. If $\val_X(K)<\iy$, then
$\val'_X(K)=\val_X(K)=m$. If
$A(K)=\{\al_1<\al_2<\dots<\al_m<\al_{m+1}=\al_1\}$ in the sense of the cyclic order, then components $C_j$ of
$X\sm K$ can be numbered so that $C_j$ corresponds to $I_j=(\al_j, \al_{j+1})$
in the sense that for any $\be\in I_j$ we have $\imp(\be)\subset
\ol{C_j}\subset C_j\cup K$, and $\pr(\be)\cap C_j\ne \0$.
\end{lem}

\begin{proof}
First \emph{we show that if $\val'_X(K)\ge n$ then there are at least $n$
external rays with principal sets in $K$}. Collapse $K$ to a point and denote
the corresponding collapsing map $\psi$. By the Moore Theorem \cite{m62}, the
resulting topological space is still the plane on which $k=\psi(K)$ is a
cutpoint of $\psi(X)$. By a nice result of McMullen (see Theorem 6.6 of
\cite{mcm94}), if there are $n>1$ components of $\psi(X)\sm k$, then there are at
least $n$ external rays to $\psi(X)$ landing at $k$ (
if $n=1$ then
there might exist \emph{no} rays with principal sets in $K$). Their
$\psi$-preimages are curves non-homotopic outside $X$ with principal sets in
$K$. By Lindel\"of's theorem (see, e.g., \cite{Pom}) this implies that there
exist at least $n$ external rays with principal sets in $K$.

\emph{Let us now prove that if there are finitely many rays with principal sets
in $K$ then their number equals $\val'_X(K)$}. Indeed,  in this situation by
the previous paragraph $\val'_X(K)=m<\iy$, and there are \emph{at least} $m$
external rays with principal sets in $K$. Let us show that there are
\emph{exactly} $m$ such rays. Suppose otherwise. Then there must exist two
external rays $R_1$ and $R_2$ with principal sets in $K$ such that one of the
wedges formed by $R_1, R_2$ and $K$ contains no points of $X$ while the other
wedge contains $X\sm K$. This implies that \emph{all} external rays contained
in the first wedge will have their principal sets in $K$. Since there are
infinitely many of them, we get a contradiction with the assumption.

Let us introduce the notation which we need to complete the proof. Namely, let
the set of arguments of the rays with principal sets in $K$ be
$A(K)=\{\al_1<\al_2<\dots<\al_m<\al_{m+1}=\al_1\}$ and set $I_j=(\al_j,
\al_{j+1})$.

\emph{Now we show that there is a unique component $C=C_j$ of $X\sm K$ such
that for any angle $\be\in I_j$ we have $\pr(\be)\cap C\ne \0$ and
$\imp(\be)\subset \ol{C_j}\subset C_j\cup K$}. Denote by $E_j$ the open wedge
formed by the rays $R_{\al_j}, R_{\al_{j+1}}$ and the continuum $K$, such that
$E_j$ contains rays of angles from $I_j$. Then there is at least one component
of $X\sm K$ in $E_j$ (otherwise, as in the second paragraph of the proof,
infinitely many angles from $I_j$ will have principal sets in $K$, a
contradiction). Since $\val'(K)=m$, there is a \emph{unique} component $C_j$ of
$X\sm K$ in $E_j$. Since none of the angles $\be\in I_j$ can have the principal
set inside $K$, $\pr(\be)\cap C_j\ne \0$. To see that $\imp(\be)\subset
\ol{C_j}\subset C_j\cup K$, choose two sequences of angles $\ta_i<\be<\ga_i$
such that the rays $R_{\ta_i}, R_{\ga_i}$ land and connect their landing points
$x_i, y_i\in C_j$ with crosscuts $T_i$ forming a fundamental chain of
crosscuts. It follows that $\imp(\be)= \cap \ol{\mathrm{Sh}(T_i)}\cap X\subset
\ol{C_j}\subset C_j\cup K$.
\end{proof}

Observe that Theorem 6.6 of \cite{mcm94} cannot be extended to show that the
valence of a cutpoint $x$ always equals the cardinality of the number of rays
landing. E.g., a cone over a Cantor set has a vertex of uncountable valence at
which only countably many external rays land. Also, easy examples show that
$\val'_X(K)$ can be finite while $|A(K)|$ is uncountable (for example consider
an arc $I$ containing a non-degenerate subarc $K$ not containing an endpoint of
$I$).

\begin{lem}\label{inters}
Suppose that $X\subset J_P, Y\subset J_P$ are disjoint continua and there are
closed sets $Q\subset A(X), T\subset A(Y)$. Then $Q$ and $T$ are unlinked.
Thus, if $A(X)$ and $A(Y)$ are finite, then they are unlinked.
\end{lem}

\begin{proof}
Clearly, $Q\cap T=\0$. Hence if $Q, T$ are not unlinked, there must exist
angles $\al, \be\in Q$ and $\al', \be'\in T$ which are pairwise distinct
and such that the chord $\al\be$ intersects the chord $\al'\be'$. For geometric
reasons this implies that $X$ and $Y$ intersect, a contradiction.
\end{proof}

Let us now go back to dynamics. If $Z\subset J_P$ is a point of $\RR$, then, by
\cite{hubbdoua85, el89}, $|A(Z)|=\val(Z)$ is finite. We show that wandering
cut-continua are, as far as providing a tool for separating the plane and the
Julia set, analogous to points of $\RR$. So, assume that $W$ is a wandering
cut-continuum and study its dynamics.

\begin{lem}\label{nonsep}
If $P$ is a polynomial with arbitrary (perhaps, not connected) Julia set and
$W\subset J_P$ is a wandering continuum, then $W$ does not separate the plane.
\end{lem}

\begin{proof}
If $W$ is separating, the set $T(W)$ contains a Fatou domain which must be
(pre)periodic, contradicting the fact that $W$ is wandering.
\end{proof}

Let us now define the grand orbit of a wandering continuum $W$. Take a forward
image $W'$ of $W$ so that $P^n(W'), n\ge 0,$ contain no critical points. The
pullbacks (i.e. components of $P^{-m}(P^k(W'))$) of sets from the forward orbit
of $W'$ form the \emph{grand orbit $\G(W)$}\index{grand orbit of a continuum}
of $W$. The construction is necessary because of the following. Imagine that a
forward image $P^m(W)$ of $W$ contains a critical point $c$, but is smaller
than the one-step pullback of $P^{n+1}(W)$ containing $P^n(W)$ (i.e. $P^n(W)$
is not ``symmetric'' with respect to the naturally defined ``symmetry'' around
$c$). Then there is an ambiguity in defining the element of the grand orbit of
$W$ containing $P^n(W)$. Our definition allows us to avoid this ambiguity and
is consistent because it does not depend on the choice of $W'$ (as long as it
satisfies the conditions above).

\begin{lem}\label{l-nonpcr} Suppose that $W\subset J_P$ is a cut-continuum from
the grand orbit of a wandering continuum. Then the map $P^n|_{\tai(W)}$ is
not one-to-one if and only if $W$ contains a critical point of $P^n$ (in
this case there are two rays in $\tai(W)$ mapped to one ray).
\end{lem}

\begin{proof}
By Lemma~\ref{wand-fin1} the set $\tai(W)$ includes some rays and is,
therefore, non-degenerate. Suppose that $P^n|_{\tai(W)}$ is not one-to-one.
Note that  $P^{n}(\tai(W))=\tai(P^{n}(W))$. By Lemma~\ref{nonsep},
$P^{n}(\tai'(W))$ is a  non-degenerate continuum which does not separate the plane, and has no
interior in the plane.  Then by \cite{hea96} there is a critical point $c$ of
$P^n$ in $\tai(W)$. Since $J_P$ is connected this implies that in fact $c\in
W$.

Now, suppose that there is a critical point $c$ of $P^n$ in $W$. Collapse $W$
and $P^n(W)$ by a map $\psi$ of the plane to points $a$ and $b$. Consider the
induced map $g$ from a neighborhood of $a$ to a neighborhood of $b$. Since
$c\in W$ is a critical point of $P^n$, the map $g$ is $k$-to-$1$ with $k>1$.
Take a ray $R$ from $\tai(W)$, map it forward by $P^n$, and then take all rays
which are preimages (pullbacks) of $P^n(R)$. Then $\psi(P^n(R))=g(\psi(R))$ has
$k$ preimage-rays which land at $a$.  Hence there are $k$ rays with principal
sets in $W$ and the $P^n$-image of these $k$ rays is a single ray.
\end{proof}

Lemma~\ref{l-nonpcr} allows us to introduce the following notion.

\begin{dfn}\label{d-nonpcr}
A wandering continuum $K\subset J_P$ is said to be
\emph{non-(pre)critical}\index{non-(pre)critical continuum} if $\tai(K)$ has
the following property: for every $n$ the map $P^n|_{\tai(K)}$ is one-to-one.
By Lemma~\ref{l-nonpcr}, $K$ is non-(pre)critical if and only if $\tai(K)$
contains no (pre)critical points.
\end{dfn}

By Lemma~\ref{l-nonpcr}, $\eval_{J_P}(W)$ for a wandering continuum $W$ is
well-defined and equals $\val_{J_P}(P^N(W))$ where $N$ is big enough to
guarantee that $P^N(W)$ is non-(pre)critical. Also, the claim as in
Lemma~\ref{l-nonpcr} holds for disconnected Julia sets too, and so literally
the same definition as Definition~\ref{d-nonpcr} can be given in that case.
However to prove Lemma~\ref{l-nonpcr} in the disconnected case we need to study
in detail the family of external rays in that case, thus we postpone it until
Section~\ref{disc1} (see Lemma~\ref{l-nonpcr-1}).

\begin{cor}\label{wand-fin}
Let $W\subset J_P$ be a wandering cut-continuum. Then $1<m=\val_{J_P}(W)\le
2^d$, and there are exactly $m$ positively ordered angles
$A(W)=\{\al_1<\al_2<\dots<\al_m<\al_{m+1}=\al_1\}$ with principal sets in $W$.
Also, if $W$ is non-(pre)critical, then $m\le d$, $\ch(A(W))$ is wandering
non-(pre)critical, and $|\si^k(A(W))|=m$ for any $k$.

In particular, if $Q$ is a wandering cut-continuum or a point of
$\RR$, then there are finitely many rays with principal sets in $Q$ and $\val_{J_P}(Q)=\val'_{J_P}(Q)$.
\end{cor}

Recall that $\val'_{J_P}(Q)$ is the number of components of $J_P\sm Q$.

\begin{proof}
First \emph{let us show that there are at most $2^d$ external rays of $P$ with
principal sets in $W$}. Indeed, otherwise there is a set $Q$ of $2^d+1$
distinct external rays of $J_P$ whose principal sets are contained in $W$. Then
the angles of $\si^m(Q)$ will have principal sets in $P^m(W)$ for every $m\ge
0$. Since $W$ is wandering, Lemma~\ref{inters} now implies that all sets
$\si^m(Q)$ are unlinked. However, by \cite{kiwi02} this is impossible.

By Lemma~\ref{wand-fin1} the existence and the desired properties of
the set of angles $A(W)=\{\al_1<\al_2<\dots<\al_m<\al_{m+1}=\al_1\}$
follow. Suppose that $W$ is non-(pre)critical; then by definition
$\si^N|_{A(W)}$ is one-to-one for any $N$, $\ch(A(W))$ is
non-(pre)critical, and $|\si^k(A(W))|=m$ is constant. By
\cite{kiwi02} this implies that $m\le d$. Finally, the last claim of the corollary
follows from Lemma~\ref{wand-fin1}.
\end{proof}

So, wandering cut-continua in $J_P$ contain the principal sets of finitely many
rays and are in this respect analogous to repelling periodic points.

\subsection{The correspondence between the plane and the disk}\label{planlam}

In this subsection we consider cuts of the plane, generated by wandering
cut-continua and/or by rays landing at points of $\RR$.

\subsubsection{Grand orbits of wandering collections}\label{3.2.1}
We call a collection $\B_\C=\{B^1_\C, \dots, B^k_\C\}$ a \emph{wandering
collection of non-(pre)critical cut-continua}\index{wandering collection of
cut-continua, $\B_\C$} if $P^k(B^i_\C)\cap P^l(B^j_\C)=\0$ unless $k=l$ and
$i=j$. Take grand orbits $\G(B^i_\C)$, as defined right after
Lemma~\ref{nonsep}, of the sets $B^i_\C$ and then the union $\G(\B_\C)=\bigcup
\G(B^i_\C)$, called the \emph{grand orbit}\index{grand orbit of $\B_\C$} of
$\B_\C$. Observe that since the $B^i_\C$'s are non-(pre)critical, the
construction of the grand orbit of $\G(\B^i_\C)$ is simplified in this case.
Let $\U(\B_\C)$ be the union of all sets from $\G(\B_\C)$.

In the case of points of $\RR$ the construction of their grand
orbits is easier than for wandering non-(pre)critical cut-continua;
in fact, by definition the set $\RR$ is fully invariant, hence we
can write $\RR=\G(\RR)=\G^*(\RR)$. Let the collection of sets
$\G(\B_\C)\cup \RR$ be $\G(\B_\C, \RR)$ and the union of all points
of these sets be $\U(\B_\C, \RR)$. For $Q\in \G(\B_\C, \RR)$, set
$G(Q)=\ch(A(Q))$.

\subsubsection{Some important prelaminations}

By Lemma~\ref{inters}, the sets $A(Q)$ with $Q\in \G(\B_\C, \RR)$ are pairwise
unlinked, hence boundary chords of the sets $G(Q)$ with $Q\in \RR$ ($\G(\B_\C,
\RR), \G(\B_\C)$) form a geometric prelamination $\lam^\RR$ ($\lam^{\B_\C,
\RR}, \lam^{\B_\C}$). Say that the sets $G(Q)$ are \emph{elements of the
corresponding prelamination}\index{elements of a prelamination} (even though
formally leaves in the boundaries of the sets $G(Q)$, and not the sets $G(Q)$
themselves, are elements of the prelaminations). The closures of these
prelaminations are the geo-laminations $\ol{\lam^\RR}$, $\ol{\lam^{\B_\C, \RR}},
\ol{\lam^{\B_\C}}$. Observe that by construction all elements $Q$ of the grand
orbit $\G(\B_\C, \RR)$ have valences greater than $1$.

\begin{dfn}
%
If we make a statement about \emph{all} geometric prelaminations
$\lam^\RR,$ $\lam^{\B_\C},$ $\lam^{\B_\C, \RR},$ $\ol{\lam^\RR},$
$\ol{\lam^{\B_\C}},$ $\ol{\lam^{\B_\C, \RR}}$, we may jointly
denote them by $\lam$ or $\clam$. The collections $\G(\RR)=\RR,$
$\G(\B_\C),$ $\G(\B_\C, \RR)$
are sometimes jointly denoted by $\G$ while sets $\RR, \U(\B_\C),
\U(\B_\C, \RR)$
are sometimes jointly denoted by $\U$. If $\RR=\0$, we take
$\ol{\lam^\RR}$ as the empty lamination with all leaves degenerate
and a unique infinite gap coinciding with $\ol{\disk}$.
\end{dfn}

Recall that a gap-leaf is \emph{all-critical} if its $\si$-image is a
singleton.

\begin{lem}\label{no-crit-leaf-2} The following claims hold.

\begin{enumerate}

\item There are no critical leaves in $\lam$; in particular, there are no
all-critical gap-leaves in $\lam$.

\item The only critical leaves of $\clam$ must belong to
    all-critical gap-leaves with all boundary leaves being
    limit leaves.

\item Boundary leaves of any (pre)periodic gap-leaf are
    (pre)periodic.

\end{enumerate}

\end{lem}

\begin{proof}
(1) Let us prove that there are no critical leaves in $\lam$.
Suppose that $\ell\in \lam$ is a critical leaf. Then there is a set
$Q\in \G$ with $\ell=\al\be\subset \bd(G(Q))$. If $Q$ is a periodic
point then it cannot be a critical point, hence $\si$ is one-to-one
on $A(Q)$ and so a critical leaf cannot belong to the boundary of
$G(Q)$. Similarly we deal with non-critical preperiodic points.

Let now $Q$ be a wandering continuum or a (pre)\-pe\-ri\-odic
critical point. Then by Lemma~\ref{l-nonpcr}, there is a critical
point in $Q$. On the other hand, by our assumption $|A(P(Q))|\ge 2$.
Hence there is an angle $\ga\ne \si(\al)$ whose ray has a principal
set in $P(Q)$. Then by pulling back we can see that preimages of
$\ga$ separate preimages of $\si(\al)$ in $\mathbb{R}/\mathbb{Z}$. This shows that $\al$
cannot be adjacent to $\be$ in $A(Q)$, a contradiction. The rest of
the lemma is easy if the gap-leaf is finite and follows from
Lemma~\ref{no-crit-leaf} otherwise.
\end{proof}

\subsubsection{Disk to plane and back again}

In this subsection we establish a correspondence between certain subsets on the
plane and of the disk. It is generated by the above introduced sets $Q$ and
$G(Q)$, $Q\in \G$. If need be, we use the superscript $\G$ in our notation to
indicate which family generates the introduced objects, yet mostly $\G$ will be
assumed to be fixed, so if it does not cause confusion we will not use $\G$ in
the notation. First we introduce a family of \emph{planar cuts}.

\begin{dfn}[Planar cuts]\label{placut}
Let $\ell=\al\be\in \lam$ and $\al\ne \be\in G'(Q)$ be adjacent angles from
$G'(Q)$ where $Q\in \G$. Denote the set $R_\al\cup R_\be\cup Q$ by $\cu^\ell$
and call it a \emph{planar cut (centered at $Q$ and generated by
$\ell$)}\index{planar cut, $\cu^\ell$}.
\end{dfn}

Next we define \emph{planar wedges}.

\begin{dfn}[Planar wedges]\label{plawed}
Consider the set $W=\C\sm \cu^\ell$. Clearly, $W$ is an open set
with two components each of which is called a \emph{planar wedge
(centered at $Q$ and generated by $\ell$)}\index{planar wedge,
$W^\ell_\C(\cdot)$} and is denoted $W^\ell_\C$. By a \emph{closed
planar wedge (centered at $Q$ and generated by
$\ell$)}\index{closed planar wedge, $\bw^\ell_\C(\cdot)$}
$\bw^\ell_\C$ we mean the closure of $W^\ell_\C\cup Q$.
\emph{Hence, a closed planar wedge is \textbf{\emph{not}} the
closure of the corresponding (open) planar wedge.} All planar
wedges described above are said to \emph{border} on the cut
$\cu^\ell$ and to have $Q$ as their \emph{center}. If $z\in \C\sm
\cu^\ell$, the closed and open planar wedges defined by $\ell$
and containing $z$ are unique and are denoted by $W^\ell_\C(z)$ and
$\bw^\ell_\C(z)$.

Fix $Q\in \Ga$. Then for $z\in \C\sm \tai(Q)$ the component of
$\C\sm \tai(Q)$ containing $z$ is denoted by $W^Q_\C(z)$ and is
called an \emph{open planar wedge centered at $Q$, containing $z$}
(clearly, this is an open planar wedge centered at $Q$). Similarly,
$\bw^Q_\C(z)$ is the closure of $W^Q_\C(z)\cup Q$ and is called a
\emph{closed planar wedge centered at $Q$, containing $z$}. Thus,
if $z\in K_P\sm\U$ then these wedges are well-defined for any $Q\in \Ga$.
\end{dfn}

Figure~\ref{fig:plawe} shows planar wedges centered at a continuum $Q$
(we assume that $A(Q)=\{\al, \be, \ga\}$ is the set of all angles whose
rays accumulate inside $Q$); it also shows the appropriate leaf $\ell$
on the boundary of the triangle in the unit disk corresponding to $Q$
and the appropriate disk wedge.

\begin{figure}[ht]
{\includegraphics{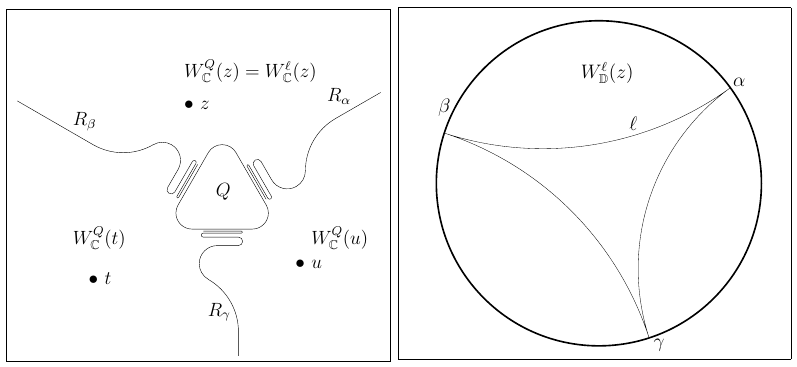}}
\caption{A continuum $Q$ and its planar and disk wedges.}
\label{fig:plawe}
\end{figure}

The definition of a \emph{disk wedge} is similar to that of a planar wedge.

\begin{dfn}[Disk wedges]\label{dwed}
Let $\ell=\al\be\in \lam$ and $\al, \be\in G'(Q)$ where $Q\in \G$. Let
$W^\ell_\di$ be a component of $\disk\sm \ell$, called a \emph{disk
wedge}\index{disk wedge, $W^\ell_\di(\cdot)$}. Also, let $\bw^\ell_\di$
be the closure of $W^\ell_\di$ called a \emph{closed disk
wedge}\index{closed disk wedge, $\bw^\ell_\di(\cdot)$}. These disk wedges are
said to be \emph{centered} at $G(Q)$ (or at $Q$), and to \emph{border} on
$\ell$.
If $z\in \disk\sm \ell$, then closed and open disk wedges defined by
$\ell$ and containing $z$ are unique and are denoted by
$W^\ell_\di(z)$ and $\bw^\ell_\di(z)$.
If $z\in \di\sm G(Q)$, then there exists a unique leaf $m$ in the
boundary of $G(Q)$ which separates $G(Q)\sm m$ from $z$. Then we
define $W^Q_\di(z)$ as $W^m_\di(z)$ and define $\bw^Q_\di(z)$ as $\bw^m_\di(z)$.
\end{dfn}

Depending on what is known about a wedge, a superscript $Q$ or a superscript
$\ell$ is used. Clearly, not only points $z$ but also sets $Y\subset \C$ can
define wedges containing $Y$ which are denoted similarly to the above. 

The correspondence between planar wedges and disk wedges is as follows: a
planar wedge $W^\ell_\C$ and a disk wedge $W^\ell_\di$ are \emph{associated (to
each other)}\index{associated wedges} if $W^\ell_\C$ contains rays with
arguments coming from the boundary \emph{circle} arc of $W^\ell_\di$. Associated planar and
disk wedges will be denoted the same way except for the subscripts $\C$ and
$\di$ respectively.
Clearly, there are countably many planar wedges and countably many
disk wedges. Let us now define \emph{disk blocks and fibers}.

\begin{dfn}[Disk blocks and fibers]\label{difa}
A non-empty intersection of fi\-ni\-te\-ly ma\-ny closed disk wedges is said to be a
\emph{disk block.}\index{disk block} A disk block is said to \emph{border} on
its \emph{boundary leaves} which are defined in a natural way.

\emph{Any} intersection $F_\di$ of closed disk wedges is called a \emph{disk
fiber}\index{disk fiber, $F_\di$} if it is minimal in the following sense: for
any set $Q\in \G$, either $G(Q)$ is disjoint from $F_\di$, or there are two
adjacent angles $\al, \be\in A(Q)$ such that the leaf $\al\be$ is contained in
$\bd(F_\di)$. For a disk fiber $F_\di$ we define its \emph{basis}
$F'_\di=F_\di\cap \uc$ whose points are said to be \emph{vertices} of $F_\di$.
\end{dfn}

Disk fibers are not necessarily disjoint, yet by Lemma \ref{basic2}
it is easy to see that if two non-degenerate disk fibers meet, than
their intersection is a leaf from $\lam$.

\begin{lem}\label{basic2}
Non-degenerate disk fibers are exactly gap-leaves of $\clam$ and
leaves of $\lam$. Also, if $G$ is a disk fiber, then
$G=\bigcap\{\bw^\ell_\di(G)\mid \ell\in\lam\}$. Moreover, the
$\si^*$-image of a disk fiber is a disk fiber.
\end{lem}

\begin{proof}
A leaf $\ell\in \lam$ is a disk fiber because it is the intersection of the two
closed wedges generated by $\ell$. Let $G$ be a gap-leaf of $\clam$ which is
not an element of $\lam$. Then, if $G$ is a leaf approximated from both sides
by leaves of $\lam$, the appropriate disk wedges generated by these leaves will
have $G$ as their intersection. Suppose that $G$ is a gap of $\clam$. For each
leaf $\ell\subset \bd(G)$ which belongs to $\lam$ choose $W^\ell_\di(G)$. For each
$\ell\subset \bd(G)$ which does not belong to $\lam$ we can choose a sequence of
leaves of $\lam$ converging to $\ell$ from outside of $G$ and then the sequence
of closed disk wedges generated by these leaves, all containing $G$. The
intersection of the just constructed family of closed disk wedges is $G$, and
clearly $G$ satisfies all the necessary properties, hence $G$ is a disk fiber.

On the other hand, let $G$ be a disk fiber which is neither a leaf
of $\lam$ nor a gap-leaf of $\clam$. Suppose that $G$ is a leaf.
Since $G$ is not a gap-leaf of $\clam$, $G$ is a boundary leaf of a gap $H$ of
$\clam$. Since $G$ is not a leaf of $\lam$, it is the limit leaf of
a sequence of leaves from outside of $H$. Again, since $G$ is not a
leaf of $\lam$, it follows that $H\subset \bw^\ell_\di(G)$ for
every $\ell\in \lam$, a contradiction with the assumption that $G$ is a fiber.
 Finally, assume that $G$ is
not a leaf. Since by definition $G$ cannot contain any leaves of
$\clam$ in its interior, $G$ must be a gap of $\clam$. The proof of
the remaining two statements of the lemma is left to the reader.
\end{proof}

Now, to define \emph{the planar fiber} of a point, we first define
\emph{planar blocks}.

\begin{dfn}[Planar blocks]\label{plabl}
A non-empty intersection of finitely many closed planar wedges is
said to be a \emph{planar block.}\index{planar block} In
particular, a planar wedge is a planar block. A planar cut whose
rays are contained in the boundary of a planar block, is called a
\emph{boundary cut (of the block)}, and the block is then said to
\emph{border} on its planar cuts.
\end{dfn}

\begin{dfn}[Planar fibers]\label{plafi}
If $G$ is a disk fiber, then by Lemma~\ref{basic2}
$G=\bigcap\{\bw^\ell_\di(G)\mid \ell\in\lam\}$. If
$\{\bw^{\ell}_\C(G)\}$ is the sequence of associated closed planar
wedges, then we say that $F_\C(G)=\bigcap \bw^{\ell}_\C(G)$ is the
\emph{planar fiber of $G$ (or associated to $G$)}\index{planar
fiber of $G$, $F_\C^\Gamma(G)$}. Observe that if $G$ is a leaf
$\ell\in\lam$, then $F_\C(\ell)=\cu^\ell$ and if $Q\in\Gamma$ and
$G=G(Q)$, then $F_\C(G)=\tai(Q)$.

Given a point $z\in\C\sm \bigcup_{E\in \Gamma} \tai(E)$ and
$Q\in\Gamma$, there exists a unique planar wedge $\bw^Q_\C(z)$
which contains $z$. For such $z$ we denote by $F_\C(z)$ the
\emph{planar fiber of $z$}\index{planar fiber of a point $z$,
$F_\C^\Gamma(z)$}, the intersection of \emph{all} the wedges
$\bw^Q_\C(z)$. Moreover, for every planar wedge $\bw^Q_\C(z)$ let
$\widehat{\bw}^Q_\di(z)$ be the associated disk wedge. Then it is easy to see
that $F_\di(z)=\bigcap \widehat{\bw}^Q_\di(z)$ is a disk fiber and we call it
the \emph{disk fiber of $z$}\index{disk fiber of $z$,
$F^\Gamma_\di(z)$}. We will also say that the fibers $F_\C(z)$ and
$F_\di(z)$ are \emph{associated to each other}.
\end{dfn}

Figure~\ref{fig:pladifi} shows a planar fiber $F_\C$ and its associated
disk fiber $F_\di(z)$ together with some sets $\tai(Q), Q\in \Gamma$ and
corresponding sets $G(Q)$.

\begin{figure}[ht]
{\includegraphics{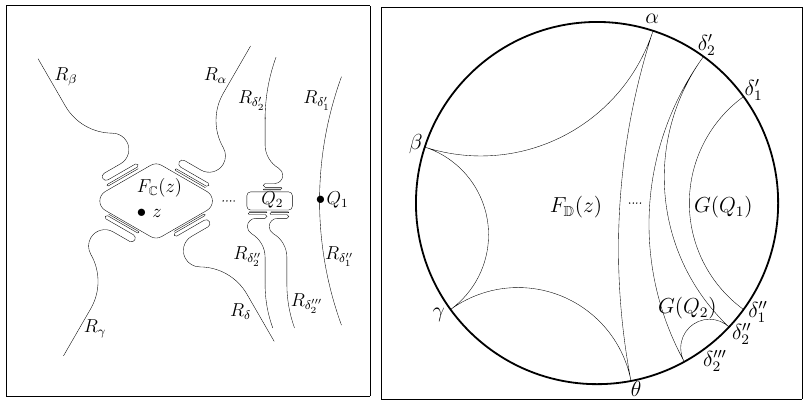}}
\caption{A planar fiber and its associated disk fiber.}
\label{fig:pladifi}
\end{figure}

A planar fiber can be represented as a countable intersection of a
nested sequence of planar blocks. Clearly, $z\in F_\C(z)$. Also, by
definition the fiber $F_\C^\RR(z)\cap K_P$ consists exactly of all
points of $K_P$ which are weakly non-separated from $z$. The
relation between other types of fibers may be more complicated.

\begin{lem}\label{fibincl}
For a point $z\in \C\sm\bigcup_{E\in \Gamma} \tai(E)$ let $F_\C(z)$
be the planar fiber of $z$, and let $G=F_\di(z)$ be the associated
disk fiber of $z$. Then $F_\C(G)\subset F_\C(z)$.
\end{lem}

\begin{proof}
Consider a planar wedge $\bw^Q_\C(z)$ and its associate disk wedge
$\widehat{\bw}^Q_\di(z)$. Then there exists $\ell\in \lam$ such that
$\widehat{\bw}^Q_\di(z)=\bw^\ell_\di(G)$. If now 
$F_\C(G)$ is the
associated planar fiber, then $\bw^\ell_\C(G)=\bw^Q_\C(z)$. Hence
$F_\C(G)\subset F_\C(z)$ as desired.
\end{proof}

Lemma~\ref{basic1} is a simple corollary of the definitions.

\begin{lem}\label{basic1}
A planar fiber $F_\C$ is the union of the non-se\-pa\-ra\-ting
in the plane continuum $F_\C\cap K_P$ and rays with angles in the associated
disk fiber. Let $G$ be a disk fiber such that there exists a point
$z\in F_\C(G)\sm \bigcup_{E\in \Gamma} \tai(E)$. Then $G\ne G(Q)$ for
any $Q\in\G$, and $G$ is not a leaf of $\lam$. Moreover,
$F_\C(G)=F_\C(z)$ and $G=F_\di(z)=\ch\{\al\in\uc\mid R_\al\subset
F_\C(z)\}$.
\end{lem}

\begin{proof}
Note that $F_\C\cap K_P$ is the intersection of
planar continua which do not separate the plane (which are the intersections of the appropriate
closed planar wedges and $K_P$). Hence $F_\C\cap K_P$ is a
 continuum which does not separate the plane. By definition, rays of angles from the
associated disk fiber are contained in $F_\C$ while all other says
are disjoint from $F_\C$. This proves the first part of the lemma.

To prove the rest of the lemma, observe first that it easily
follows if $G$ is degenerate. Now, let $G$ be a non-degenerate disk
fiber such that there exists a point $z\in F_\C\sm \bigcup_{E\in
\Gamma} \tai(E)$. By definition this implies that $G\ne G(Q)$ for
any $Q\in\G$, and $G$ is not a leaf of $\lam$. Since $G$ is a disk
fiber, it now follows from Lemma~\ref{basic2} that $G$ is either a
double sided limit leaf in $\clam\sm\lam$ or a gap of $\clam$ such
that $G\ne G(Q)$ for all $Q\in\G$. The required equality
$F_\C(G)=F_\C(z)$ follows since the two families of closed planar
wedges whose intersections are, respectively, $F_\C(z)$ and
$F_\C(G)$, are identical. The last claim of the lemma is left to
the reader.
\end{proof}

\subsubsection{Dynamics and correspondence between sets}

Notice, that by a Theorem of Kiwi \cite{kiwi02} all infinite gaps
of $\clam$ are (pre)periodic. Mark a point in each periodic parabolic
Fatou domain, and let $\an$ \index{$\an$} (``attracting and
neutral'' points) be the set of all attracting, Siegel, Cremer, or
marked points; given $p\in \an$, let $c(p)$ be the period of $p$ or
(for a marked point) the period of its parabolic domain. The next
lemma is an application of the tools developed so far.
Recall that the linear extension $\si^*$ was defined in the
beginning of Subsection~\ref{geoprel}. Note that if $p\in \an$,
then $p\in K_P\sm \U$ and both $F_\C(p)$ and $F_\di(p)$ are
defined. We now show that the correspondence between disk fibers
and planar fibers is dynamical.

We will need the following definition. Let $X$ be a connected
topological space. Then $X$ is unicoherent provided that for any closed
connected subsets $A$ and $B$ of $X$, if $X=A\cup B$, then  $A\cap B$
is connected. Thus, an interval is unicoherent while the circle is not.

\begin{lem}\label{dynam}
Let $F_\di(z)$ and $F_\C(z)$ be the disk fiber and the planar fiber of a point
$z\in\C\sm\Gamma^*$. Then $P(F_\C(z))=F_\C(P(z))$ and
$\si^*(F_\di(z))=F_\di(P(z))$ are the planar fiber and the disk fiber of the
point $P(z)$. Moreover, if $G$ is a disk fiber and $H=\ch(\si(G'))$, then $H$
is a disk fiber and $P(F_\C(G))=F_\C(H)$.
\end{lem}

\begin{proof}
Suppose that $F_\C(z)$ is the planar fiber of a point $z\in\C\sm\Gamma^*$.
Clearly $w=P(z)\in \C\sm\bigcup_{E\in \Gamma} \tai(E)$ and the fiber $F_\C(w)$
is well-defined.

\emph{We will show first that $P(F_\C(z))\subset F_\C(w)$.} Suppose
that $x\in F_\C(z)$ and $P(x)\not\in F_\C(w)$. Then there exists
$Q\in\Gamma$ such that $w$ and $P(x)$ are in distinct components of
$\C\sm \tai(Q)$. If $C= P^{-1}(\tai(Q))$ does not separate $z$ and
$x$, there exists an arc $A\subset\C\sm C$ joining $x$ and $z$. But
then $P(A)$ is a continuum in $\C\sm \tai(Q)$ joining $w$ and
$P(x)$, a contradiction. Hence $C$ separates $x$ and $z$ and, since
$\C$ is unicoherent and locally connected, a component $C'$ of $C$
must separate $x$ and $z$. Since $C'=\tai(Q')$ for some component
$Q'$ of $P^{-1}(Q)$, we get a contradiction with the fact that
$x\in F_\C(z)$. Hence we have shown that $P(F_\C(z))\subset
F_\C(w)$.

\emph{We show next that $P(F_\C(z)))=F_\C(w)$.} Suppose  that $v\in
F_\C(w)$ and $P^{-1}(v)\cap F_\C(z)=\0$. Since
$P^{-1}(v)=\{u_1,\dots,u_d\}$ is finite, there exists for each $j$
a set $Q_j\in\Gamma$ such that $u_j$ and $z$ are in distinct
components of $\C\sm \tai(Q_j)$. Since $v\in F_\C(w)$, there exists
an arc $A\subset \C\sm \bigcup \tai(P(Q_j))$ joining $w$ and $v$.
Since $P$ is an open map, the component $A'$ of $P^{-1}(A)$
containing the point $z$ contains some point $u_j$. Since $A'\cap
\tai(Q_j)=\0$, $u_j$ and $z$ are in the same component of $\C\sm
\tai(Q_j)$, a contradiction $P^{-1}(v)\cap F_\C(z)=\0$. Hence
$P(F_\C(z))=F_\C(w)$ as desired.

\emph{We show next that $\si^*$ maps disk fibers to disk fibers.}
Let $G$ be a disk fiber. If $G$ is degenerate, the lemma easily
follows. So we can assume that $G$ is a non-degenerate disk fiber.
By Lemma~\ref{basic2}, $G$ is a gap-leaf of $\clam$ or a leaf of
$\lam$. If $\si^*(G)$ is a leaf of $\lam$ or a gap-leaf of $\clam$,
we are done. Otherwise $\si^*(G)$ is either a leaf of $\clam\sm
\lam$ or a point. Clearly, this implies that $G$ is a finite gap or
a single leaf. If $\bd(G)$ contains a leaf of $\lam$, then
$\si^*(G)$ will be a leaf of $\lam$, a contradiction. Hence $G$ is
a finite gap all of whose boundary leaves are one-sided limit
leaves. Therefore $\si^*(G)$ is a gap-leaf, a contradiction. Note
that by the above, $P(F_\C(z))=F_\C(w)$. It now follows easily
that $\si^*(F_\di(z))=F_\di(w)$.

Let $G$ be a disk fiber. By the above, $\si^*(G)=H$ is also a disk
fiber and $H=\ch(\si(G'))$. Let $F_\C(G)$ and $F_\C(H)$ be the
associated planar fibers. \emph{We will show that
$P(F_\C(G))=F_\C(H)$}. If $G=G(Q)$ for some $Q\in \Gamma$ or
$G=\ell$ for some $\ell\in\lam$, then it follows easily that
$P(F_\C(G))=F_\C(H)$ and we are done. Hence we may assume that
there exists a point $z\in F_\C(G)\sm\bigcup_{E\in \Gamma}
\tai(E)$. Now, by Lemma~\ref{basic1} $F_\C(z)=F_\C(G)$. By the
first part of this lemma, $P(F_\C(z))=F_\C(P(z))$.

\emph{Let us show that $F_\C(H)=F_\C(P(z))$.} To this end, let us show
that the arguments of external rays contained in both sets, are the
same. Indeed, by the first, already proven, claim of this lemma,
$P(F_\C(z))=F_\C(P(z))$. Hence the arguments of the rays inside the
set $F_\C(P(z))$ form a set $\si(G)'=H'$ and $F_\di(P(z))=H$. Since
$G\ne G(Q)$ for any $Q\in\G$ and $G$ is not a leaf of $\lam$, the
same holds for $H$. Hence by Lemma~\ref{basic1},
$F_\C(P(z))=F_\C(H)$ and $P(F_\C(G))=F_\C(H)$ as desired.
\end{proof}

\begin{lem}\label{finmany}
If $F$ is a planar fiber and $G$ is its associated disk fiber, then the
following holds.

\begin{enumerate}

\item Let $\al\be$ be a leaf in $\bd(G)$ such that the circular
    arc $(\al,\be)$ is disjoint from $G'$, and $\ga\in (\al,
    \be)$. If $\al\be\nin \lam$ then $\imp(\ga)\cap F=\0$. On
    the other hand, if $\al\be\in \bd(G(Q))$ for some $Q\in
    \Ga$ then $\imp(\ga)\cap F\subset Q$. Moreover, there are
    at most finitely many angles $\ga\in (\al, \be)$ with
    $\pr(\ga)\subset F$.

\item If $G$ is finite, there are finitely many angles
    $\ga$ with $\pr(\ga)\subset F$.

\item There are at most finitely many repelling or parabolic periodic
    points in $F$ at which two or more external rays land.

\item If $x\in F$ is a preimage of a repelling or a parabolic point, then
    there exists $\al\in G'$ with $R_\al\subset F$ landing at $x$.

\end{enumerate}

\end{lem}

\begin{proof}
If $G=G(Q)$ for some $Q\in\G$, or a leaf of $\lam$, or a two-sided
limit leaf, the lemma follows easily. Thus, by Lemma~\ref{basic2}
we may assume that $G$ is a gap of $\clam$ that is not an element
of $\lam$.

(1) If $\al\be$ is a limit of leaves in $\lam$ with endpoints in
$(\al,\be)$ then for all $\ga\in(\al,\be)$, $\imp(\ga)\cap F=\0$.
Otherwise $\al\be$ is a boundary leaf of an element $H$ of $\lam$
corresponding to $Q\in\G$. Clearly, then $\imp(\ga)\cap F\subset Q$
as desired. This proves the first part of (1). Now, by the above if
$\al\be$ is a limit of leaves in $\lam$ with endpoints in
$(\al,\be)$ then there are no angles $\ga$ with $\pr(\ga)\subset
F$. If $\al\be$ is a boundary leaf of an element $H$ of $\lam$
corresponding to $Q\in\G$ then, again by the above, the fact that
$\pr(\ga)\subset F$ would imply that $\pr(\ga)\subset Q$. Since
there are finitely many angles $\ga'$ with $\pr(\ga')\subset Q$,
then there are at most finitely many angles $\ga\in(\al,\be)$ with
$\pr(\ga)\subset F$.

(2) Follows from (1). 

(3), (4) \emph{Let us show that if $x\in F$ is a repelling or
parabolic point or a preimage of it, then there is at least one (if
$x\nin \RR$) and at least two (if $x\in \RR$) rays landing at $x$
and contained in $F$.} Indeed, if $x\nin
\RR$ let $C=R\cup \{x\}$ where $R$ is a ray landing at $x$. 
If $x\in \RR$ let $\bw_x(F)$ be the closed wedge at $x$ containing
$F$, and let $C$ be the union of two boundary rays of $\bw_x(F)$
and $\{x\}$. Then in either of these cases by definition $C\subset
F$.

Now, the previous paragraph immediately implies (4). To prove (3),
observe that by the previous paragraph each point of $\RR$ in $F$
corresponds to a boundary leaf of $G$. However by
Lemma~\ref{no-crit-leaf} there are at most finitely many periodic
leaves in $\bd(G)$. Hence there are finitely many points of $\RR$
in $F$ which implies (3).
\end{proof}

\begin{dfn}
We will call a attracting or parabolic Fatou domain an \emph{parattracting
domain}\index{parattracting domain}.
\end{dfn}

Lemma~\ref{sep-gap} relates periodic planar fibers and disk fibers.

\begin{lem}\label{sep-gap}
Let $G$ be a disk fiber which maps into itself by $(\si^*)^n$ and
let $F=F_\C(G)$ be the associated planar fiber. Then in fact
$(\si^*)^n(G)=G$ and the following claims hold.

\begin{enumerate}
\item $P^n(F)=F$.

\item If $G'$ is finite then $F$ is a periodic point. If in addition
    $|G'|>1$ then there exists $x\in \RR$ such that $G=G(x)$.

\item If $G'$ is infinite, then there exists $p\in \an$ such
    that $F_\di(p)=G$. Conversely, for each $p\in \an$,
    $F_\di(p)$ is an infinite periodic gap.  
    If $\RR\subset \G$, this correspondence
     between $\an$ and all infinite periodic gaps of $\clam$ is one-to-one,
     and $p$ is a unique point of $\an$ in $F_\C(p)$.

\end{enumerate}

\end{lem}


\begin{proof}
Assume that $n=1$. By Lemma~\ref{basic2} $G$ is either a gap-leaf, or a leaf of
$\lam$; by Lemma~\ref{no-crit-leaf-2}, there are no critical leaves in $\lam$.
Thus, if $G$ is a gap-leaf, then Lemma~\ref{no-crit-leaf}(2) implies
$\si^*(G)=G$ and $\si(G')=G'$, and if $G$ is a leaf of $\lam$ then $\si^*(G)=G$
and $\si(G')=G'$ too.

(1) Follows immediately from Lemma~\ref{dynam}.

(2) We consider only the case when $G$ is non-degenerate and $|G'|>1$; if $G$
is a degenerate gap-leaf (i.e., a point in $\uc$ which is separated from the
rest of $\uc$ by a sequence of leaves converging to it), the arguments are
almost literally the same and are left to the reader.

\emph{We first prove that $F$ cannot contain a parattracting Fatou domain}.
Indeed, otherwise by \cite{przdu94} there are infinitely many repelling
periodic points in $\bd(U)$ which contradicts Lemma~\ref{finmany}.

\emph{Let us show that a CS-point $p$ cannot belong to $F$}.
Indeed, since $G$ is periodic under $\si^*$, all the angles in $G'$
are periodic. Hence, since $F$ is closed, their principal sets
(i.e., in this case landing points) are contained in $F$. Now,
suppose that a CS-point $p$ belongs to $F$. Then by
\cite{pere94a,pere97} (see Subsection~\ref{hedge}) there is a
critical point in $F$ (if not, $F$ is a subset of a
\emph{hedgehog} and cannot contain periodic landing points of
angles of $G'$). Consider two cases.

(i) Suppose that there exists a critical point $c\in F\cap J_P$.
Choose $\al\in \uc$ so that $c\in \imp(\al)$ and $\al$ is not
periodic (this is possible because, due to the symmetry of the map
$P$ around $c$, the set of all angles whose impressions contain $c$
must contain pairs of angles mapping to the same angle). Then
$\al\nin G'$ (because all angles in $G'$ are periodic) and there
exists a boundary leaf $\ell=\be\ga$ of $G$ with $\al\in (\be,
\ga)$ and $G'\cap (\be, \ga)=\0$.

By Lemma~\ref{finmany} the fact that $c\in \imp(\al)\cap F$ implies
that $\be\ga$ is a boundary leaf of some element $H$ of $\lam$
corresponding to $Q\in \G$, and $c\in Q$. Since $\be\ga$ is periodic,
this implies that $Q$ is a point of $\RR$, a contradiction with $c\in Q$.

(ii) Suppose that $F\cap J_P$ contains no critical points. Let $E$
be the component of $P^{-1}(P(F))$ containing $F$. \emph{We claim that in
this case $E=F$.} Indeed, suppose that $F$ is a proper subset of
$E$. Then there exists a sequence $z_i\in E\sm F$ converging to a
point $z\in F$. We may assume that one of the following two
possibilities holds.

(a) There exists a leaf $\ell=\be\ga$ of $\lam$ such that the cut
$\cu^\ell$ separates points $z_i$ from $F\sm \cu^\ell$. Then $z$
is a periodic point from $\RR$. Therefore $P$ is one-to-one in a
small neighborhood of $z$. Choose points $x_i\in F$ such that
$P(x_i)=P(z_i)$ for all $i$'s. Then all these points must be
positively distant from $z$. Assuming that $x_i\to x\in F$ we see
that $x\neq z$. By continuity $P(x)=P(z)$ and so $x$ is a
preperiodic point from $\RR$. However, by Lemma~\ref{finmany}(4)
there exists a periodic angle $\ta\in G'$ whose ray lands at $x$, a
contradiction.

(b) There exists a sequence of sets $Q_i\in \G$ and boundary leaves
$\ell\in \bd(G(Q_i))$ such that $z_1$ is separated from $z$ by cuts
$\cu^{\ell_i}$. Then it follows that each $Q_i$ intersects $E$ and
hence $P(Q_i)$ intersects $P(F)$ for every $i$. However, by
Lemma~\ref{dynam} $P(F)$ coincides with the planar fiber
$F_\C(\si(G))$ associated with the \emph{finite} disk fiber
$\si(G)$, a contradiction.

So, we have proved that $E=F$. On the other hand, by the above
there is a critical point in $F$. Hence $P|_F$ is a non-trivial
branched covering map onto $P(F)$. Choose an angle $\al\in G'$ and
let $y$ be the landing point of $R_\al$. Then there exists a point
$y'\ne y$ in $F$ such that $P(y)=P(y')$. Hence again by
Lemma~\ref{finmany}(4) we have a contradiction. This shows that
there are no CS-points in $F$.

By the above, $F$ contains no parattracting Fatou domains and no
CS-points. By Lemma~\ref{basic1}, $F\cap K_P$ is a
continuum which does not separate the plane. By (1), $P(F\cap K_P)=F\cap K_P$. By
Lemma~\ref{finmany}(3) and Theorem~\ref{degimp}, $F\cap K_P$ is a
periodic point $x$. If $|G'|>1$ this implies that $x\in \RR$.

(3) \emph{We claim that if $G'$ is infinite and periodic, then $\an\cap F\ne
\0$.} Indeed, otherwise $G'$ is infinite and periodic with neither a CS-point
nor a Fatou domain in $F$. As above, by (1), Theorem~\ref{degimp} and
Lemma~\ref{finmany}(3), $F\cap K_P$ is a point $x\in \RR$ which is impossible.
Notice that by Lemma~\ref{basic1} $G=F_\di(p)$ as desired.

Now, \emph{let $p\in \an$ and $P(p)=p$ (otherwise the proof is
similar), and prove the second statement of (3)}. By
Lemma~\ref{dynam}, $P(F_\C(p))=F_\C(p)$ (because $n=1$ by the
assumption). By Lemma~\ref{dynam}, $\si^*(F_\di(p))=F_\di(p)$. By
(2), $F'_\di(p)$ is infinite.

Finally, let $\RR\subset \G$. By \cite{gm93, kiwi02} sets $\tai(x),
x\in \RR$ separate CS-points, attracting points and marked points
from each other. Hence the just defined association between
infinite periodic gaps of $\clam$ and points of $\an$ is
one-to-one.
\end{proof}

Observe that even if $p$ is an attracting or marked point, the corresponding set
$F_\C(p)$ is not necessarily the closure of the corresponding Fatou domain. Indeed, suppose
that $p$ is a fixed attracting point, $E$ is its Fatou domain, and all periodic
points on its boundary are not cutpoints of the Julia set $J_P$. Suppose that
there is a non-(pre)periodic critical point on its boundary. Then there exists
a pullback $E'$ of $E$, attached to $E$ at $c$. As follows from the definition,
$E'$ must be contained in $F_\C(p)$ too. Moreover, appropriate pullbacks of
$E'$ will also have to be contained in $F_\C(p)$ because they will not be
separated from $p$ by a cut generated by a point of $R$. Thus, in this case
$F_\C(p)$ includes not only $E$ but the entire family of attached to it
pullbacks of $E$.

\section{Non-repelling cycles and wandering continua}\label{wan-theo}

If $p$ is a CS-point, the orbit of the set $F_\C^\RR(p)$ is called a
\emph{CS-set} and is denoted by $F_\C^\RR(\orb p)$. In this section we
use the tools developed in Section~\ref{cs-wan} in order to study CS-sets in
connection with wandering non-(pre)critical cut-continua as well as recurrent
critical points. This is necessary for our study because it is through CS-sets
that both phenomena which we are interested in - wandering non-(pre)critical
cut-continua and recurrent critical points - are related.

\subsection{Limit behavior of orbits of wandering non-(pre)critical cut-continua}

In this subsection we show that wandering cut-continua cannot live in CS-sets
(Theorem~\ref{no-wand}). This is used in Corollary~\ref{wrinr} which relates
geometric prelaminations $\lam^\RR, \clam^\RR$ and $\lam^{\B_\C},
\clam^{\B_\C}$.

\begin{thm}\label{no-wand} Let $p\in \an$ and let $Q\subset J_P$ be a wandering
non-(pre)\-cri\-tical cut-continuum. Then the CS-set $F_\C^\RR(\orb p)$ is disjoint
from $Q$.
\end{thm}

\begin{proof}
Consider $\B_\C=\{Q\}$ as a wandering ``collection'' of cut-continua. Assume
that $p\in\an$ and $P(p)=p$.  Set $F=F_\C^\RR(p)$ and $G=F_\di^\RR(p)$. Then by
Lemma~\ref{dynam} $P(F)=F$ and, by Lemma~\ref{sep-gap}, $\si^*(G)=G$. By
Lemma~\ref{wand-fin} $|A(Q)|=\val(Q)>1$ is finite. By way of contradiction
assume that $F$ is not disjoint from $Q$. Consider an element $\hq$ of the
grand orbit $\G(Q)$ of $Q$ and prove a few claims.

\emph{First we show that if $F\cap \hq\ne \0$ then $\hq\subset F$}. Indeed, let
$\hq\not\subset F$ and $z\in \hq\sm F$. Then by definition there is a point
$y\in \RR$ such that the planar wedges $W^y_\C(p)$ and $W^y_\C(z)$ are distinct
and therefore disjoint. Since $y$ is (pre)periodic, it cannot belong to $\hq$
(which is wandering), hence $\hq\subset W^y_\C(z)$ which implies that $\hq\cap
F=\0$, a contradiction.

\emph{Next we prove that $A(\hq)\subset G'$ and no point of $A(\hq)$ is an
endpoint of a circle arc complementary to $G'$}. By Lemma~\ref{finmany}(1) if
there exists $\al\in A(\hq)\sm G'$ then $\pr(\al)\cap F$ is a (pre)periodic
point which contradicts $\pr(\al)\subset \hq$ and $\hq$ being wandering. Also,
if $\al\in A(\hq)$ is an endpoint of a boundary leave of $G$ then by
Lemma~\ref{sep-gap}(3) and Lemma~\ref{no-crit-leaf}(3) $\al$ is (pre)periodic,
contradicting that $\hq$ wanders. This proves the claim.

\emph{Now we prove that $\hq\subset F$ cuts $F$ into at least $\val(\hq)$
components}. Indeed, by the above, no point of $A(\hq)$ is an endpoint of an
arc complementary to $G'$. Hence, for adjacent angles $\al, \be\in A(\hq)$ (so
that $(\al, \be)\cap $ contains no points of $A(\hq)$), there is an angle
$\ga\in G'\cap (\al, \be)$ which is an endpoint of an arc complementary to
$G'$, (pre)periodic by Lemma~\ref{no-crit-leaf}(3). The landing point $z$ of
$R_\ga$ does not belong to $\hq$ and can be associated to the arc $(\al, \be)$.
Clearly, two points associated to such distinct arcs are separated in $F$ by
the set $\tai(\hq)$. Hence $\hq\subset F$ cuts $F$ into at least $\val(\hq)$
components.

Consider the laminations $\lam=\lam^{\B_\C, \RR}, \clam=\clam^{\B_\C, \RR}$ and
the set $\G=\G(\B_\C, \RR)$. Set $\tg=F_\di^\G(p)$. By Lemma~\ref{dynam}
$P(F_\C^{\G}(p))=F_\C^{\G}(p)$ and, by Lemma~\ref{sep-gap} (3),
$\si^*(\tg)=\tg$. Also, by Lemma~\ref{sep-gap} $\tg$ is an infinite invariant
gap, and by Lemma~\ref{no-crit-leaf}(3) all leaves on the boundary of $\tg$ are
(pre)periodic. By the construction $F_\C^{\G}(p)\subsetneqq F_\C^{\RR}(p)$ and
$\tg\subsetneqq G$.

\smallskip

\noindent \textbf{Claim A.} \emph{Except for $\tg$ and leaves from
$\bd(G)$, there are no fixed or
periodic disk fibers of $\clam$ contained in $G$. All periodic points or leaves in $\bd(G)$ which are not contained in
$\bd(\tg)$, are limits of elements of $\lam$ from within $G$ which separate
these periodic points or leaves from the rest of $G$. Moreover, all periodic
leaves in $\bd(G)$ that are not contained in  $\bd(\tg)$, are pairwise
disjoint.}

\smallskip

\noindent\emph{Proof of Claim A.} Let us first show that \emph{if $E\subset G,
E\ne \tg$ is a periodic disk fiber of $\clam$, then $E$ is a leaf from
$\bd(G)$}. Indeed, by Lemma~\ref{basic2} $E$ is either a leaf of $\lam$ or a
gap-leaf of $\clam$. In the first case the claim follows since
$\lam=\lam^{\B_\C, \RR}$ and $\B_\C$ is formed by a wandering cut-continuum
$Q$. So we may assume that $E$ is a gap-leaf of $\clam$ which is not a leaf of
$\lam$. If $E$ is infinite, then by Lemma~\ref{sep-gap}(3) $F_\C^\G(E)$
contains a point $p'\in \an$. Since $E\ne \tg$, then $p'\ne p$. However, by the
construction $F_\C^\G(E)\subset F$ and by Lemma~\ref{sep-gap}(3) the set
$F_\C^{\RR}(p)$ contains a unique point of $\an$, namely $p$. This
contradiction implies that $E$ is finite. Then by Lemma~\ref{sep-gap}(2) there is a
periodic point $x\in \RR$ such that $E\subset G(x)$. Since $E\subset G$ and  by
the construction it is easy to see that $E$ is  a boundary leaf of $G$.

Suppose that a periodic leaf $\ell\subset \bd(G)$ or a periodic point $\al\in
G'$, which is not contained in  $\bd(\tg)$, is not a limit of elements of
$\lam$ from within $G$. Then there must exist a periodic gap of $\clam$
contained in $G$ and containing $\ell$ (or $\al$) in its boundary. This
contradicts the previous paragraph and shows that all periodic points or leaves
in $\bd(G)$ which do not come from $\bd(\tg)$, are limits of elements of $\lam$
from within $G$.

Finally, it is easy to see that all periodic leaves in $\bd(G)$ which do not
come from $\bd(\tg)$, are pairwise disjoint; indeed, elements of $\lam$, which
approach a periodic leaf in $\bd(G)$, cut it off other leaves in $\bd(G)$, that
implies the desired and proves Claim A. \hfill \qed

Since $\tg\subsetneqq G$, there are leaves of $\bd(\tg)$ \emph{inside} $G$. By
Lemma~\ref{no-crit-leaf-2} they are (pre)periodic. Let $\ell\subset \bd(\tg)$ be a
(pre)periodic leaf inside $G$; \emph{we show that $\ell$ can be assumed to have
fixed endpoints}. Indeed, $\ell$ is a limit leaf of sets $\ch(A(\hq_i))$ where
$\hq_i$ are elements of the grand orbit of $Q$. By the properties of such sets,
established in the beginning of the proof, all sets $A(\hq_i)\subset G'$
consist of non-endpoints of complementary to $G'$ arcs. Therefore and by
continuity of $\si$, $\ell$ can never be mapped to the boundary leaves of $G$.
Replacing $\ell$ by its appropriate image and using a power of $\si$, we may
assume that $\ell$ has fixed endpoints.

Let $\tq\in \G(Q)$ be such that the convex hull $\ch(A(\tq))\subset G$ is close
to $\ell$. Then $A(\tq)$ is repelled away from $\tg$ to a component of $G\sm
\ch(A(\tq))$ disjoint from $\tg$. Denote by $Z$ this component united with
$\ch(A(\tq))$. Let us now construct a set $\tz$. Denote by $Y_1, \dots, Y_k$
the fixed leaves in $\bd(Z)$ and the fixed points in $\bd(Z)$ which are not
endpoints of complementary to $G'$ arcs. By Claim A the sets $\{Y_i\}$ are pairwise
disjoint. Choose pairwise disjoint elements $\ts_i$ of $\lam$ contained in $Z$
very close to each $Y_i$ (this is possible by Claim A). Let $F_i$ be the
component of $Z\sm \ts_i$ containing $Y_i$. Set $\tz=Z\sm \bigcup F_i$. By
choosing $\ts_i$ very close to $Y_i$, we may assume that all $F_i$ are pairwise
disjoint with each other and with $\ch(A(\tq))$ and that their images are
contained in $\tz$.

Let $r:G\to \tz$ be a retraction. Define a new map $g=r\circ \si^*:\tz\to \tz$.
Let $a\in \tz$ be a $g$-fixed point. Then it is easy to see that by the
construction $a\nin \bd(\tz)$. Therefore $a$ is actually $\si^*$-fixed. If $a$
belongs to the interior of a gap of $\clam$, then this gap must be
$\si^*$-invariant which contradicts Claim A. If $a$ belongs to a leaf of
$\clam$, then, since by the construction this leaf cannot belong to $\lam$, it
follows that there exists a fixed gap-leaf of $\clam$ containing $a$. This
again contradicts Claim A.
\end{proof}

Now assume that $\B_\C$ is a wandering collection of non-(pre)critical
cut-continua. Lemma~\ref{no-wand} implies the next corollary.

\begin{cor}\label{wrinr}
Every element of $\lam^{\B_\C}$ is contained in a finite wandering gap of
$\clam^\RR$. The infinite gaps of $\clam^\RR$ and $\clam^{\B_\C, \RR}$ are the
same. Gap-leaves of $\clam^\RR$, and gap-leaves of $\clam^{\B_\C, \RR}$
disjoint from leaves of $\lam^{\B_\C}$, are the same. Any limit leaf of
$\clam^{\B_\C}$ is a limit leaf of $\clam^\RR$ from the same side. All-critical
gap-leaves of $\clam^{\B_\C}$ are all-critical gap-leaves of $\clam^\RR$.
\end{cor}

\begin{proof}
First we show that \emph{every element of $\lam^{\B_\C}$ is contained in a
finite wandering gap of $\clam^\RR$}. Clearly, every element of $\lam^{\B_\C}$
is contained in a gap of $\clam^\RR$. By Theorem~\ref{no-wand} this gap of
$\clam^\RR$ is finite. Thus, if we add $\lam^{\B_\C}$ to $\clam^\RR$, we can
possibly break some finite gaps of $\clam^\RR$ into smaller gaps but otherwise
we will not change $\clam^\RR$. Obviously, the finite gaps of $\clam^\RR$,
containing wandering gaps from $\lam^{\B_\C}$, are wandering themselves. This
implies that the \emph{infinite gaps of $\clam^\RR$ and $\clam^{\B_\C, \RR}$
are the same}.

\emph{Let us prove that gap-leaves of $\clam^\RR$ disjoint from leaves of
$\lam^{\B_\C}$, and gap-leaves of $\clam^{\B_\C, \RR}$ disjoint from leaves of
$\lam^{\B_\C}$, are the same}. Clearly, a gap-leaf of $\clam^\RR$, disjoint
from leaves of $\lam^{\B_\C}$, remains a gap-leaf of $\clam^{\B_\C, \RR}$. Now,
let $G$ be a gap-leaf of $\clam^{\B_\C, \RR}$ disjoint from leaves of
$\lam^{\B_\C}$. Then its boundary leaf $\ell$ is either from $\clam^\RR$, or is
a limit leaf of $\lam^{\B_\C}$. In the latter case the elements of
$\lam^{\B_\C}$ which approach $\ell$ are contained in gaps or leaves of
$\clam^\RR$ (by the already proven). Hence in any case $\ell\in \clam^\RR$. So,
\emph{all} leaves in $\bd(G)$ belong to $\clam^\RR$ and $G$ is a gap-leaf of
$\clam^\RR$.

\emph{Next we show that any limit leaf $\ell'$ of $\clam^{\B_\C}$ is a limit leaf
of $\clam^\RR$ from the same side}. By the above, leaves of $\lam^{\B_\C}$,
converging to $\ell'$, are contained in finite gap-leaves of $\clam^\RR$; we
may assume that these gap-leaves of $\clam^\RR$ are all distinct. Hence $\ell'$
can be approximated from this side by distinct gap-leaves of $\clam^\RR$, and
therefore it can be approximated from the same side by leaves of $\lam^\RR$.
The last claim of the lemma concerning all-critical gap-leaves now follows from
this and Lemma~\ref{no-crit-leaf}.
\end{proof}

Thus, with the help of Theorem~\ref{no-wand} we have established a relation
between the geometric prelaminations $\lam^\RR, \clam^\RR$ and $\lam^{\B_\C},
\clam^{\B_\C}$.

\subsection{Recurrent critical points in CS-sets}\label{recinhe}

In this subsection we show that each CS-set contains a recurrent critical point
whose limit set contains the mother hedgehog associated to the CS-set. To this
end we need a result of \cite{bm} (in \cite{bm} it was used to study Milnor
attractors of rational functions with dendritic critical limit sets).

Let $g$ be a rational function. For a Jordan disk $V$ with a pullback $W$, let
the \emph{recurrent} criticality of $W$ be the number of \emph{recurrent}
critical points (with multiplicities) in the pullbacks of $V$ all the way to
$W$. Given two concentric round disks $D_1\subset D_2$ of radii $r_1<r_2$ say
that $D_1$ is \emph{$k$-inside $D_2$} if $r_1/r_2<k$. Let $\e>0, 0<k<1,
\gamma>0, r\in \mathbb N$. Then by Theorem 3.5~\cite{bm} there exists $\da>0$
with the following properties. Let $V'$ be a round disk of diameter less than
$\da$, $\ga$-distant from parabolic and attracting points. If the recurrent
criticality of a $g^N$-pullback $V''$ of $V'$ is $r$, then for any disk
$U'\subset V'$ which is $k$-inside $V'$, the diameter of any $g^N$-pullback
$U''\subset V''$ of $U'$ is less than $\e$ and the criticality of $g^N|_{U''}$
is at most $d+r$. A standard argument, based upon the Shrinking
Lemma~\cite{lm}, then implies that \emph{the diameter of pullbacks $U''$ of
$U'$ tends to zero uniformly with respect to $N$}.

Theorem~\ref{mh1} uses notation from Subsection~\ref{hedge} and ideas of
\cite{chi06}. It implies Theorem~\ref{intro-assoc}(1) for connected Julia sets.

\begin{thm}\label{mh1}
Let $p$ be a CS-point and $\orb p$ be its cycle. Then there exists a recurrent
critical point $c_{\orb p}$, weakly non-separated from a point $q\in \orb p$,
such that $\bd(M_p)\subset \om(c_{\orb p})$. Distinct CS-cycles correspond to
distinct recurrent critical points so that the number of CS-cycles is less than
or equal to the number of recurrent critical points of $P$.
\end{thm}

\begin{proof}
Assume that $P(p)=p$ and, by abuse of notation, $\orb p=p$. By definition the
set of all points which are weakly non-separated from $p$ is $F^\RR_\C(p)\cap
K_P$, so we need to find the desired critical point in $F^\RR_\C(p)$. Clearly,
$F^\RR_\C(p)$ contains $M_p$: if there are hedgehogs, it follows from the fact
that $F^\RR_\C(p)$ contains all hedgehogs at $p$ (recall, that hedgehogs
contain no periodic points distinct from $p$), and if $M_p=\ol{\Da}$ is the
closure of a Siegel disk $\Da$, then it follows from the fact that $\Da\subset
F^\RR_\C(p)$.

We need the following construction which begins with the choice of constants.
Choose $N$ so that if $X$ is the union of sets $\tai(x)$ over the set of
periodic points of $P$ of period less than $N$ then there exists $n$ such that
the following holds:
\begin{enumerate}

\item if $A$ is the component of $\C\sm P^{-n}(X)$ containing $p$ then all
    critical points of $P|_A$ belong to $F^\RR_\C(p)$; and

\item each component of $\C\sm P^{-n}(X)$ contains at most one Cremer point
    or Fatou component.

\end{enumerate}
Clearly, if $N$ is big then (1) follows by the definition of $F^\RR_\C(p)$
while (2) follows from \cite{kiwi00}. By definition, $P$ is a proper map of
$A$ onto $P(A)$. Moreover, $\RR$ is invariant, and so if $U\subset A$ is a Jordan disk, then its pullbacks are
either contained in $A$, or disjoint from $A$. Thus, if we choose a backward
orbit of $x\in A$ which consists of points of $A$, then all corresponding
pullbacks of $U$ are contained in $A$.

Let the set of recurrent critical points of $P$ in $A$ be $E$; then by (1) we
have $E\subset F^\RR_\C(p)$. Let the union of their limit sets be $\om(E)$. By
way of contradiction suppose that $\bd(M_p)\not\subset \om(E)$. Choose a
non-parabolic point $x\in \bd(M_p) \sm \om(E)$. By Theorem~3.5 \cite{bm},
described in the beginning of Subsection~4.2, this implies that a small
neighborhood of $x$ has pullbacks inside $A$ which converge to $0$ in diameter
uniformly with respect to the order of the pullback (alternatively, one can
refer here to a similar result of \cite{ts00}).

However, this contradicts the fact that $P$ on a hedgehog (or, in the case when
$M_p=\ol{\Da}$ is the closure of a Siegel disk, on the closed invariant Jordan
disk contained in $\ol{\Da}$) is a recurrent diffeomorphism (see
Subsection~\ref{hedge}). The contradiction implies that $\bd(M_p)\subset
\om(E)$.

\emph{Let us show that then there exists at least one critical point $c_p$ with
$\bd(M_p)\subset \om(c_p)$}. Consider first the case when there are no true
hedgehogs and $M_p=\ol{\Da}$ where $\Da$ is a Siegel disk. Then there exists a
point $x\in \bd(\Da)$ with a dense orbit in $\bd(\Da)$ (see, e.g.,
\cite{her85}). It is now enough to choose a point $c_p\in E$ such that $x\in
\om(c_p)$. Now, suppose that there are true hedgehogs. Since the map is
transitive on each hedgehog, similarly to the above for each hedgehog $H$ there
exists at least one critical point $c_H\in E$ such that $H\subset \om(c_H)$.
By way of contradiction assume that there is no critical point $c\in E$ such that
$\bd(M_p)\subset \om(c)$. This means that for each critical point $c\in E$ there exists
a hedgehog $H_c\ni p$ such that $H_c\not\subset \om(c)$. Consider the set
$H'=\bigcup_{c\in E}H_c$ contained (by construction) in $\bd(M_p)$. Since all hedgehogs are
invariant and by the Maximum Principle, the set $H'$ is forward invariant and onto.

We claim that there exists a hedgehog $H\supset H'$. Indeed, consider the case when $p$ is a Cremer
fixed point (the case when $p$ is Siegel is similar). Then we claim that the set $H'$
is a  continuum which does not separate the plane. Indeed, suppose otherwise. Since
each hedgehog is non-separating, this can only happen if there exists a bounded Fatou domain $U$
complementary to $H'$. By Sullivan \cite{sul85} we may assume that $U$ is periodic.
Then by Kiwi \cite{kiwi00} there exists a repelling or parabolic point $z$ and
two rays landing at $z$ such that their union separates $U$ from $p$. However this would imply
that $p\in H'$ contradicting the fact that hedgehogs do not contain repelling or
parabolic periodic points. Thus, $H'$ is a   continuum which does not separate the plane.
We then can choose a tight topological disk $V$ containing $H'$ and not containing any
critical points. Clearly, the hedgehog $H=H(U)$ generated by $U$ contains $H'$ as desired.

It remains to observe that by the above there exists a critical point $c\in E$ such that
$E\supset H\supset H'$ while on the other hand the construction implies that this is impossible.
This contradiction shows that we can find a
critical point $c_p\in E\subset F^\RR_\C(p)$ with $M_p\subset \om(c_p)$.
Observe that by (2) distinct fixed CS-points $p$ correspond to distinct
recurrent critical points. The result for periodic CS-points can be proven
similarly. Summing up over all CS-cycles we get the last claim of the theorem.
\end{proof}

\section{Main theorem for connected Julia sets}\label{maint}

Section~\ref{maint} contains the proof of the main theorem in the connected
case (see Theorem~\ref{maintech}). We find an upper bound on the number of
dynamical phenomena such as non-repelling cycles and wandering
non-(pre)critical branch continua which inevitably has to depend on the degree
of the polynomial. We also suggest a bound which depends on specific types of
critical points of a map. This is reflected in Theorem~\ref{intro-assoc} and
Theorem~\ref{intro-ineq}, where we speak \emph{only} of weakly recurrent
critical points and escaping critical points (the latter does not apply in the
case of  connected Julia sets). As we will see, the critical points which we
need to use can be drawn from an even more narrow class.

Below we first study all-critical recurrent gap-leaves of $\clam^{\RR}$. Note
that a disk fiber with a critical leaf on its boundary cannot be a leaf of
$\lam^{\B_\C, \RR}$ because there are no critical leaves in $\lam^{\B_\C, \RR}$.
Hence by Lemma~\ref{basic2} a disk fiber of $\clam^{\B_\C, \RR}$ with a critical
leaf on its boundary is a gap-leaf of $\clam^{\B_\C, \RR}$; by
Lemma~\ref{no-crit-leaf} this disk fiber is an all-critical gap-leaf.

\begin{lem}\label{wandlem1}
Let $G_1, \dots, G_l$ be the all-critical gap-leaves of $\ol{\lam^\RR}$. Then the
following properties hold.

\begin{enumerate}

\item For each $i$ and $m$, the sets $F^\RR_\C((\si^*)^m(G_i))$ are
    disjoint from impressions of all angles not from $(\si^*)^m(G_i')$ and
    do not contain preimages of points of $\an$; moreover,
    $P^k(F^\RR_\C(G_i))\cap F^\RR_\C(G_i)=\0$ for any $k>0$. In particular:

    \begin{enumerate}

    \item[(a)] $F^\RR_\C(G_i)$ contains no periodic points, and if
        $G_i$ is (pre)periodic then $F^\RR_\C(G_i)\cap K_P$ is
        degenerate;

    \item[(b)] any point of any image of $F^\RR_\C(G^i)$ is
        weakly separated from any point outside that image.

    \end{enumerate}

\item For each $i$ there is at least one  critical point in $F^\RR_\C(G_i)$.

\item $G_i$ is recurrent if and only if any $x\in F_\C^\RR(G_i)$ is weakly
    recurrent (in particular, in this case $x$ is not (pre)periodic).
    If $G_i, G_j$ have distinct grand orbits then any two points
    from the grand orbits of $F_\C^\RR(G_i), F_\C^\RR(G_j)$ can be separated
    by a set $\tai(a), a\in \RR$.

\item If $x_i\in F_\C^\RR(G_i), x_j\in F_\C^\RR(G_j)$ with $G_i, G_j$
    recurrent, then we have $\om(G'_i)=\om(G'_j)$ if and only if for any
    $a\in\om(x_i)$ there is $b\in\om(x_j)$ such that $a$ and $b$ are
    weakly non-separated and vice versa (in this case call $\om(x_i)$ and
    $\om(x_j)$ \emph{weakly equivalent}).
\end{enumerate}

\end{lem}

\begin{proof}
(1) \emph{We show that, for any $i$ and $m$, $F^\RR_\C((\si^*)^m(G_i))$
contains no preimages of points of $\an$}.  By Lemma~\ref{no-crit-leaf} all
boundary leaves of $G_i$ are limit leaves and all $\si^*$-images of $G_i$
(which are points because $G_i$ is all-critical) are separated from the rest of
the circle by sequences of leaves of $\lam^\RR$. By Lemma~\ref{finmany}(1)
$F_\C^\RR((\si^*)^m(G_i))$ is disjoint from impressions of all angles not from
$(\si^*)^m(G_i')$. So, if $F^\RR_\C((\si^*)^m(G_i))$ contains a point of $\an$,
then by Lemma~\ref{sep-gap}(3) there are infinitely many angles with principal
sets in $F^\RR_\C((\si^*)^m(G_i))$, a contradiction. Sets
$F^\RR_\C((\si^*)^m(G_i))$ are non-separating with no preimages of Cremer
points.

\emph{Let us prove the rest of \rm{(1)}}. By Lemma~\ref{no-crit-leaf}(1)
$\si^s(G_i)\cap G_i=\0, s>0$. By the above, Lemma~\ref{finmany}(1), and
Lemma~\ref{dynam}, $P^k(F_\C^\RR(G_i))\cap F_\C^\RR(G_i)=\0$ for any $k>0$. The
claims (1a) and (1b) now follow easily.

(2) Since by (1) $\tai'(F_\C^\RR(\si^*(G_i)))$ is a tree-like continuum, and
$P|_{\tai'(F_\C^\RR(G_i))}:\tai'(F_\C^\RR(G_i))\to \tai'(F_\C^\RR(\si^*(G_i)))$
is not one-to-one, by \cite{hea96} there are critical points in
$F_\C^\RR(G_i)$.

(3) Both claims follow easily from the definitions and
Lemma~\ref{no-crit-leaf}.

(4) \emph{Let us prove that if $\om(G'_i)=\om(G'_j)$ then the sets $\om(x_i),
\om(x_j)$ are weakly equivalent.} If $a\in \om(x_i)$, then $P^{s_n}(x_i)\to a$
for a sequence $n_i\to \iy$. Assume that $\si^{s_n}(G'_i)\to \al\in
\om(G'_i)=\om(G'_j)$ and choose a sequence $t_n$ such that $\si^{t_n}(G'_j)\to
\al$. Since $G_i$ is recurrent, $\si(G'_i)\in \om(G'_j)$ and hence we may
assume that $\si^{t_n}(G'_j)$ approach $\al$ from the same side as
$\si^{s_n}(G'_i)$. By compactness we may assume that $P^{t_n}(x_j)\to b$. Let
us show that $a$ and $b$ are weakly non-separated. Indeed, otherwise there
exists a cut $\cu^\ell, \ell=\be\ga\in \lam^\RR$ which separates $a$ from $b$.
Choose $N$ so large that $P^{s_n}(x_i)\in W^\ell_\C(a)$ and $P^{t_n}(x_j)\in
W^\ell_\C(b)$ for $n\ge N$. Since $a$ and $b$ are separated by $\cu^\ell$, the
open planar wedges $W^\ell_\C(a)$ and $W^\ell_\C(b)$ are disjoint, and hence
disk wedges $W^\ell_\di(a)$ and $W^\ell_\di(b)$ are disjoint.

Since all points of $P^{s_n}$-images of rays with arguments from $G'_i$ are
weakly non-separated from $P^{s_n}(x_i)$, the entire set $P^{s_n}(F_\C(G_i))$
is contained in $\bw^\ell_\C(a)$ and hence $\si^{s_n}(G'_i)$ belongs to the
disk wedge $W^\ell_\di(a)$. Analogously, the angles $\si^{t_n}(G'_j)$ belong to
the disk wedge $W^\ell_\di(b)$. This contradicts the fact that
$\si^{t_n}(G'_j)$ approach $\al$ from the \emph{same} side as $\si^{s_n}(G'_i)$
and proves that sets $\om(x_i), \om(x_j)$ are weakly equivalent.

\emph{Let us now prove that if the limit sets $\om(x_i), \om(x_j)$ are weakly
equivalent then $\om(G'_i)=\om(G'_j)$.} Suppose that $\om(G'_i)\not \subset
\om(G'_j)$ while sets $\om(x_i), \om(x_j)$ are weakly equivalent. Since both
$G_i$ and $G_j$ are recurrent, $\om(G'_i)\not \subset \om(G'_j)$ implies that
$\si(G'_i)\nin \om(G'_j)$. Choose a sequence $\Sigma$ of leaves of $\lam^\RR$
which converge to $\si(G'_i)$ separating it from the rest of $\uc$. Then leaves
of $\Sigma$ eventually separate $\si(G'_i)$ from $\om(G'_j)$ which implies that
$P(x_i)$ cannot be weakly non-separated from a point of $\om(x_j)$, a
contradiction with $\om(x_i), \om(x_j)$ being weakly equivalent.
\end{proof}

\begin{dfn}\label{ctilwr}
We introduce the following sets of critical points.

\begin{enumerate}

\item Let $C_{at}$ \index{$C_{at}$} be the set of critical points, belonging to
    parattracting periodic Fatou domains.

\item Let $C_{cs}$ \index{$C_{cs}$} be the set of \emph{recurrent} critical
    points $c$ belonging to CS-sets.

\item If $c\in F_\C^\RR(G)$ is a critical point, where $G$ is an
    all-critical recurrent gap-leaf of $\lam^\RR$, then $c$ is called
    \emph{all-critical (associated to $G$)} \index{all-critical point} ;
    denote by $C^{ac}_{wr}$ \index{$C^{ac}_{wr}$} the union of all such
    critical points.

\end{enumerate}

\end{dfn}

Clearly, $C_{at}\cap C_{cs}=\0$ and $C_{at}\cap C^{ac}_{wr}=\0$. By
Theorem~\ref{no-wand} (see also Lemma~\ref{wandlem1}), $C_{cs}\cap
C^{ac}_{wr}=\0$. It is clear that $C_{at}\cup C_{cs}\subset C_{wr}$.
Since the all critical gap-leaf $G$ is recurrent,
 $C^{ac}_{wr}\subset C_{wr}$. Now we define an equivalence relation among the limit sets of
critical points from $C_{at}\cup C_{cs}\cup C^{ac}_{wr}$ (for points of
$C^{ac}_{wr}$ it is already introduced in Lemma~\ref{wandlem1}).

\begin{dfn}\label{weakeq}
Limit sets $\om(c), \om(d)$ of critical points $c, d$ are called
\emph{weakly equivalent} \index{weakly equivalent} if (1) $c, d\in
C_{at}$ belong to the same cycle of parattracting Fatou domains, or (2)
$c, d\in C_{cs}$ belong to the same CS-set, or (3) $c, d\in
C^{ac}_{wr}$ so that for any $a\in \omega(c)$, there is $b\in
\omega(d)$ weakly non-separated from $a$, and vice versa.
\end{dfn}

By  Lemma~\ref{wandlem1} the weak equivalence is indeed an
equivalence relation.

\begin{lem}\label{characwr}
A critical point $c$ belongs to $C^{ac}_{wr}$ if and only if all its images are
weakly separated from points of $\an$, $c$ is weakly recurrent, and $P(c)$ does
not belong to a wandering cut-continuum.
\end{lem}

\begin{proof}
Suppose that $c\in C^{ac}_{wr}$. By Lemma~\ref{wandlem1}(1) and by
Lemma~\ref{sep-gap}(3), all images of $c$ are weakly separated from $\an$. By
the above $c$ is weakly recurrent. Now, by definition there exists an
all-critical recurrent gap-leaf $G$ with $c\in F_\C^\RR(G)$. If $P(c)\in W$
where $W$ is a wandering cut-continuum, then by Lemma~\ref{wand-fin1} there are
at least two rays with principal sets in $W$. Since by Lemma~\ref{wandlem1}
$\si(G')$ is the \emph{only} angle whose impression is non-disjoint from
$F_\C^\RR(\si(G'))$, then $W\not\subset F_\C^\RR(\si(G'))$. Hence $W$ connects
the point $P(c)\in F_\C^\RR(\si(G'))$ to points outside this fiber which
implies that $W$ cannot be wandering, a contradiction.

Suppose now that $c$ is a weakly recurrent critical point with all images
weakly separated from points of $\an$, and $P(c)$ does not belong to a
wandering cut-continuum. Then $c$ does not map to an attracting or CS-cycle.
By definition of weak recurrence $c$ does not map to a parabolic or repelling cycle and
$c$ is not (pre)periodic. Let us show that the disk fiber $F^\RR_\di(P(c))$ is a point
separated from the rest of the circle by leaves of $\lam^\RR$. Indeed,
otherwise there are the several cases. First, by Lemma~\ref{wandlem1}
$F^\RR_\di(P(c))$ cannot be an infinite gap because then some image of $c$ will
be weakly non-separated from a point of $\an$. Second, $F^\RR_\di(P(c))$ cannot
be (pre)periodic since otherwise some image of $F^\RR_\di(P(c))$ is a finite
periodic disk fiber which by Lemma~\ref{sep-gap}(2) implies that $c$ is
(pre)periodic, a contradiction.

Hence $F^\RR_\di(P(c))$ is a finite wandering disk fiber. If there are
more than one angles in its basis, then the associated planar fiber
$F^\RR_\di(P(c))$ is a wandering cut-continuum. Indeed, choose a
rational angle in each circle arc adjacent to the basis of
$F^\RR_\di(P(c))$. The corresponding rays have landing points which
belong to distinct components of $J_P\sm F^\RR_\C(P(c))$, and so
$F^\RR_\C(P(c))$ is a wandering cut-continuum, a contradiction.
\end{proof}

By Lemma~\ref{characwr} the set $C^{ac}_{wr}$ can be defined in pure
topological terms (without the system of external rays). It is easy to see that
the same applies also to the sets $C_{at}$, $C_{cs}$. Thus, in terms of
formulations, our results can be viewed as having a topologically dynamical
nature. However, of course, the proofs heavily rely upon the combinatorics of
the map $\si$ and do require constant usage of the system of external rays
which allows one to relate this combinatorics and the dynamics of $P$.

Now we prove Theorem~\ref{maintech} which implies Theorem~\ref{intro-assoc} and
Theorem~\ref{intro-ineq} in the connected case. The relation between wandering
non-(pre)critical branch continua and weak equivalence classes of weakly
recurrent critical points is more complicated than that between non-repelling
cycles and associated critical points, hence Theorem~\ref{maintech} is more
quantitative than Theorem~\ref{mh1}. We use the following notation. For
$H\in\{C^{ac}_{wr},C_{at},C_{cs}\}$, let $K(H)$ \index{$K(H)$} be the number of
classes of weak equivalence of grand orbits of points of $H$ and $L(H)$
\index{$L(H)$} be the number of classes of weak equivalence of the limit sets
of points of $H$.

\begin{thm}\label{maintech}
Consider a non-empty wandering collection $\B_\C$ of non-(pre)critical branch
continua $\{Q_i\}$. Then

$$\sum_{\B_\C} (\val_{J_P}(Q_i)-2)\le K(C^{ac}_{wr})-L(C^{ac}_{wr})\le K(C^{ac}_{wr}) - 1 \le |C^{ac}_{wr}|-1$$

and

$$N_{FC}=K(C_{at})+K(C_{cs})$$

which implies that

$$\sum_{\B_\C} (\val_{J_P}(Q_i)-2)+N_{FC}\le K(C_{wr})-1\le |C_{wr}|-1\le d-2.$$

\end{thm}

\begin{proof}
By Subsection~\ref{planlam}, $\B_\C=\{Q_i\}$ gives rise to a wandering
collection of gaps $\ch(A(Q_i))=G(Q_i)$, all non-(pre)critical by
Lemma~\ref{no-crit-leaf-2}. Therefore Theorem~\ref{doug} applies to the
collection $\{G(Q_i)\}=\B_\di$. By Theorem~\ref{doug}(1) there are critical
leaves which are limits of forward orbits of the sets $G(Q_i)$. By
Lemma~\ref{no-crit-leaf-2} and Corollary~\ref{wrinr}, these leaves come from
the boundaries of all-critical gap-leaves of $\ol{\lam^\RR}$, recurrent by
Theorem~\ref{doug}(1). Denote the collection of these gap-leaves by $\mac_l$.

If $m$ is the number of distinct grand orbits of elements of $\mac_l$, then by
Lemma~\ref{wandlem1}(3),(4) $m$ equals the number of classes of weak
equivalence of grand orbits of all-critical weakly recurrent points from sets
$F_\C^\RR(H_j), H_j\in \mac_l$. If $l$ is the number of distinct limit sets of
elements of $\mac_l$, then by Lemma~\ref{wandlem1} (3),(4) $l$ equals the number
of classes of weak equivalence of limit sets of all-critical weakly recurrent
points from sets $F_\C^\RR(H_j), H_j\in \mac_l$. By Theorem~\ref{doug}(3)
$\sum_{\B_\C} (\val_{J_P}(Q_i)-2)\le m-l$.

Now, denote by $\mac$ the collection of \emph{all} all-critical recurrent
gap-leaves. By definition and Lemma~\ref{wandlem1} (2), (3) \emph{each}
all-critical recurrent gap-leaf corresponds to all-critical weakly recurrent
point(s) in $J_P$. Again, by Lemma~\ref{wandlem1} (3),(4) the number of distinct
grand orbits of these gap-leaves equals $K(C^{ac}_{wr})$ and the number of
distinct limit sets of these gap-leaves equals $L(C^{ac}_{wr})$. The collection
$\mac$ can be obtained by adding new elements to the collection $\mac_l$.
Adding one class of weak equivalence of the grand orbit of an all-critical
weakly recurrent point to $\mac_l$ increases $m$ by \emph{exactly} $1$ and
increases the current count for $l$ by \emph{at most} $1$. Hence, $\sum_{\B_\C}
(\val_{J_P}(Q_i)-2)\le m-l\le K(C^{ac}_{wr})-L(C^{ac}_{wr})$ as desired. The
rest follows from $K(C^{ac}_{wr})\le |C^{ac}_{wr}|$ and $L(C^{ac}_{wr})\ge 1$.

The equality $N_{FC}=K(C_{at})+K(C_{cs})$ follows by definition. Thus,

$$K(C_{at})+K(C_{cs})+K(C^{ac}_{wr})-L(C^{ac}_{wr})\le K(C_{wr})-1\le |C_{wr}|-1\le d-2;$$

\noindent obtained by adding the preceding two inequalities and observing that
$C_{at}, C_{cs}$ and $C^{ac}_{wr}$ are pairwise disjoint subsets of $C_{wr}$.
\end{proof}

Let us show how Theorem~\ref{intro-assoc} and Theorem~\ref{intro-ineq} for
connected Julia sets follow from our results (except the parts dealing with
disconnected Julia sets). Clearly, Theorem~\ref{mh1} implies
Theorem~\ref{intro-assoc}(1) for connected Julia sets (observe that by
Lemma~\ref{l-nonpcr} we can talk about a wandering collection of
non-(pre)critical branch continua $Q_i$ and $\val_{J_P}(Q_i)$ instead of
eventual continua $\hq_i$ and $\eval_{J_P}(\hq_i)$). Since points from
$C^{ac}_{wr}$ are weakly recurrent and weakly separated from all non-repelling
periodic points, then the first inequality of Theorem~\ref{maintech} implies
Theorem~\ref{intro-assoc}.

The statement of Theorem~\ref{intro-ineq} includes an inequality for connected
Julia sets, an inequality concerning phenomena which can happen only in
disconnected Julia sets, and their sum. Thus, now it suffices to consider only
the first inequality of Theorem~\ref{intro-ineq}. If $J_P$ is connected and
there are no wandering non-(pre)critical branch continua, the constants from
Theorem~\ref{intro-ineq} are $N_{co}=0$ and $m=0$. In this case
Theorem~\ref{intro-ineq} claims that $N_{FC}\le |C_{wr}|$ and follows from the
fact that $N_{FC}=K(C_{at})+K(C_{cs})\le |C_{at}|+|C_{cs}|\le |C_{wr}|$. If
there is a non-empty wandering collection $\B_\C$ of non-(pre)critical branch
continua $\{Q_i\}$, then $N_{co}=1, m>0$ and Theorem~\ref{intro-ineq} claims
that $N_{FC}+1 +\sum_{i=1}^m (\val_{J_P}(Q_i)-2) \le |C_{wr}|$ which is what
Theorem~\ref{maintech} proves.

In the rest of the paper we deal with  disconnected Julia sets. Before we
switch to them, we would like to comment on an important difference between the
connected and the disconnected cases. As was mentioned in remark (5) in the
Introduction, in the connected case the objects involved in the inequality are
all of topological nature and can be defined with no regards to the system of
external rays. That system plays a crucial role in the proofs, but can be
avoided as one states the results in the connected case.

This is not so in the disconnected case. More precisely, there are two notions
which simply cannot be defined without invoking the system of external rays.
These are the notion of the valence of a wandering component of $J_P$ and the
notion of a periodic repelling point at which infinitely many rays land.
E.g., the fact that a component $A$ of $J_P$ is wandering, is independent of
the system of rays. However the number of rays accumulating in $A$ cannot be
defined in a way which does not depend on the system of rays (as the valence in
the connected case)

\section{External rays to periodic components of the Julia set}\label{disc1}

This section enables us to use the results for connected Julia sets on
$p$-periodic non-degenerate components $E$ of a disconnected Julia set. We
relate the (polynomial-like) map $P^p$ on a neighborhood of $E$ to a polynomial
$f$, with connected Julia set $J_f$, such that $P^p|_E$ and $f|_{J_f}$ are
conjugate, and establish a connection between external rays of $P$, with
principal sets in $E$, and external rays of $J_f$.

Fix an \emph{arbitrary} polynomial $P$ of degree $d$, with not necessarily
connected Julia set. Set $\iU=\iU(J_P)$. The \emph{equipotential} containing a
point $z\in \iU$ is defined as the closure of the union of all preimages
$P^{-n}(P^n(z))$, $n=1,2,...$ \cite{sul}. Then $\iU$ is foliated by
equipotentials defined by the dynamics of $P$.
Critical points $c\in \iU$ are called \emph{escaping}. Denote by $C_*$ the
set of all preimages $P^{-n}(c)$, $n=0,1,2,...$, of escaping critical points $c$.
A component of an equipotential is a smooth curve if and only if it does not
contain a point of $C_*$.


The \emph{flow of external rays} of $P$ is defined as the gradient flow to the
equipotentials. More precisely, by an \emph{external ray} $R_t$ of $P$ we mean
an unbounded curve $R$, such that either $R$ is smooth, crosses every
equipotential orthogonally and terminates in the Julia set of $P$,\index{smooth
external ray} or $R$ is a one-sided limit of such smooth rays (then the ray is
called \emph{non-smooth} \index{non-smooth external ray} or
\emph{one-sided}).\index{one-sided external ray} An external ray is smooth if
and only if it is disjoint with $C_*$. Every point of $\iU$ belongs to an
external ray, and smooth external rays are dense in $\iU$. \emph{Every}
external ray, whether smooth or not, accumulates in one component of $J$.

The argument $t\in \mathbb{R}/\mathbb{Z}$ of $R_t$ is defined uniquely as the
angle at which $R_t$ goes asymptotically to infinity. If the ray is non-smooth,
then there is precisely one more (non-smooth) external ray with the same
argument. Nevertheless, this will not cause ambiguity, because we will be
speaking about external rays rather than their arguments. Observe that if a
ray is periodic then its argument must be periodic. Vice versa, if an argument
of a ray is periodic, then the ray must be periodic. For the general theory of
external rays, see \cite{arjo}, and for the theory of external rays of
polynomials with disconnected Julia sets, see, e.g., \cite{gm93, leso, lepr}.

The equipotentials and external rays for the polynomial $P_0(z)=z^d$ are
standard circles $|z|=\exp(a)$, $a>0$, and rays $\{r\exp(2\pi it): r>1\}$,
$t\in \mathbb{R}/\mathbb{Z}$, respectively. A more traditional way to define
equipotentials and external rays for an arbitrary polynomial $P$ is as follows.
The map $P$ is conjugate to $P_0$ in a neighborhood of infinity by a univalent
change of coordinates $B$ (the \emph{B\"{o}ttcher coordinates}). Then the
equipotentials and rays of $P$ near infinity are the preimages by $B$ of the
standard circles and rays respectively near infinity. By applying branches of
the inverse function $P^{-n}$, the equipotentials and rays are spread over the
entire basin of infinity $\iU$.

The \emph{level} of a point $z\in \iU$ is a positive number $a=a(z)$
defined as follows. If $|z|$ is large enough, then $B(z)$ is well defined, and
$a(z)$ is said to be the number $\log|B(z)|$. For any other $z$, we choose
$n>0$, such that $|P^n(z)|$ is large, and set $a(z)=d^{-n}
a(P^n(z))=d^{-n}\log|B(P^n(z))|$. It is easy to see that $a(z)$ is well defined
(in fact, $a(z)$ is the so-called \emph{Green's function} of $\iU$). The
levels of two points are equal if and only if they belong to the same
equipotential. Therefore, one can define the \emph{level of an equipotential}
as the level of a point of the equipotential.

The level function also defines the direction from infinity to $J_P$ on every
external ray. For any external ray $R$, the function $a$ restricted to $R$
decreases monotonically from $+\infty$ near $\infty$ to $0$ near the Julia set.
In particular, every external ray is homeomorphic to the standard (open) ray
$\R_+=\{x>0\}$. Every subarc of an external ray starts either at infinity or at a
finite point of $\iU$, and either ends at another point of $\iU$ or accumulates in
the Julia set.

The equipotential of level $a_0$ splits the plane into finitely many open
components, so that the level of a point in the unbounded component is strictly
bigger than $a_0$, and the level of a point in the bounded components is
strictly smaller than $a_0$. If two points $z_1$, $z_2$ lie in different
bounded components of the complement of an equipotential of a given level,
then the subarcs of the external rays through these points between $z_1$ and
$J_P$, and between $z_2$ and $J_P$ respectively, are disjoint (even their
closures are disjoint).

Obviously, \emph{all} equipotentials as well as external rays are smooth
if and only if the Julia set is connected, or, equivalently, the set
$C_*$ is empty. In this case $B$ extends
to a Riemann map from $\iU$ onto the complement
of the unit disk, and one can define   equipotentials
and  rays of $P$ directly by taking preimages by $B$
of the standard circles and rays outside  the unit disk.

In the rest of Section~\ref{disc1}, we assume that $J_P$ is {\it not}
connected. Then $J_P$ has infinitely many components. Consider
$\mathbb{S}^1=\mathbb{R}/\mathbb{Z}$, always understanding it as a circle at
infinity (e.g., arguments of external rays belong to $\uc$). Denote by
$\mathbb{D^*}$ the exterior of the closed unit disk, and let $S^1$ be its
boundary, always understood as a subset of the plane. As usual, we consider the
map $\sigma: z\mapsto z^d$ for $z\in S^1$. We also denote  the map $t\mapsto
dt$ of $\mathbb{S}^1=\mathbb{R}/\mathbb{Z}$ to itself by $\si$. The following
lemma, though simple, serves as a useful tool in what follows.

\begin{lem}\label{cap}
If two different rays $R, R'$ have a common point, then $R, R'$ are both
non-smooth. The intersection $L=R\cap R'$ is connected and can contain a
point of $C_*$ only as an endpoint. Furthermore, one and only one of the
following cases holds:

\begin{enumerate}

\item[(i)] $L$ is a smooth curve joining infinity and
a point of $C_*$,

\item[(ii)] $L$ is a single point of $C_*$,

\item[(iii)] $L$ is a smooth closed arc between two points of $C_*$,

\item[(iv)] $L$ is a smooth curve from a point of $C_*$
to $J_P$ and, moreover, the rays $R, R'$ are not periodic.

\end{enumerate}

Except for the last case, the rays $R, R'$ have their principal sets in
different components of $J_P$.
\end{lem}

\begin{proof}
A smooth ray is disjoint from all other rays. Now, assume that two
\emph{different} non-smooth rays $R, R'$ are not disjoint. Since rays fill up
$\iU$ and smooth rays are dense in $\iU$, the intersection $L$ of $R$
and $R'$ is a connected set (otherwise there is a ``lake'', i.e. a component of
$\C\sm [R\cup R']$, unreachable by smooth rays). Hence $L$ is either (i) a
smooth curve from infinity to a point in $\iU$, or a (ii) single point, or
(iii) a smooth closed arc between two points in $\iU$, or (iv) a smooth curve
from a point of $\iU$ to $J_P$ (a smooth curve from infinity to $J_P$ is
impossible as $R\ne R'$).

Let us show next that $L$ can only contain points of $C_*$ as endpoints.
Let $q\in C_*$ be a point of $L$. Suppose by way of contradiction that $q$ is
not an endpoint of $L$. Consider the component $\gamma$ of the equipotential
through the point $q$. Then $q$ is a \emph{singular point} (\emph{branch
point}) of $\gamma$, and $\mathbb{C}\setminus \gamma$ contains at least two
bounded components with the only joint point on their boundaries to be $q$.

Let $U_1, \dots, U_m$ be the bounded components of $\C\sm \gamma$ containing
$q$ in their closures. Let $U_1$ be the component containing points of $L$.
Choose a neighborhood $W$ of $q$ such that $W\sm \ol{\bigcup U_i}$ consists of
$m$ open components $V_1, \dots, V_m$. Since $q$ is not an endpoint of $L$, $L$
intersects \emph{only} one of the sets $V_1, \dots, V_m$, say, $V_1$. However, as
$L$ is approached by smooth rays converging to $R, R'$ from two \emph{distinct}
sides (of $L$), the smooth rays located on distinct sides of $L\cap V_1$ must
enter distinct sets $U_i$, a contradiction with $U_1$ being the component
containing points of $L$. So, $q$ is an endpoint of $L$.

Note that if $z\in R$ (resp., $z\in R'$) and $z\notin C_*$, then, in a
neighborhood of $z$, $R$ (resp., $R'$) is a \emph{smooth} curve. Hence,
(i) if $L$ is a smooth curve from infinity to a point in $\iU$, then it joins
infinity and a point of $C_*$, (ii) if $L$ is a single point, then it is a
point of $C_*$, and (iii) if $L$ is a smooth closed arc, then its endpoints
belong to $C_*$. The remaining possibility is that $L$ is a smooth curve
joining a point of $C_*$ and $J_P$. Let us show that in this case neither $R$
nor $R'$ can be periodic. Indeed  if $R$ is periodic and contains a point $q\in
C_*$, then $R$ contains infinitely many preimages of $q$ converging to $J_P$.
Hence $L$ would contain infinitely many preimages of $q$, a contradiction.
\end{proof}

\begin{example}
The cases (i) - (iii) are already possible for quadratic polynomials $z^2+c$
with $c>1/4$. Case \rm{(i)} is realized for the two one-sided rays
$R_{0^+}=\lim_{t\to 0^+}R_t$, $R_{0^-}=\lim_{t\to 1^-}R_t$, so that the
intersection of $R_{0^+}$ and $R_{0^-}$ is the positive real axis. Case
\rm{(ii)} happens for the rays $R_{0^+}$ and $R_{1/2^+}=\lim_{t\to 1/2^+}R_t$,
with $R_{0^+}\cap R_{1/2^+}=\{0\}$. Case \rm{(iii)} holds if there are two
points from $C_*$ on the same ray, e.g., the intersection of $R_{0^-}$ and
$R_{1/2^+}$ is an arc joining $0$ and the first preimage of $0$ in the lower
half plane. Finally, if $P(z)=z^2+c$ with $c>1/4$, then any non-smooth ray is
(pre)periodic which by Lemma~\ref{cap} makes case (iv) impossible for $P$.
But it is realized for any $z^2+c$ with $c$ outside of the Mandelbrot set,
for which the external arguments of $0$ are not periodic.
\end{example}




Given $E\subset K_P$, let $A(E)$ \index{$A(E)$} be the set of the arguments of all external
rays with principal sets in $E$ (clearly, these principal sets are in fact
contained in $\bd(E)\subset J_P$). Similarly, for $z\in \iU$ let $A(z)$ be
the set of the arguments of all external rays containing $z$ (since for every
ray its argument is well-defined, the definition is consistent). For $z\in
\iU$ any angle from $A(z)$ is said to be an \emph{(external) argument} of
$z$. Lemma~\ref{Lemma 1} is simple and well-known; we add it here for the sake
of completeness.


\begin{lem}\label{Lemma 1} For a component $E$ of the filled-in Julia set
$K_P$ of $P$, the set $A(E)$ is a non-empty compact
subset of $\mathbb{S}^1$.
\end{lem}

\begin{proof}
Take the arguments of all external rays that cross a component $\gamma$ of the
equipotential of a given level $a>0$ and enter the bounded component of
$\mathbb{C}\setminus \gamma$ which contains $E$.
It is a non-empty compact subset $A_a(E)$
of $\mathbb{S}^1$. As $a\to 0$, these compacta shrink to a non-empty
compact set, which is the set $A(E)$.
\end{proof}


\begin{thm}[Theorem 1 of \cite{lepr}]\label{Theorem LP1}
Let $z$ be a repelling or parabolic periodic point of $P$ of period $m$. Then
the following claims hold.

\begin{enumerate}

\item $A(z)$ is a non-empty compact subset of $\mathbb{S}^1$,
     invariant under $\sigma^m$.

\item If $A(z)$ is infinite, then the point $\{z\}$ is a periodic component
     of $K_P$. The set $A(z)$ contains external arguments $t_q, t_q'$ of a
     critical point $q\in \iU$ of $P^m$. Moreover, the set $A(z)$ is a
     Cantor set, and every forward $\sigma^m$-orbit in $A(z)$ is dense in
     $A(z)$.

\item If $\{z\}$ is not a component of $K_P$, then $A(z)$ is finite.

\item The set $A(z)$ is finite if and only if it contains a periodic point.
     In this case every $t\in A(z)$ is periodic under $\sigma^m$, all with
     the same period.

\end{enumerate}

\end{thm}

From now on assume that $E$ is a periodic \emph{non-degenerate} component of
$K_P$ of period $p$. It happens if and only if $P^p$ has a critical point in
$E$. Since $P$ is a polynomial, by the Maximum Principle, $E$ does not separate
the plane. Fix such $E$, and denote by $\psi: \mathbb{C}\setminus E\to
\mathbb{D^*}$ the Riemann map of the exterior of $E$ onto the exterior of the
unit disk, with $\psi(z)\sim k z$ as $z\to \infty$, for some $k>0$.

For a non-closed curve $l$ from infinity or a finite point in $\C$ to a bounded
region in $\C$, we can define its \emph{principal set $\pr(l)$}
\index{principal set of a curve} analogously to how it is done for conformal
external rays (see Subsection~\ref{cont}). For a continuum $M$, a curve $l$
with $\pr(l)\subset M$ is called a \emph{curve to $M$}
(e.g., this terminology applies to some rays). If $R$ is an external ray of $P$
then $\psi(R)$ is a curve in $\mathbb{D}^*$; the \emph{argument} of $\psi(R)$
is set to be the argument of $R$.
An external ray $R$ (of $P$) to $E$ has $\psi$-image $\hat R:=\psi(R)$.
Then  $\hat R$ is called
an {\it $E$-related ray} (see \cite{lepr}) if and only if $\pr(\hat R)\subset S^1$.
Each $E$-related
ray is a curve from $\infty$ to $S^1$.

The $E$-related ray $\hat R$ is called {\it (non-)smooth} if and only if the
external ray $R$ is (non-)smooth. Fix a simply-connected neighborhood $V$ of
$E$ bounded by an equipotential. Choose a component $U$ of $P^{-p}(V)$, that
is also a neighborhood of $E$. One can assume further that $P^p$ has no
critical points in $\bar U\setminus E$. Denote $\hat V=\psi(V\setminus E), \hat
U=\psi(U\setminus E)$. Note that $\hat V, \hat U$ are ``annuli'' with the inner
boundary $S^1$. Call the intersections of $E$-related rays with $\hat V$
{\it $E$-related arcs} (of $E$-related rays).

The Riemann map $\psi$ induces a conjugated map $g: \hat U\to \hat V$ as
follows: $g=\psi\circ P^p\circ \psi^{-1}$.
It is well known that $g$ extends through $S^1$ to an analytic map in a
neighborhood of $S^1$, and, moreover, $g$ is expanding: there are $n>0$ and
$\lambda>1$, such that, $|(g^n)'(w)|>\lambda$ provided $g^n(w)$ lies in the
closure of $\hat U$, see \cite{pr}, \cite{dh85b} (Proof: by the Reflection
Principle \cite{Ahl1}, $g$ extends to a holomorphic (unbranched) covering map
$g: A\to B$, where $A\subset B$ are ``annuli'' containing $S^1$ in their
interiors, and $A$ is compactly contained in $B$. Then $g$ is lifted to a
univalent map $\hat g: \hat A\to \hat B$ where $\hat A\subset \hat B$ and $\hat
B$ is the universal cover of $B$. It follows that the inverse map $\hat g^{-1}$
strictly contracts the hyperbolic metric on $\hat B$ which implies the
expanding property of $g$.)

Now, $g$ maps intersections of $E$-related rays with $\hat U$ onto
$E$-related arcs. Abusing the notation, say that $g$ \emph{maps $E$-related
rays to $E$-related rays} (i.e., $g$ maps an $E$-related ray of argument $t$ to
an $E$-related ray of argument $\sigma^p(t)$). A curve $l:\R\to \di^*$ with
$\lim_{t\to \iy} l(t)=\{w\}\subset S^1$ approaches $w$ \emph{non-tangentially}
if for some $T$ the set $l([T, \iy))$ is contained in a sector of angle less
than $\pi$ with the vertex at $w$ symmetric with respect to the standard ray
through $0$ and $w$.

\begin{lem}[Lemma 2.1 of \cite{lepr}]\label{Proposition 1}
The following claims hold.

\begin{enumerate}

\item Every $E$-related arc has a finite length, and hence lands at a
    unique point of $S^1$.

\item Every point $w\in S^1$ is a landing point of at least one $E$-related
    ray, and the arguments of the $E$-related rays landing at $w$ form a
    compact subset of $\mathbb{S}^1$.

\item An $E$-related arc $l$ goes to a point $w_l\in S^1$ non-tangentially.

\end{enumerate}

\end{lem}

\noindent \emph{Sketch of the proof}. Part (1) holds as $g$ is uniformly
expanding, so the local branches of inverses $g^{-k}$ are uniformly
exponentially contracting as $k\to \infty$. For part (2) notice, that by
Lemma~\ref{Lemma 1}, there is at least one $E$-related ray. If we take
preimages of an $E$-related ray by all branches of $g^{-k}$, we see (since $g$
is expanding) that $E$-related rays land inside every arc on $S^1$. By the
intersection of compacta we get a non-empty compact set of $E$-related rays
landing at a given point of $S^1$. \hfill \qed

\begin{thm}[Theorem 2 of \cite{lepr}]\label{Theorem LP2}
If $a\in E$ is accessible from the complement
of $E$, then $a$ is accessible by an external ray of $P$.
More precisely, if a curve $l\in \mathbb{C}\setminus E$
converges to $a$, then there exists an external ray $R$
of $P$, which lands at $a$ and is such that $l$ and $R$ are homotopic
among the curves in $\mathbb{C}\setminus E$ which land at $a$.
\end{thm}

\noindent \emph{Sketch of the proof.} Indeed, if a point $a$ of $E$ is
accessible by a curve $l$ from outside of $E$, then the curve $\psi(l)$ lands
at a point $w$ of $S^1$ and $a\in\bd(E)$. Consider an $E$-related ray $L$
landing at $w$. By Proposition~\ref{Proposition 1} (2), it exists, and by
Proposition~\ref{Proposition 1} (3), it tends to $w$ non-tangentially. Hence,
by Lindel\"{o}f's theorem (see, e.g., Theorem 2.16 of \cite{Pom}),
$\psi^{-1}(L)$ and $l$ tend to the same point $a$. \hfill \qed

Lemma~\ref{disconper} studies periodic points of $g|_{S^1}$ and $P^p|_E$.

\begin{lem}\label{disconper}
Let $w\in S^1$ be a periodic point of $g|_{S^1}$. Then the non-tangential limit
of $\psi^{-1}$ at $w$ exists and is a repelling or parabolic periodic point of
$P$ in $\bd(E)$. Moreover, the set of $E$-related rays landing at
$w$ is finite, and each of them is periodic of the same period.
\end{lem}

\begin{proof}
Let $l$ be a curve in $\mathbb{C}\setminus E$ with its principal set in $E$,
invariant under some iterate $P^k$ of $P$. Then $l$ lands at a periodic point
$a\in \bd(E)$ of $P$ (the proof goes back to Fatou, see \cite{fa}, p.81, and
also \cite{Pom2}, \cite{pr2}). By the Snail Lemma (see, e.g., \cite{miln00}),
$a$ is repelling or parabolic. If $w$ is of period $m$, it is easy to find a
$g^m$-invariant curve $\gamma$ landing at $w$; then the curve
$l=\psi^{-1}(\gamma)$ is $P^m$-invariant and, by the above, accumulates on a
repelling or parabolic point $a\in \bd(E)$ of $P$. By Lindel\"{o}f's theorem,
$a$ is the non-tangential limit of $\psi^{-1}$ at $w$. The remaining claim of
the lemma follows from Theorem~\ref{Theorem LP1}(3).
\end{proof}

The map $G:=P^p: U\to V$ is a polynomial-like map of degree $m\ge 2$, such that
$E$ is the (connected) filled-in Julia set $K_G=\{z: G^n(z)\in U,
n=0,1,2,...\}$ of $G$. The map $g$ defined above is called in \cite{dh85b} the
\emph{external map} of $G$. By \cite{dh85b}, $G: U\to V$ is \emph{hybrid
equivalent} to a polynomial $f$ of degree $m$, i. e. there is a quasiconformal
homeomorphism $h$ defined on $V$, which is conformal a.e. on $E$, such that
$f\circ h=h\circ G$ in $U$. The map $h$ is called the \emph{straightening map}.
The filled-in Julia set $K_f=h(E)$ of $f$ is connected. Hence, the B\"{o}ttcher
coordinate $B$ of $f$ is well defined in the basin of attraction of infinity
$\mathbb{C}\setminus K_f$ of $f$. We have there that $B(f(z))=(B(z))^m$.

Since $K_f$ is connected, external rays of $f$ are smooth. For an external ray
$R^f_\tau$ of $f$ of argument $\tau$, its $h^{-1}$-image
$l_\tau:=h^{-1}(R^f_\tau)$ in $V$ is called the {\it polynomial-like ray (to
$E$)} \index{polynomial-like ray} of argument $\tau$. Fix the straightening map
$h$; then the polynomial-like rays are well-defined.
As $h: V\to h(V)$ is a homeomorphism,
$\pr(l_\tau)=h^{-1}(\pr(R^f_\tau))$. Below we refer to different planes and
objects in them by the names of maps acting in them. Thus, $E$-related rays lie
in the $g$-plane, external rays of $P$ and polynomial-like rays are in the
$P$-plane, etc.

The main results of the present section are Theorems~\ref{Theorem 1'} and
\ref{Theorem 1} below. They complete Proposition~\ref{Proposition 1} and
Theorem~\ref{Theorem LP2}.

\begin{thm}\label{Theorem 1'}
Any point $w\in S^1$ is the landing point of precisely one $E$-related ray,
except for when one and only one of the following holds:

\begin{enumerate}

\item[(i)] $w$ is the landing point of exactly two $E$-related rays,
which are non-smooth and have a common arc that goes from
a point of $\psi(C_*)$ to the point $w$;

\item[(ii)] $w$ is a landing point of at least two disjoint rays in which
    case $w$ is a (pre)periodic point of $g$ and some iterate $g^n(w)$
    belongs to a finite (and depending only on $E$) set $\hat Y(E)$ of
    $g|_{S^1}$-periodic points each of which is the landing point of
    finitely many, but at least two, $E$-related rays, which are
    smooth and periodic of the same period depending on the landing point.
\end{enumerate}

Moreover, if $w$ is periodic then {\rm (i)} cannot hold.
\end{thm}


\begin{proof}
Assume that there are two $E$-related rays landing at a point $w\in S^1$ and
that (i) does not hold. We need to prove that then (ii) holds. If (i) does not
hold, then there exist disjoint $E$-related rays landing at $w$. Let us study
this case in detail.

Associate to any such pair of rays an open arc $(\hat R_t, \hat R_{t'})$ of
$\mathbb{S}^1$ ($\uc$ is viewed as the circle at infinity in the $g$-plane) as
follows. Two points of $\mathbb{S}^1$ with the arguments $t, t'$ split
$\mathbb{S}^1$ into two arcs. Let the arc $(\hat R_t, \hat R_{t'})$ be the one
of them that contains no arguments of $E$-related rays except for possibly
those that land at $w$. Geometrically, it means the following. The $E$-related
rays $\hat R_t, \hat R_{t'}$ together with $w\in S^1$ split the plane into two
domains. The arc $(\hat R_t, \hat R_{t'})$ corresponds to the one of them,
disjoint from $S^1$. Let $L(\hat R_t, \hat R_{t'})=\da$ be the angular length of
$(\hat R_t, \hat R_{t'})$. Clearly, 
$0<\da<1$. Now we make a few observations.

(1) \emph{If $E$-related disjoint rays of arguments $t_1, t_1'$ land at a
common point $w_1$ while $E$-related disjoint rays of arguments $t_2, t_2'$
land at a point $w_2\not=w_1$, then the arcs $(\hat R_{t_1}, \hat R_{t_1'})$,
$(\hat R_{t_2}, \hat R_{t_2'})$ are disjoint.}

This follows from the definition of the arc
$(\hat R_{t}, \hat R_{t})$.

(2) \emph{If disjoint $E$-related rays $\hat R_t, \hat R_{t'}$ of arguments $t,
t'$ land at a common point $w$, then $E$-related rays $g(\hat R_t), g(\hat
R_{t'})$ are also disjoint and land at the common point $g(w)$. Moreover,}
$$L(g(\hat R_t), g(\hat R_{t'}))\ge \min\{d^p \da(\text{mod}\, 1),
1-d^p \da(\text{mod}\, 1)\}>0.$$

Indeed, the images $g(\hat R_t), g(\hat R_{t'})$ are disjoint near $g(w)$,
because $g$ is locally one-to-one. By Lemma~\ref{cap}, $g(\hat R_t)\cap g(\hat
R_{t'})=\0$.
Since the argument of $g(\hat R_t)$ is $\si^p(t)=d^pt(\text{mod}\, 1)$, we get
the inequality of (2).

Let us consider the following set $\hat Z(E)$ of points in $S^1$: $w\in \hat
Z(E)$ if and only if there is a pair of disjoint $E$-related rays $\hat R, \hat
R'$, which both land at $w$, and such that $L(\hat R, \hat R')\ge 1/(2d^p)$.
Denote by $\hat Y(E)$ a set of periodic points which are in forward images of
the points of $\hat Z(E)$.

(3) \emph{If the set $\hat Z(E)$ is non-empty, then it is finite,
and consists of (pre)periodic points.}

Indeed, $\hat Z(E)$ is finite by (1). Assume $w\in \hat Z(E)$. Then, by (2)
some iterate $g^n(w)$ must hit $\hat Z(E)$ again.

To complete the proof, choose disjoint $E$-related rays $\hat R_t, \hat R_{t'}$
landing at $w\in S^1$ and use this to prove that all claims of (ii) hold.

{\it We show that the orbit $w, g(w),\dots$ cannot be infinite.} Indeed,
otherwise by (1)-(2), we have a sequence of non-degenerate pairwise disjoint
arcs $(g^n(\hat R_t), g^n(\hat R_{t'}))\subset \uc$, $n=0,1,...$. By (2), some
iterates of $w$ must hit the finite set $\hat Z(E)$ and hence $\hat Y(E)$
(which are therefore non-empty), a contradiction.

{\it Hence for some $0\le n<m$, $g^n(w)=g^m(w)$; let us verify that other
claims of {\rm(ii)} holds.} Replacing $w$ by $g^n(w)$, we may assume that $w$ is a
(repelling) periodic point of $g$ of period $k=m-n$. By (2), $w\in \hat Y(E)$.
By Theorem~\ref{Theorem LP1}, the set of $E$-related rays landing at $w$ is
finite, and each $E$-related ray landing at $w$ is periodic with the same
period. By Lemma~\ref{cap}, each such ray is also smooth. 
Hence, (ii) holds. Finally, the last claim of the lemma follows from by
Lemma~\ref{cap}.
\end{proof}


Let the set $Y(E)$ be the set of non-tangential limits of $\psi^{-1}$ at the
points of $\hat Y(E)$; by Lemma~\ref{disconper} $Y(E)$ is a well-defined finite
set of repelling or parabolic periodic points of $P$ in $\bd(E)$. By
Theorem~\ref{Theorem 1'} all external rays landing at points in $Y(E)$ are
smooth, and at each point finitely many, but at least two, land. All rays
landing at the same point in $Y(E)$ have the same period.

\begin{thm}\label{Theorem 1}
For each external ray $R$ to $E$ 
there is exactly one po\-ly\-no\-mial-like ray $l=\lambda(R)$ with $\pr(l)=\pr(R)$
and the curves $l$ and $R$ homotopic in $\mathbb{C}\setminus E$ among curves
with the same limit set.


Moreover, $\lambda: R\mapsto l$ maps the set of external rays to $E$
\textbf{onto} the set of polynomial-like rays to $E$, and is ``almost injective'':
$\lambda$ is one-to-one except for when one and only one of the following
holds. Suppose that $\lambda^{-1}(\ell)=\{R_1,\dots,R_k\}$ with $k>1$. Then either:

\begin{enumerate}

\item[(i)] $k=2$ and both rays $R_1, R_2$ are non-smooth and share a common arc
to $E$, or


\item[(ii)] there is a (pre)periodic point $z$ such that $\pr(R_i)=\{z\},
    i=1,\dots, k$, at least two of the rays $R_1, \dots, R_k$ are disjoint,
    and, for some $n\ge 0$, $P^{pn}(z)$ belongs to $Y(E)$.

\end{enumerate}

\end{thm}

\begin{proof}
Let $h$ be a quasiconformal homeomorphism defined on a neighborhood of $E$
which conjugates $P$ (restricted on a smaller neighborhood) to a polynomial $f$
with connected Julia set $h(E)$ restricted to a neighborhood of $h(E)$. We can
extend the map $h$ onto the entire $\C$ as a quasiconformal homeomorphism even
though the conjugacy between $P$ and $f$ will only hold on a neighborhood of
$E$. Let $B:\C\to \disk^*$ be the B\"{o}ttcher uniformization map of $f$.

Consider the map $\Psi:=\psi\circ h^{-1}\circ B^{-1}: \disk^*\to \disk^*$ from
the uniformization plane of the polynomial $f$ to the $g$-plane. It is a
quasiconformal homeomorphism
which leaves $S^1$ invariant. For $c\in S^1$, let $L_c=\Psi(r_c\cap
\mathbb{D^*})$ where $r_c=\{tc: t>0\}$ is a standard ray in the uniformization
plane of $f$.


\smallskip

\noindent \textbf{Claim A.} \emph{The curve $L_c$ tends non-tangentially to a
unique point $w_0$ of the unit circle $S^1$. Moreover, for every $w\in S^1$
there exists a unique $c$ such that $L_c$ lands on $w$.}

\smallskip

\noindent \emph{Proof of Claim~A.} This follows from properties of
quasiconformal mappings \cite{Ahl2}. Extend $\Psi$ to a quasiconformal
homeomorphism $\Psi^*$ of $\C$, symmetric with respect to $S^1$, by the
symmetry $\zeta\mapsto 1/\ol{\zeta}$ with respect to $S^1$. Consider the curve
$L^*_c=\Psi^*(r_c)$. It is an extension of the curve $L_c$, which crosses $S^1$
at the point $w_0=\Psi^*(c)$. As a quasiconformal image of the straight line,
the curve $L^*_c$ has the following property \cite{Ahl2}: there exists $C>0$,
such that $|w-w_0|/|w-1/\ol{w}|<C$, for every $w\in L^*_c$. Therefore, $L^*_c$
tends to $w_0$ non-tangentially. The last claim follows from the fact that
$\Psi^*$ is a homeomorphism. \hfill \qed

Let $R$ be an external ray to $E$. By Proposition~\ref{Proposition 1}(3), the
$E$-related ray $\hat R=\psi(R)$ tends to a point $w_0\in S^1$
non-tangentially. By Clam A there exists a unique $L_c$ which lands at $w_0$.
Set $\lambda(R)=  \psi^{-1}(L_c)$. By Lindel\"{o}f's theorem,
$R=\psi^{-1}(\hat R)$ and $l=\psi^{-1}(L_c)$ have the same limit set in $E$.
Since $\hat R$ and $L_c$ are homotopic among the curves which land at $w_0$
non-tangentially, the claim about homotopy follows. By Claim A, the map
$\lambda$ is onto. Observe that the conditions that $\pr(R)=\pr(\ell)$ and that
$R$ and $\ell$ are homotopic outside $E$ among curves with the same limit set,
uniquely determine the polynomial-like ray $\lambda(R)$.

It remains to prove the ``almost injectivity'' of $\lambda$. This is a direct
consequence of Theorem~\ref{Theorem 1'} and the construction above.
\end{proof}

Now we study wandering continua in the disconnected case. Let us
make some remarks. If a wandering continuum $W$ is contained in a
(pre)periodic component of $J_P$, the situation is like  the
connected case, thanks to Theorem~\ref{Theorem 1}; otherwise, the
entire component of $J_P$ containing $W$ wanders. A continuum
$W\subset J_P$ is called \emph{a wandering cut-continuum (of $J_P$)}
\index{wandering cut-continuum} if (1) $W$ is a wandering
\emph{component} of $J_P$ with at least two external rays
accumulating in $W$, or (2) $W\subset E$, where $E$ is a
(pre)periodic component of $J_P$ and $W$ is a wandering
cut-continuum of $E$. The set $\tai(W)$ can be defined in the
disconnected case as in the connected case (only now some rays
accumulating in $W$ may be non-smooth).

Let us now reprove Lemma~\ref{l-nonpcr} in the disconnected case. For
convenience we restate it here with necessary amendments.

\begin{lem}\label{l-nonpcr-1}
If $W$ is a wandering cut-continuum of $J_P$, then $P^n|_{\tai(W)}$ is not
one-to-one if and only if $\tai(W)$ contains a critical point of $P^n$ (in this
case there are two rays in $\tai(W)$ mapped to one ray).
\end{lem}

\begin{proof}
If $W$ is contained in a (pre)periodic component $E$ of $J_P$, then the claim
follows from the proof of the original Lemma~\ref{l-nonpcr} and
Theorem~\ref{Theorem 1} (recall, that Lemma~\ref{nonsep} holds for arbitrary
Julia sets).

Let $W$ be a wandering component of $J_P$. By Lemma~\ref{Lemma 1} the set
$A(W)$ of arguments of all rays accumulating in $W$ is non-empty and compact.
If $P^n|_{\tai(W)}$ is not one-to-one, then, as before, by \cite{hea96} and
Lemma~\ref{nonsep}, there is a critical point $c$ of $P^n$ in $\tai(W)$. Now,
let $c\in \tai(W)$ be a critical point of $P^n$. If $c\in W$, then using
$A(W)\ne \0$ and repeating the arguments from Lemma~\ref{l-nonpcr} we complete
the proof. If, however, $c\nin W$ then $c$ belongs to a ray included in
$\tai(W)$, and Lemma~\ref{cap} completes the proof.
\end{proof}

Lemma~\ref{l-nonpcr-1} shows that Definition~\ref{d-nonpcr} can be
given in the disconnected case in  literally the same way, as in the
connected case. That is, a wandering continuum $W\subset J_P$ is
said to be \emph{non-(pre)critical} if $\tai(W)$ is such that for
every $n$ the map $P^n|_{\tai(W)}$ is one-to-one. Equivalently, $W$
is non-(pre)critical if and only if $\tai(W)$ contains no
(pre)critical points (or if and only if no iterate of $\tai(W)$
contains a critical point of $P$); then, clearly, each ray $R$ with
$\pr(R)\subset W$ is smooth. If $W$ is contained in a (pre)periodic
component of $J_P$, this component (which must be non-degenerate) is
denoted by $\he(W)$, the corresponding component of $K_P$ is denoted
by $E(W)$, and $\val'_{J_P}(W)$
\index{valence!$\val'_{J_P}(\cdot)$} is defined as
$\val'_{\he(W)}(W)$, i.e. the number of components of
$\he(W)\setminus W$. Recall also, that we define $\val_{J_P}(W)$ as the number of
external rays of $P$ with principal sets in $W$.

\begin{cor}\label{wandiscon}
Let $W\subset J_P$ be a wandering non-(pre)critical cut-continuum contained in
a periodic component $\he(W)$ of $J_P$. Then $\val'_{J_P}(W)=\val_{J_P}(W)=|A(W)|=M<\iy$. The
polygon $B_W$, whose basis is $A(W)$, is wandering and non-(pre)critical under
$\sigma$, and if $W_1, W_2$ are two continua as above with disjoint orbits,
then the $\sigma$-orbits of the polygons $B_{W_1}, B_{W_2}$ are pairwise
unlinked.

Moreover, $M$ equals the number of components of $E(W)\sm W$.
Also, if $W'$ is any non-(pre)critical element of the grand orbit $\G(W)$,
then $\eval(W')= \val_{\he(W)}(W)=\val_{P^n(\he(W))}(P^n(W))$ for all $n\ge 0$.
\end{cor}

\begin{proof}
Let $M=\val_{J_P}(W)=|A(W)|$. Let us consider the relation between
polynomial-like rays to $\he(W)$ and external rays to $\he(W)$. To
each polynomial-like ray $T$ to $\he(W)$ we associate by
Theorem~\ref{Theorem 1} a \emph{unique} external ray $R$ homotopic
to $T$ outside $\he$; the ray $R$ is unique because $W$ is
non-(pre)critical (and hence the case (i) from Theorem~\ref{Theorem
1} is impossible) and wandering (and hence the case (ii) from
Theorem~\ref{Theorem 1} is impossible). Since by
Theorem~\ref{Theorem 1} this describes \emph{all} external rays
whose principal sets are in $W$, we see that there is the same
number of external rays to $\he(W)$ and polynomial-like rays to
$\he(W)$. Thus, there are $M$ polynomial-like rays to $\he(W)$. By
Corollary~\ref{wand-fin}, $M$ equals $\val'_{J_P}(W)$, the number of
components of $\he(W)\setminus W$, as desired. Moreover, since $W$
is non-(pre)critical, then by Lemma~\ref{l-nonpcr-1},
$P^n_{\tai(W)}$ is one-to-one and
$M=\val_{\he(W)}(W)=\val_{P^n(\he(W))}(P^n(W))$ for all $n\ge 0$.
This implies that $\eval(W')= \val_{\he(W)}(W)$ for any
non-(pre)critical element $W'$ of the grand orbit $\G(W)$.

\emph{We claim that $W$ is disjoint from the boundary of any Fatou
domain}. Indeed, suppose otherwise. Then we may assume that $W\cap
\bd(U)\ne \0$ where $U$ is a fixed Fatou domain. Consider two rays
$R_1, R_2$ with principal sets in $W$; define $T(R_1, R_2)=T_0$ as the
component of $\C\sm [R_1\cup R_2\cup W]$ disjoint from $U$ (we will
call such components \emph{wedges}). We can define similar wedges
$T(f^i(R_1), f^i(R_2))=T_i$. Note that $T_i$'s are pairwise disjoint
because $W$ is wandering.
It follows that there exists $N$ such that for every $n>N$ the wedge
$T_n$ contains no critical points. Then $f(T_i)=T_{i+1}$ for all $i>N$.
Clearly, this contradicts the expansion on the circle at infinity.

\emph{We claim that $M$ equals the number of
components of $E(W)\sm W$.} Indeed, let $U$ be a Fatou domain of $E(W)$. Then
$\bd(U)$ is a connected set disjoint from $W$. Hence $\bd(U)$ is contained in
exactly one component of $E(W)\sm W$. Hence $\ol{U}$ is contained in this component,
and therefore the number of components of $\he(W)\sm W$ does not change if we add
all Fatou components of $E(W)$ to $\he(W)$ as desired.
\end{proof}

By Corollary~\ref{wandiscon}, for a wandering branch continuum $W\subset J_P$,
$\eval_{J_P}(W)$ is well-defined. We will use the following notation.

\begin{dfn}
A \emph{valence stable wandering collection $\B_\C$
\index{valence stable wandering collection}
of  continua} is
a finite collection of wandering continua $\{W_1,\dots,W_n\}$ with pairwise disjoint grand orbits
such that for each $j$ and $n\ge 0$, $P^n|_{\tai(p^n(W_j))}$ is one-to-one and $|A(W_j)|\ge 3$.
Denote by $\B^\infty_\C$ elements of $\B_\C$ which are wandering components of $J_P$ and by
$\B^p_\C$ elements of $\B_\C$ which are contained in a (pre)periodic component of $J_P$.
\end{dfn}

Note that if $W\in\B^p_\C$ then 
$|A(W)|=\val'_{\he(W)}(W)=\val_{\he(W)}(W)=\eval(W)$.





\section{The \fss for polynomials with disconnected Julia sets}\label{sect2}

In Section~\ref{sect2} we prove Theorem~\ref{intro-ineq}. Throughout the
section we deal with a valence stable wandering collection $\B_\C$ as
introduced above.

Let $W\in\B_\C$ be a wandering component of $J_P$. Then no iterate of $W$
intersects a periodic component of $K_P$. It has recently been shown
\cite{kozstr, qiuyin} that every wandering component of $J_P$ is a point.
However we will not rely on this in our paper. Let $\omega_P(W)=\limsup P^n(W)$
be the set of all limit points of $P^n(W)$.

\begin{lem}\label{Lemma 2}
If $W$ is a wandering component of $J_P$ then $\omega_P(W)$ cannot be contained
in a finite union of cycles of components of $K_P$.
\end{lem}

\begin{proof} Let $F$ be such a union.
Choose a neighborhood $U$ of $F$ bounded by a finite union of equipotentials of
the same (small) level, such that, if $P^n(x)\in U$ for all $n\ge 0$, then
$x\in F$. Since $W$ never maps in $F$, iterates of $W$ leave  $U$ infinitely
many times and $\omega_P(W)$ is not contained in $F$ as desired.
\end{proof}

In Theorem~\ref{Theorem 2} we associate to a wandering non-(pre)critical
component $W$ of $J_P$ specific sets of external arguments and critical
points.

\begin{thm}\label{Theorem 2}
If $W\in\B_\C$ is a  wandering component with $M=|A(W)|$, then
$B:=\ch(A(W))$ is a wandering non-(pre)critical $M$-gon under the
map $\sigma$ and there exist $M-1$ critical points $c_1, c_2, \dots,
c_{M-1}$ with disjoint orbits such that for every $j=1,\dots,M-1$
there exist a component $T_j$ of $J_P$, external arguments $t_j\ne
t_j'$, (possibly, one-sided) external rays $R_{t_j}, R_{t_j'}$ of
arguments $t_j, t_j'$ with $\sigma(t_j)=\sigma(t_j')$, and the
following claims hold.

\begin{enumerate}

\item[(a)] The leaf $\ell_j=t_jt_j'$ is a limit leaf of a sequence
of $\sigma$-iterates of $B$.

\item[(b)] Either $c_j\in R_{t_j}\cup R_{t_j'}$, or $c_j\in T_j$. Moreover,
    principal sets of $R_{t_j}$ and $R_{t_j'}$ are contained in $T_j$ and
    one of the following holds:

\begin{enumerate}

\item[(b1)] if $c_j\in T_j$, then $\{t_j, t_j'\}\subset A(T_j)$ and $A(T_j)$
is all-critical;

\item[(b2)] if $c_j\in R_{t_j}\cup R_{t_j'}$, then $R_{t_j}$ and $R_{t_j'}$
are one-sided rays having a common arc from $c_j$ to $T_j$, and $A(T_j)=\{t_j,
t_j'\}$. Also, $P(T_j)$ is a component of $J_P$ and $P(R_{t_j}) =
P(R_{t_j'}) = R_{\si(t_j)}$ is a unique (smooth) ray which accumulates in
$P(T_j)$.

\end{enumerate}

\item[(c)] $T_j$ is a wandering component of $J_P$.

\end{enumerate}

\end{thm}

\begin{proof}
Set $B':=A(W)$. Then $\sigma^n(B')\cap \sigma^m(B')=\emptyset$ if $m\ne n$.
Indeed, otherwise let $\al\in \sigma^n(B')\cap \sigma^m(B')$. Since $W$ is
non-(pre)critical, the ray $R_{\al}$ is smooth and $\pr(R_\al)\subset
P^n(W)\cap P^m(W)\ne \0$, a contradiction.

As in Lemma~\ref{inters}, $B$ is a wandering non-(pre)critical $M$-gon under
the map $\sigma$. Take the grand orbit $\Gamma(W)$ (see
Subsection~\ref{3.2.1}), and associate to each $W'\in\Gamma(W)$ the sets of
arguments $A(W')$ of  rays to $W'$ and the polygons $\ch(A(W'))$. Then the
set $\tai(W)$ is wandering. Indeed, since $W$ is wandering and by the previous
paragraph $\tai(W)$ is non-wandering only if two distinct forward images of
rays from $\tai(W)$ intersect. By Lemma~\ref{cap} then there are non-smooth
rays to some image of $W$, and $W$ is \emph{not} non-(pre)critical, a
contradiction.

The pullbacks of sets from the forward orbit of $\tai(W)$ form the \emph{grand
orbit $\G(\tai(W))$} of $\tai(W)$. We consider the set $\tai(W)$ instead of $W$
because in the disconnected case there are critical points outside $J_P$, hence
to catch \emph{all} criticality which shows along the orbit of $W$ we have to
consider $W$ \emph{together with} external rays to $W$.

Associate to all sets from $\G(\tai(W))$ the sets of the arguments of rays in
them. As in Subsection~\ref{planlam} this collection $\lam^B$ of polygons is a
geometric prelamination without critical leaves (see
Lemma~\ref{no-crit-leaf-2}).

Consider the family $L_{\lim}^B$ of limit leaves of polygons from $\lam^B$
(including degenerate leaves). By Theorem~\ref{doug} there exist at least $M-1$
recurrent critical leaves $\ell_1,...,\ell_{M-1}$ in $L_{\lim}^B$ with pairwise
disjoint infinite orbits and the same $\omega$-limit set $X$ (since
$\si$-images of $\ell_i$'s are points on the circle, $X\subset \uc$), such that
\emph{$X$ intersects every leaf in $L_{\lim}^B$.}

By Lemma~\ref{no-crit-leaf}(1), applied to the geometric prelamination
$\lam^B$, for each $j, 1\le j\le M-1$ the leaf $l_j=t_jt'_j$ is contained in an
all-critical gap-leaf $C_j$ of $\clam^B$. If $\si^{n_i}(B)$ approach $\ell_j$,
then the gaps $\si^{n_i+1}(B)$ separate the point $\si(\ell_j)$ from the rest
of the circle. Now we prove a few claims.

(i) \emph{The rays $R_{t_j}, R_{t_j'}$ have limit sets in the same
component $T_j$ and hence are included in $\tai(T_j)$.} This follows
from the connectedness 
of the set $\limsup (P^{n_i}(W))$,
where $\limsup$ is taken over a sequence $n_i$ of iterations of
$\si$, along which images of $B$ converge to $l_j$.

(ii) \emph{$P(T_j)$ is a component of $J_P$ and $P(R_{t_j}) = P(R_{t_j'}) =
R_{\si(t_j)}$ is a (unique) smooth ray which accumulates in $P(T_j)$.} Since
$\si^{n_i+1}(B)$ separate the point $\si(\ell_j)$ from the rest of the circle,
$R_{\si(\ell_j)}$ is the \emph{only} ray accumulating in $P(T_j)$.
Moreover, from the properties of the system of external rays described in
Section~\ref{disc1} it follows that $R_{\si(\ell_j)}$ is smooth.

(iii) \emph{Since $l_j$ has an infinite orbit, $t_j, t'_j$ are not (pre)periodic.}

(iv) \emph{$T_j$ is wandering.} Indeed, otherwise $P(T_j)$ is a (pre)periodic
component of $J_P$ such that $R_{\si(\ell_j)}$ is a unique ray accumulating in
it. This implies, that $\si(\ell_j)=\si(t_j)$ is (pre)periodic, a contradiction
with $\ell_j$ having infinite orbit by Theorem~\ref{doug}. This proves (c).

Choose points $u_j\in R_{t_j}$  ($u'_{j}\in R_{t'_j}$) so that the closed rays
$[u_j,\infty)$ ($[u'_j,\infty)$) from $u_j$ ($u'_j$, respectively) to infinity
contain no points from $C^*$. Let $A_j=R_{t_j}\sm (u_j,\infty)$,
$A'_j=R_{t'_j}\sm (u'_j,\infty)$ and $Z'=T_j\cup A'_j\cup A_j$. By (i) $P$ is
not one-to-one on $Z'$. By (iv) and Lemma~\ref{nonsep} $Z'$  is a
non-separating continuum with no interior in the plane. As before, by
\cite{hea96}, $Z'$ contains a critical point of $P$. Denote it by $c_j$. Then
there are two possibilities.

(b1) $c_j\notin R_{t_j}\cup R_{t_j'}$; then $c_j\in T_j$ as required.

(b2) $c_j\in R_{t_j}\cup R_{t_j'}$; by Lemma~\ref{cap}, $R_{t_j}$ and
$R_{t_j'}$ are one-sided rays sharing an arc from $c_j$ to $T_j$. We
show that $A(T_j)=\{t_j, t_j'\}$. Let the closed arcs of $R_{t_j}, R_{t_j'}$
from infinity to $c_j$ be $Q_j$ and $Q'_j$. Then $Q_j\cup Q'_j$ separates $\C$
into components $U$ and $V$ with $U\supset T_j$. An external ray in $V$ has the
principal set in $\ol{V}$, disjoint from $T_j$.
Also, the closure of an external ray in $U$ is separated from $T_j$ by
a forward image $P^{n_i}(W)$ of $W$ with its associated external rays
(see Figure~\ref{fig:72}). Hence $A(T_j)=\{t_j, t'_j\}$ proving (b).

\begin{figure}
\centerline{\includegraphics{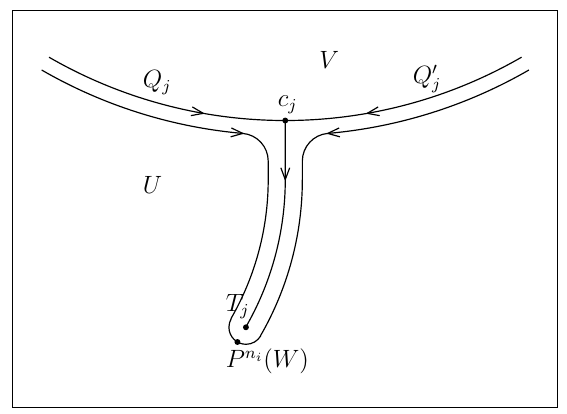}}
\caption{An illustration for the proof of Theorem~\ref{Theorem 2}.}
\label{fig:72}
\end{figure}

\end{proof}

\begin{dfn}[The set $C^w_\infty$]\label{cwinf}
\emph{A critical point $c$ of $P$ lies in $C^w_\infty$} \index{$C^w_\infty$} if and only if there
exists a wandering component $T_c$ of $J_P$ and two external arguments $t_c, t_c'$, such
that:

(i) $t_c, t_c'\in A(T_c)$, and $\sigma(A(T_c))$ is a point (thus,
$\si(t_c)=\si(t'_c)$);

(ii) either $t_c$ or $t_c'$ is recurrent under the map $\sigma$;
and

(iii) (a) $c$ belongs to the connected set $R_{t_c}\cup T_c\cup R_{t_c'}$, (b)
$P(c)\in R_{\si(t_c)}\cup P(T_c)$, (c) $R_{\si(t_c)}$ is a unique ray whose
closure is non-disjoint from $P(T_c)$ (moreover, $\pr(R_{\si(t_c)})\subset
P(T_c)$ and $R_{\si(t_c)}$ is a smooth ray).
\end{dfn}

By $C'_{wr}$ \index{$C'_{wr}$} we denote the number of weakly recurrent
critical points in wandering components of $J_P$. By Definition~\ref{cwinf},
$C^w_\infty\subset C'_{wr}$. Denote by $K(C^w_\infty)$ \index{$K(C^w_\infty)$}
the number of different grand orbits of $t_c$ ($t_c'$), and by  $L(C^w_\infty)$
\index{$L(C^w_\infty)$}  the number of different limit sets of $t_c$ under the
map $\sigma$, for $c\in C^w_\infty$. Theorem~\ref{doug}, Theorem~\ref{Theorem
2}, and the inequality $K(C^w_\infty)\le |C^w_\iy|\le |C'_{wr}|$ imply
Theorem~\ref{Theorem 3full} (if $\B^\infty_\C\ne \0$, then $C^w_\infty\ne \0$).

\begin{thm}\label{Theorem 3full}
Consider valence stable wandering collection of $m'\ge 0$ components $Q'_j$ of
$J_P$. If $m'>0$, then

$$\sum^{m'}_{j=1} (\val_{J_P}(Q'_j)-2) \le K(C^w_\infty)-L(C^w_\infty)
\le |C^w_\iy|-1\le |C'_{wr}|-1\le d-2.$$

\noindent So, with $\chi(m')$ defined as $1$ for $m'>0$ and $0$ otherwise, we have

\begin{equation}\label{discineq11}
\chi(m')+\sum^{m'}_{j=1} (\val_{J_P}(Q'_j)-2)\le \chi(m')|C'_{wr}|
\end{equation}

\end{thm}

Recall, that by $N_{irr}$ we denote the total number of repelling cycles $O$
such that the set $A(O)$ of arguments of external rays landing at points of $O$
contains no periodic angles; by Theorem~\ref{Theorem LP1}(2) such cycles $O$ are
exactly the cycles for which $A(O)$ is infinite. Also, by Theorem~\ref{Theorem
LP1} (2), each point of $O$ is a component of $J_P$.

\begin{dfn}[The set $C^{p}_\infty$]\label{cpinf}\index{$C^{p}_\infty$}
\emph{A critical point $c$ of $P$ belongs to $C^{p}_\infty$} if and only if
there exists a repelling cycle $O$ with infinite set $A(O)$, $c\in \iU$
has two external non-(pre)periodic arguments $t_c, t'_c\in A(O)$ with
the same $\si$-image (by \cite{lepr} we may assume that there are
no periodic external rays landing at $O$ and $t_c, t'_c$ are recurrent with the same infinite
minimal limit set).
\end{dfn}

Recall, that $C_{esc}$ is the set of all escaping critical points; then
$C^{p}_\infty\subset C_{esc}$. As always, denote by $K(C^p_\iy)$
\index{$K(C^p_\iy)$} the number of distinct grand orbits of points of
$C^p_\iy$. Theorem~\ref{Theorem LP1}(2) and the obvious inequality
$K(C^p_\iy)\le |C^p_\iy|\le C_{esc}$ imply Theorem~\ref{ninftyfull}.

\begin{thm}\label{ninftyfull}
The following inequality holds.

\begin{equation}\label{discineq1a}
N_{irr}\le K(C^p_\iy)\le |C^p_\iy|\le |C_{esc}|\le d-1
\end{equation}
\end{thm}

\begin{rem}
If $P_v(z)=z^2+v$ is a quadratic polynomial with disconnected Julia set, then it can have
at most one cycle $O_v$ with an infinite set $A(O_v)$. If this happens, the set
$A(O_v)$ contains two external arguments $t_0, t_0'$ as above of the critical
point $0\in U^\infty$ sharing the same image $t_*=\sigma(t_0)=\sigma(t_0')$.
Let $m$ be the period of $O_v$. Assume that the base-$2^m$ representation
of $t_*\in (0,1)$ contains only two digits.
(For example, it obviously holds, if $m=1$.) Then, by~\cite{lev94}, as $v$ approaches the Mandelbrot set $M$
along the external ray of $M$ of argument $t_*$, the multiplier of $O_v$ tends to some point
$e^{2\pi \nu}$ of the unit circle, where $\nu$ is irrational. Hence, the ray ends at a point $v_*$ of the boundary of a hyperbolic component of $M$
and $O_v$ tends to either a Cremer or a Siegel cycle of $P_{v_*}$.
Converse statement is also true and follows essentially from
Yoccoz's result about the local connectivity of the Mandelbrot set
at the boundaries of hyperbolic components, see e.g.~\cite{sch04}.
Note that by Theorem~\ref{mh1} the critical point of $P_{v_*}$ is
recurrent and weakly non-separated from a point of the (Cremer or
Siegel) cycle.
\end{rem}

\begin{cor}\label{Corollary 1}
Consider a valence stable wandering collection of $m'\ge 0$ components $Q'_j$
of $J_P$. Then we have
\begin{equation}\label{discineq2}
N_{irr}+\chi(m')+\sum^{m'}_{j=1} (\val_{J_P}(Q'_j)-2)\le \chi(m')|C'_{wr}| +|C_{esc}|
\end{equation}
\end{cor}

\begin{proof}
Let us show that $C^w_\infty\cap C^p_\infty=\0$. Indeed, let (i) $c\in
C^w_\infty$ and (ii) $c\in C^p_\infty$. Then, because of (i), by
Definition~\ref{cwinf} there is a smooth ray $R$ accumulating in a component
$T$ of $J_P$ such that $P(c)\in R\cup T$ and $R\cup T$ is disjoint from
closures of all rays other than $R$. However, because of (ii), by
Definition~\ref{cpinf}, there must also exist a ray $R'$ such that $c\in R'$
and $\pr(R')$ is a periodic point at which infinitely many other rays land.
Thus, we get a contradiction which shows that $C^w_\infty\cap C^p_\infty=\0$.
Now we can add inequalities (\ref{discineq11}) and (\ref{discineq1a}) which
implies the desired inequality.
\end{proof}

Corollary~\ref{Corollary 1} proves the second inequality of
Theorem~\ref{intro-ineq}. We prove the first one in Lemma~\ref{intro-ineq-2}.
Recall that $N_{co}$ is the number of cycles of components of $J_P$ containing
non-(pre)critical branch continua.

\begin{lem}\label{intro-ineq-2}
Suppose that $\B^p_\C=\{Q_i\}$ is a valence stable wandering
collection of continua which consists of $m$ elements contained in
periodic components of $J_P$. Then
\begin{equation}\label{perwand1}
N_{FC}+ N_{co} +\sum_{i=1}^m (\val_{J_P}(Q_i)-2) \le |C_{wr}|.
\end{equation}
\end{lem}

\begin{proof}
\emph{Let us show, that we may deal with a valence stable wandering collection
of continua which maximizes $\sum_{i=1}^m (\val_{J_P}(Q_i)-2)$.} Indeed, for
such a collection all cycles of components of $J_P$ which contain some
wandering non-(pre)critical branch continua must be used in the sense that
wandering continua contained in the cycle  should be part of the collection
(otherwise they can be added to the collection increasing the sum in question).
Hence if the collection maximizes $\sum_{i=1}^m (\val_{J_P}(Q_i)-2)$, then it
maximizes $N_{co} +\sum_{i=1}^m (\val_{J_P}(Q_i)-2)$, and it suffices to prove
the inequality for such a maximal collection.

Take a non-degenerate periodic component $E$ of $K_P$ of period $p$.
Suppose that it contains $n_E\ge 0$ elements of a chosen maximal
valence stable collection of wandering continua ($n_E=0$ would mean
that it contains no such elements). By Corollary~\ref{wandiscon} and
the remarks after that, we may assume that only $E$ contains the
continua $Q_i$. By \cite{dh85b}, $P^p|_{E}$ is a polynomial-like
map. In particular, there exists a sufficiently tight neighborhood
$U$ of $E$ such that $P^p|_U$ is conjugate to $f|_V$ for a
polynomial $f$ with connected Julia set $J_f$, filled-in Julia set
$K_f$, and a tight neighborhood $V$ of $K_f$.

Any such conjugacy transports wandering continua, Fatou domains and CS-points
of $P^p$ to wandering continua, Fatou domains and CS-points of $f$ because
these objects are defined topologically. The same holds for critical points of
$P^p|_U$ and the valence of subcontinua of $J_P$. Moreover, weakly recurrent
critical points are also transported by any conjugacy because so are periodic
cutpoints and their preimages, and by definition only the cuts in the Julia set
made by periodic cutpoints and their preimages are necessary to define weakly
recurrent points.

Therefore Theorem~\ref{maintech} implies the inequality
\begin{equation}\label{perwand2}
\chi(n_E)+\sum_{Q_i\subset E} (\val_{J_P}(Q_i)-2)+N_{FC}(P^p|_E) \le
K(C_{wr}(P^p|_E))
\end{equation}
in which by $N_{FC}(P^p|_E)$ we denote the number of cycles of Fatou domains and
Cremer cycles of $P^p|_E$ and by $C_{wr}(P^p|_E)$ we denote all the weakly
recurrent critical points of $P^p|_E$ (recall also, that then
$K(C_{wr}(P^p|_E))$ denotes the number of grand orbits of critical points from
$C_{wr}(P^p|_E)$ under the map $P^p$). It is obvious that $N_{FC}(P^p|_E)$ coincides
with the number $N_{FC}(\orb_P(E))$ of cycles of Fatou domains and Cremer cycles in
the entire (periodic) orbit $\orb_P(E)$ of $E$. Also, it is easy to see that
all critical points of $P^p|_E$ are in fact preimages of critical points of $P$
belonging to $\orb_P(E)$, and weakly recurrent critical points of $P^p|_E$ are
in fact preimages of weakly recurrent critical points of $P$ belonging to $\orb_P(E)$.
Therefore, $K(C_{wr}(P^p|_E))$ coincides with the number $K(C_{wr}\cap
\orb_P(E))$ of grand orbits of weakly recurrent critical points of $P$
belonging to $\orb_P(E)$.

Let us sum up inequality (\ref{perwand2}) over all cycles of components of $J_P$.
The left hand side of the summed up inequality coincides literally with the
left hand side of inequality (\ref{perwand1}). The right hand side will be equal
to the number of grand orbit of weakly recurrent critical points belonging to
periodic components of $P$, and the latter number is obviously less than or
equal to $|C_{wr}|$. This completes the proof of the lemma.
\end{proof}

It remains to make the following observations. The first inequality of
Theorem~\ref{intro-ineq} is inequality (\ref{discineq2}) proven in
Corollary~\ref{Corollary 1}. The second inequality of Theorem~\ref{intro-ineq}
is inequality (\ref{perwand1}) proven in Lemma~\ref{intro-ineq-2}. The sum of
these two inequalities leads to the main inequality of Theorem~\ref{intro-ineq}
(notice that since sets of critical points $C_{wr}, C'_{wr}$ and $C_{esc}$ are
obviously pairwise disjoint we have that $|C_{wr}|+|C'_{wr}|+|C_{esc}|\le d-1$).

\bibliographystyle{amsalpha}

\printindex

\end{document}